%% file: KKS_PartialActions.tex
\documentclass[letterpaper, 11pt]{amsart}
\usepackage[T1]{fontenc}
\usepackage[utf8]{inputenc}

\usepackage{mathpazo}

\usepackage[margin = 1.35in]{geometry}
\usepackage[dvipsnames]{xcolor}

\usepackage{enumitem}
\setenumerate[0]{label = (\roman*)}

\usepackage{hyperref}
\newcommand\myshade{90}
\hypersetup{
  linkcolor  = NavyBlue!\myshade!black,
  citecolor  = WildStrawberry!\myshade!black,
  urlcolor   = Magenta!\myshade!black,
  colorlinks = true,
  breaklinks = true,
  pdftitle   = {The ideal intersection property for partial C*-dynamical systems},
  pdfauthor  = {Matthew Kennedy, Larissa Kroell, Camila Sehnem},
}

\usepackage{tikz}
\usetikzlibrary{cd}
\usetikzlibrary{babel}
\usetikzlibrary{shapes,backgrounds}
\usetikzlibrary{arrows}

\usepackage[backend= biber, giveninits=true,uniquename=init,maxbibnames=10,sorting=nyt,natbib=true, style=alphabetic]{biblatex}
\bibliography{references_partial}

\usepackage{amsmath,amsthm,amssymb}
\usepackage{faktor}
\usepackage[nameinlink,capitalize,sort]{cleveref}

\input{commands}

\title{Injective envelopes of partial \texorpdfstring{$\Cast$}{C*}-dynamical systems}

\author{Matthew Kennedy}
\address{Department of Pure Mathematics, University of Waterloo, 200 University Avenue West, Waterloo, Ontario, N2L 3G1, Canada} \email{matt.kennedy@uwaterloo.ca}

\author{Larissa Kroell}
\address{Department of Pure Mathematics, University of Waterloo, 200 University Avenue West, Waterloo, Ontario, N2L 3G1, Canada} \email{lgkroell@gmail.com}

\author{Camila F. Sehnem}
\address{RIMS, Kyoto University, Kyoto 606-8502, Japan}
\email{csehnem@kurims.kyoto-u.ac.jp}

\begin{document}

\begin{abstract}

We extend Hamana's theory of injective envelopes, along with several key features of the theory, to the realm of partial $\Cast$-dynamical systems. In particular, we show that a partial $\Cast$-dynamical system has the ideal intersection property if and only if its injective envelope does. A key ingredient in our arguments is a new kind of unitization of a partial action on a unital $\Cast$-algebra~$A$ arising from the $\Cast$-algebra generated by the orbits of~$A$ in its injective envelope~$I(A)$. For an arbitrary unital partial $\Cast$-dynamical system, which is known to have an enveloping action, we establish a natural relationship between the injective envelope of the system and the injective envelope of its enveloping action. For an abelian partial $\Cast$-dynamical system, we show that our construction coincides with the algebra of continuous functions on the Furstenberg boundary of the corresponding transformation groupoid. It is crucial in our work to consider a notion of generalized unital partial $\Cast$-dynamical systems, in which the unital ideals are replaced by unital hereditary subalgebras.

\end{abstract}

\maketitle

\section{Introduction}

Partial group actions are a generalization of group actions, and appear naturally in dynamics and operator algebras. Perhaps the simplest class of examples of a partial action of a discrete group~$G$ on a $\Cast$-algebra arises from the restriction of an action of~$G$ on a compact space to not necessarily invariant open subsets, or more generally, from the restriction of an action of $G$ on a $\Cast$-algebra to ideals. However, there are many partial actions that do not arise in this way.

A partial action of~$G$ on a $\Cast$-algebra~$A$ consists of a family of ideals $\{A_g\}_{g\in G}$ in $A$ and ${}^*$-isomorphisms $\{\alpha_g\colon A_{g^\inv}\to A_g\}_{g\in G}$ such that the pair $\alpha=(\{A_g\}_{g\in G}, \{\alpha_g\}_{g\in G})$ satisfies certain axioms, as if it was the restriction of an action of~$G$ to the ideals $\{A_g\}_{g \in G}$. The triple $(A, G,\alpha)$ is then called a partial $\Cast$-dynamical system. 

Partial actions and their crossed products were introduced by McClanahan in~\cite{McClan95}, following seminal work of Exel on partial actions of the integers and $\Cast$-algebras carrying a circle action~\cite{Exe94}. Several classes of $\Cast$-algebras have been described as reduced crossed products associated to a partial action on a $\Cast$-algebra, see for example \cite{Exe94,ExLa99,ExStar16} and \cite[Theorem 5.6.41]{CELY17}. In addition, under mild assumptions a $\Cast$-algebra carrying a topological grading over~$G$ arises from a partial action of~$G$ on the fixed-point algebra up to stabilization \cite{Seh14} (see also ~\cite[Theorem~27.11]{Exel17}). More recently, partial actions and their twisted version have occurred naturally in the analysis of the structure of reduced crossed products associated to (global) group actions on~$\Cast$-algebras, see \cite{KenScha19,Brown2019,KenUrs24}.

In a series of papers starting in the late 70s, Hamana introduced and studied injective envelopes of objects in various categories of operator systems and $\Cast$-algebras. Notably, Hamana showed in \cite{Ham79} that every operator system has an injective envelope in the category of operator systems with unital completely positive (ucp) maps as morphisms, and used this to establish Arveson's conjecture on the existence of the $\Cast$-envelope of an operator system \cite{Arv69}. The existence of injective envelopes of $\Cast$-dynamical systems was established in \cite{Ham85}, extending much of the theory of injective envelopes of $\Cast$-algebras \cite{Ham79_CStar}. A few decades later, in an important application of Hamana's theory, Kalantar and Kennedy characterized simplicity of the reduced group $\Cast$-algebra of~$G$ in terms of the injective envelope of the trivial $\Cast$-dynamical system $(\mathbb{C}, G)$ \cite{KalKen17}. Since then injective envelopes of $\Cast$-algebras in general and of $\Cast$-dynamical systems in particular have played a crucial role in the analysis of the ideal structure of reduced crossed products associated to group actions, see \cite{KenScha19, Kaw17, Breu+17, Bryd22, GefUrs23, KKS25}. 

In this paper, we extend Hamana's theory of injective envelopes to the realm of partial $\Cast$-dynamical systems. We pay special attention to the relationship between the ideal structure of the reduced crossed product associated to a partial $\Cast$-dynamical system and that of the reduced crossed product associated to its injective envelope.  

As in \cite{Ham85, Bryd22, Ken+21}, the first step towards extending the notion of injective envelopes of $\Cast$-dynamical systems to the framework of partial actions is to find the correct category, which should have sufficiently many injective objects in a suitable sense. Our approach to this problem requires that we initially restrict ourselves to the setting of \emph{unital partial actions}, meaning that the corresponding family of ideals is determined by central projections in the ambient $\Cast$-algebra. We were also naturally led to a generalization of the notion of a (unital) partial action of~$G$ on a $\Cast$-algebra~$A$, in which the family of unital ideals is replaced by hereditary subalgebras that are determined by a commuting family of projections in $A$. One of our main examples of this arises from a partial ${}^*$-representation of $G$ on a Hilbert space $\hilb$, which canonically induces a generalized unital partial action of~$G$ on~$B(\hilb)$. Another important example in our work comes from the reduced crossed product associated to a unital partial action. More generally, a unital partial action of~$G$ on a C*-algebra $A$ restricts to a generalized unital partial action on a corner, provided that the corresponding family of projections is mutually commuting.

We thus consider in this paper the category whose objects are generalized unital partial $\Cast$-dynamical systems. An arrow from $(A, G,\alpha)$ to $(B,G, \beta)$ is what we call a \emph{$G$-morphism}, meaning a $G$-equivariant ucp map $\phi\colon A\to B $ that is $G$-unital, see~\cref{def: g-morphism}. When $\ell^\infty(G,A)$ is endowed with the left translation action, there is a $G$-embedding of $A$ into a corner of $\ell^\infty(G,A)$. This corner is injective in the category of generalized unital partial $\Cast$-dynamical systems whenever $A$ is injective. This allows us to establish the first main result of this paper.

\begin{introthm}[Theorem \ref{thm: existence inj env}]\label{introthm: existence inj env}
Let $(A, G, \alpha)$ be a unital partial $\Cast$-dynamical system. Then $(A, G, \alpha)$ has an injective envelope $((I_G(A), G,I_G(\alpha)),\kappa)$. Moreover, $I_G(\alpha)=\partacg{I_G(A)}{I_G(\alpha)}$ is a unital partial action, and $(I_G(A), G,I_G(\alpha))$ is unique in the sense that if $((B, G,\beta), \kappa_B)$ is an injective envelope of $(A, G, \alpha)$, then there is a unique $G$-isomorphism $\psi\colon I_G(A)\to B$ such that $\psi\circ\kappa=\kappa_B$.
\end{introthm}

If $\alpha$ is an action, then $I_G(A)$ coincides with the $G$-injective envelope of~$A$ in the sense of Hamana. We also note that the $G$-embedding of~$A$ into the corner of $\ell^\infty(G,A)$ that we use in this paper has appeared before in the context of enveloping actions, see \cite[Theorem 4.5]{DokEx05} and \cite[Proposition~2.7]{Abetal22}. In fact, unital partial actions are precisely the partial actions on unital $\Cast$-algebras that admit enveloping actions in the sense of Abadie \cite{Aba03}. We show that the injective envelope of a unital partial $\Cast$-dynamical system has a natural relationship with the injective envelope associated to its enveloping action.

\begin{introthm}[Theorem~\ref{thm: corner of enveloping action}]\label{introthm: corner of enveloping action}
    Let $(A, G, \alpha)$ be a unital partial $\Cast$-dynamical system and $(\cC^A, \zeta^A)$ the enveloping action of~$\alpha$. Let $(I_G(A),G,I_G(\alpha))$ be the injective envelope of $(A, G, \alpha)$ and $(I_G(\cC^A), G, I_G(\zeta^A))$ the injective envelope of $(\cC^A, G, \zeta^A)$. Then $I_G(A) = I_G(\cC^A)1_A$, and $I_G(\alpha)$ is the restriction of $(I_G(\cC^A), I_G(\zeta^A))$.
\end{introthm}

In order to construct the injective envelope of a partial $\Cast$-dynamical system $(A,G,\alpha)$ that is not necessarily unital, we utilize the notion of a unitization of $\alpha$. This is a unital partial $\Cast$-dynamical system $(B,G,\beta)$ with a $G$-equivariant inclusion of $\Cast$-algebras $A\subseteq B$. In this paper we consider a special unitization of~$\alpha$ arising from the extension $I(\alpha)$ of $\alpha$ to the injective envelope~$I(A)$ of~$A$ (see Proposition~\ref{prop: partac on I(A)}). The injective envelope of $(A, G,\alpha)$ is defined to be the injective envelope of $(I(A), G, I(\alpha))$, so that $I_G(A):=I_G(I(A))$ and $I_G(\alpha):=I_G(I(\alpha))$.

Another class of unitizations of $\alpha$ that plays a crucial role in our arguments arises as follows: if $(B, G, \alpha)$ is a unital partial $\Cast$-dynamical system with a $G$-equivariant inclusion of unital $\Cast$-algebras $A\subseteq B$, then the $\Cast$-algebra $\operatorname{Orb}_G^\beta(A)$ generated by the orbits of~$A$ in~$B$ is $\beta$-invariant, and hence~$\beta$ restricts to a unital partial action on~$\operatorname{Orb}_G^\beta(A)$ (see Lemma~\ref{lem:partial-action-on-orbits}). We show that, in a suitable sense, the smallest unitization of~$\alpha$ obtained in this way is precisely the restriction of $I(\alpha)$ to $\operatorname{Orb}_G^{I(\alpha)}(A)$ (Lemma~\ref{lem:co-universal-unitization}). One of the main technical ingredients is Proposition~\ref{pro: IIP-equiv. for unitization}, which implies that every ideal~$J$ of the reduced crossed product $A\crossed G$ that has zero intersection with~$A$ generates an ideal in $\operatorname{Orb}_G^{I(\alpha)}(A)\crossed G$ that has zero intersection with $\operatorname{Orb}_G^{I(\alpha)}(A)$. This enables us to extend a result of Bryder \cite[Theorem~3.2]{Bryd22} to our setting. We recall that a partial $\Cast$-dynamical system $(A,G,\alpha)$ is said to have the ideal intersection property if every nonzero ideal of the reduced crossed product $A \crossed G$ has nonzero intersection with $A$.

\begin{introthm}[Theorem~\ref{thm: Equivalent IIP A to I(A)}]\label{introthm: Equivalent IIP A to I(A)}
    Let $(A, G, \alpha)$ be a partial $\Cast$-dynamical system. Then the following are equivalent:
    \begin{enumerate}
        \item[\rm{(1)}] $(A, G, \alpha)$ has the ideal intersection property;
        \item[\rm{(2)}] every partial $\Cast$-dynamical system $(B, G, \beta)$ with $$(A, G, \alpha)\subseteq  (B, G, \beta) \subseteq (I_G(A), G, I_G(\alpha))$$ has the ideal intersection property;
        \item[\rm{(3)}] the partial $\Cast$-dynamical system $(I_G(A), G, I_G(\alpha))$ has the ideal intersection property.
    \end{enumerate}
\end{introthm}

The proof of Theorem \ref{introthm: Equivalent IIP A to I(A)} utilizes Proposition~\ref{pro: IIP-equiv. for unitization} and the corresponding result in the setting of unital partial actions (see Theorem~\ref{thm: IIP-equiv-unital-pa}). In contrast to unital partial $\Cast$-dynamical systems, in particular those associated to global actions, there is no clear reason why a ucp map from $I_G(A)\crossed G$ extending a ${}^*$-homomorphism from $A\crossed G$ would send the canonical partial isometries in $I_G(A)\crossed G$ to partial isometries or, equivalently, that the family of central projections in $I_G(A)$ should be sent to projections, thereby inducing a $G$-morphism from $I_G(A)$. This is the main technical difficulty we encounter throughout the second part of this paper when we treat arbitrary partial $\Cast$-dynamical systems. 

In addition to Theorem \ref{introthm: Equivalent IIP A to I(A)} we establish important categorical properties of the inclusion $A\subseteq I_G(A)$, along the lines of Hamana's original work. Specifically, we prove that the inclusion $A\subseteq I_G(A)$ is rigid with respect to $G$-equivariant ucp maps (see Proposition~\ref{pro: rigidity-non-unital}). Moreover, we show that $A\subseteq I_G(A)$ is essential with respect to (not necessarily $G$-unital) $G$-equivariant ucp maps that are multiplicative on $A$, in the sense that every $G$-equivariant ucp map from~$I_G(A)$ that restricts to a faithful ${}^*$-homomorphism from~$A$ is necessarily completely isometric (see Theorem~\ref{thm: essentiality-equivariant-maps}). 

In order to show that the inclusion $A\crossed G\subseteq I_G(A)\crossed G$ is essential for an arbitrary partial $\Cast$-dynamical system $(A, G, \alpha)$, we observe first that there exists a canonical generalized unital partial action on the injective envelope $I(A\crossed G)$ that extends~$\alpha$. Moreover, we prove that the canonical conditional expectation $E_A\colon A\crossed G\to A$ induces a faithful $G$-morphism $\Phi\colon I(A\crossed G)\to I_G(A)$ (Proposition~\ref{prop: gen part ac on crossed product-non-unital}). Using this we are able to extend Hamana's theorem \cite[Theorem 3.4]{Ham85} to the setting of partial $\Cast$-dynamical systems. The following is the precise statement.

\begin{introthm}[Theorem~\ref{thm: G-essential crossed prod in I(A cross G)-non-unital}]\label{introthm: G-essential crossed prod in I(A cross G)-non-unital}
    Let $(A,G,\alpha)$ and $(B, G,\beta)$ be partial $\Cast$-dy\-na\-mi\-cal  systems with a $G$-equivariant inclusion of $\Cast$-algebras $A\subseteq B$. Then
    \[
    (A,G,\alpha) \subseteq (B, G,\beta) \subseteq (I_G(A), G, I_G(\alpha))
    \]
    if and only if there is a $G$-equivariant inclusion of $\Cast$-algebras
    \[
    A\crossed G \subseteq B\crossed G \subseteq I(A\crossed  G).
    \]
    In particular, $I(A\crossed G)=I(I(A)\crossed G)=I(I_G(A)\crossed G)$.
\end{introthm}

This is an essential tool that allows us to transfer back and forth many properties of the ideal structure from $A\crossed G$ to $I_G(A)\crossed G$. The following is an immediate consequence of 
Theorem~\ref{introthm: G-essential crossed prod in I(A cross G)-non-unital}.

\begin{introcor}[Corollary~\ref{cor: primality}] 
    Let $\partacg{A}{\alpha}$ be a partial $\Cast$-dynamical system. Then the following are equivalent:
    \begin{enumerate}
        \item $A\crossed G$ is prime;
         \item  $I(A)\crossed G$ is prime;
         \item  $I_G(A)\crossed G$ is prime;
         \item $Z(I( A\crossed G))=\mathbb{C}.$
    \end{enumerate}
\end{introcor}

In the final part of this paper we consider a unital commutative $\Cast$-algebra $\mathrm{C}(X)$ and a partial $\Cast$-dynamical system $(\mathrm{C}(X),G,\alpha)$. A very natural question to ask is how the injective envelope of $(\mathrm{C}(X), G, \alpha)$ relates to the Furstenberg boundary of the transformation groupoid $\mathcal{G}_\alpha\ltimes X$ associated to the underlying topological partial action on~$X$ \cite{Bor20,Ken+21}. Although the category introduced in this paper is different from those of~\cite{Bor20,Ken+21}, we prove that the two notions of injective envelopes for $(\mathrm{C}(X), G, \alpha)$ coincide. Precisely we establish the following.

\begin{introthm}[Theorem~\ref{thm: abelian-partial-actions}]\label{introthm: abelian-partial-actions} Let $\mathcal{C}:=\mathrm{C}(\partial_H(\mathcal{G}_\alpha\ltimes X))$ be the injective envelope of the $(\mathcal{G}_\alpha\ltimes X)$-$\Cast$-algebra $\mathrm{C}(X)$ in the category of unital $(\mathcal{G}_\alpha\ltimes X)$-$\Cast$-algebras. Let $\gamma$ be the partial action on~$\mathcal{C}$ from Lemma~\ref{lem: partial-action-groupoid} and let $I(\gamma)$ be the unital partial action on~$\mathcal{C}$ extending~$\gamma$. Then there exists a $G$-isomorphism $\Psi\colon \mathcal{C}\to I_G(\mathrm{C}(X)) $ such that $\Psi\circ\iota=\id_{\mathrm{C}(X)}$. 
\end{introthm}

As a consequence of Theorem~\ref{introthm: abelian-partial-actions}, the characterization of the ideal intersection property from \cite[Theorem~7.2]{Ken+21} applies to the transformation groupoid associated to an arbitrary partial action on a compact space. This follows because by \cite[Theorem~4.2.6]{Kroell25} the analogous statement to \cite[Proposition~4.6]{KenScha19} on amenability of quasi-stabilizer subgroups holds for arbitrary partial $\Cast$-dynamical systems. However, we believe that Theorem~\ref{thm: abelian-partial-actions} and the approach taken in this work provides an alternative perspective on Furstenberg boundaries of groupoids.

This paper is organized as follows. In \cref{sec: prelim} we review the necessary background on injective envelopes of $\Cast$-algebras and partial actions, also taking the opportunity to establish our notation for partial actions and reduced crossed products. In \cref{sec: inj envelopes partial} we introduce the category of generalized unital partial actions, define injective envelopes of unital partial $\Cast$-dynamical systems and establish Theorem~\ref{introthm: existence inj env}. We also prove Theorem~\ref{introthm: Equivalent IIP A to I(A)} in the particular case of unital partial actions (see Theorem~\ref{thm: IIP-equiv-unital-pa}). In \cref{sec: enveloping action} we discuss enveloping actions and prove Theorem~\ref{introthm: corner of enveloping action}. We begin \cref{sec: non-unital} by defining injective envelopes of arbitrary partial $\Cast$-dynamical systems (see \cref{def: G-envelope-non-unital}) and establishing the main technical results needed in our proof of Theorem~\ref{introthm: Equivalent IIP A to I(A)}. After proving Theorem~\ref{introthm: Equivalent IIP A to I(A)} we consider $G$-essentiality and $G$-rigidity for the inclusion $A\subseteq I_G(A)$, and also establish Theorem~\ref{introthm: G-essential crossed prod in I(A cross G)-non-unital}. At the end of \cref{sec: non-unital} we briefly discuss pseudo-expectations in the context of arbitrary partial $\Cast$-dynamical system, and establish a characterization of the ideal intersection property in terms of pseudo-expectations along the lines of~\cite[Theorem~6.6]{KenScha19} (see Corollary~\ref{cor: pseudo-IIP}). Finally, in \cref{sec: abelian-groupoid} we restrict our attention to abelian partial $\Cast$-dynamical systems and prove Theorem~\ref{introthm: abelian-partial-actions}.

\subsection*{Acknowledgements} This paper is based on results from the PhD thesis of the second author defended at the University of Waterloo.  The first author is grateful to Ruy Exel for suggesting the setting of partial actions. The second author was supported in this research by an NSERC Canada Graduate Scholarship-Doctoral.

\section{Preliminaries}\label{sec: prelim}

\subsection{Injective envelopes of \texorpdfstring{$\Cast$}{C*}-algebras}\label{subsec: injective-alg}
Hamana introduced and extensively studied injective envelopes in various categories. See, for example, \cite{Ham79_CStar,Ham79, Ham85} and references therein. We recall key aspects of Hamana's work that will be relevant to us. We will also require the theory of operator systems and completely positive maps between them, for which we refer the reader to the book of Paulsen~\cite{Paul02}.

A unital $\Cast$-algebra $B$ is \emph{injective} if for any operator systems $S, T$, any unital completely isometric (uci) map $\kappa \colon S \to T$ and any unital completely positive (ucp) map $\varphi \colon S \to A$, there exists a ucp~map $\hat \varphi \colon T \to A$ satisfying $\hat \varphi \circ \kappa = \varphi$, i.e. $B$ is injective in the category of operator systems with ucp~maps as morphisms. An \emph{extension} of a $\Cast$-algebra $A$ is a pair $(B,\kappa)$, where $B$ is a unital $\Cast$-algebra and $\kappa \colon A \to B$ is an injective ${}^*$-homomorphism. The extension $(B,\kappa)$ of $A$ is \emph{injective} if $B$ is injective, and it is \emph{rigid} if for any ucp map $\varphi \colon B \to B$ with $\varphi\vert_{\kappa(A)} = \id\vert_{\kappa(A)}$, we have $\varphi = \id_B$. Lastly, an extension $(B,\kappa)$ of $A$ is called \emph{essential} if given a ucp map $\varphi\colon B \to S$ into an operator system $S$ such that $\varphi \circ \kappa$ is completely isometric, then $\varphi$ is completely isometric. 

An extension $(B,\kappa)$ of $A$ is an \emph{injective envelope} of $A$ if it is injective and essential. Equivalently, $(B,\kappa)$ is an injective and rigid extension of~$A$. By \cite{Ham79,Ham79_CStar} every $\Cast$-algebra $A$ admits an injective envelope $(I(A),\kappa)$, which is unique in the following sense: if $(B,\iota)$ is another injective essential extension of $A$, then there exists a unique ${}^*$-isomorphism $\rho \colon I(A) \to B$ such that $\rho \circ \kappa = \iota$. For convenience we usually suppress the embedding $\kappa$ and simply identify $A$ with the range of~$\kappa$ in~$I(A)$.

The inclusion $A\subseteq I(A)$ of $A$ in its injective envelope has many useful properties. For example, $I(A)$ is prime if and only $A$ is prime, and if $A$ is unital, then $I(A)$ is simple if $A$ is simple. In general, the inclusion $Z(A) \subseteq Z(I(A))$ holds and every nonzero ideal $J \ideal I(A)$ satisfies $J \cap A\neq  \trivial$. In addition, being an injective $\Cast$-algebra, $I(A)$ is \emph{monotone complete}, i.e. any bounded increasing net $(x_\lambda)_\lambda$ of self-adjoint elements in $I(A)$ admits a supremum $\sup_\lambda x_\lambda$ in $I(A)$. Lastly, injective envelopes behave well with respect to hereditary subalgebras, meaning that the inclusion $B\subseteq I(A)$ induces an inclusion of $\Cast$-algebras $I(B)\subseteq I(A)$ if $B$ is a hereditary subalgebra of $A$. More precisely, if $B \subseteq A$ is a hereditary subalgebra, then there exists a projection $p \in I(A)$ such that the inclusion $B\subseteq I(A)$ extends to a ${}^*$-isomorphism $I(B) \cong pI(A)p$ \cite[Proposition 6.5]{Ham82_tensor}. Specifically, the projection $p$ is the supremum of any increasing approximate identity for~$B$. If, in addition, $B$ is an ideal in $A$, then the projection~$p$ is central and hence $I(B)$ is a unital ideal of $I(A)$. Moreover, in this case $B^\perp=A\cap (1-p)I(A)$ and $I(B^\perp)=(1-p)I(A)$, see Lemma~1.1 and Theorem~1.5 of~\cite{Ham82_centre}.

\subsection{Partial Actions}

In this subsection we briefly discuss partial group actions on $\Cast$-algebras, and verify that a partial action of a group on a $\Cast$-algebra admits a natural extension to a partial action on its injective envelope. We refer to the book of Exel \cite{Exel17} for an extensive account on partial actions and their associated crossed products. 

Throughout this paper~$G$ is a discrete group with unit element~$e$.

\begin{defi}\label{def: partial action}
    A \emph{partial action} of $G$ on a $\Cast$-algebra $A$ is given by a pair $\alpha=\partacg{A}{\alpha}$, where $\{A_g\}_{g\in G}$ is a family of (closed two-sided) ideals in~$A$ and $\alpha_g \colon A_{g^\inv} \to A_g$ is a ${}^*$-isomorphism for all $g\in G$, such that 
   \begin{enumerate}
       \item $A_e = A$, and $\alpha_e=\id$,
       \item $\alpha_g(A_{g^\inv} \cap A_h) = A_{gh} \cap A_g$ for all $g,h\in G$,
       \item $\alpha_g(\alpha_h(x)) = \alpha_{gh}(x)$ for all $x\in A_{h^\inv}\cap A_{(gh)^\inv}$ for all $g,h\in G$.
   \end{enumerate}
  We call $(A,G, \{A_g\}_{g\in G}, \{\alpha_g\}_{g\in G})$, or simply $(A, G, \alpha)$, a \emph{partial $\Cast$-dynamical system}. We will say that a partial action $\alpha=\partacg{A}{\alpha}$ is \emph{unital} if each ideal~$A_g$ is unital, i.e. for each $g\in G$ there is a central projection $p_g\in A$ such that $A_g= A p_g$. When $A_g=A$ for all $g\in G$, so that $\alpha=\partacg{A}{\alpha}$ is an ordinary group action, we refer to $\alpha=\partacg{A}{\alpha}$ as a \emph{global} group action, and to  $(A, G, \alpha)$ as a global $\Cast$-dynamical system.
\end{defi}

We record the following key example classes of partial actions that will be particularly important to us.

\begin{examp}[Restrictions of global actions]
    A natural class of examples of partial actions comes from global actions via restrictions to not-necessarily invariant ideals. Let $\alpha$ be an action of~$G$ on a $\Cast$-algebra $A$ and $I \trianglelefteq A$ an ideal in~$A$. For each $g\in G$ set $I_g := I \cap \alpha_g(I)$ and $\beta_g := \alpha\vert_{I_{g^\inv}}$. Then $\beta=\partacg{I}{\beta}$ is a partial action of~$G$ on~$I$, to which we refer as the \emph{restriction} of $\alpha$ to $I$.
\end{examp} 

\begin{examp}[Partial ${}^*$-automorphisms] A \emph{partial ${}^*$-automorphism} of $A$ is a ${}^*$-isomorphism $\alpha\colon I \to J$ between two ideals $I,J\ideal A$. A partial ${}^*$-automorphism of~$A$ induces a partial action $\hat{\alpha}=(\{A_n\}_{n\in \mathbb{Z}},\{\hat{\alpha}_n\}_{n\in\mathbb{Z}})$ of the integers $\mathbb{Z}$ on~$A$ with $A_{-1}=I$, $A_1=J$ and $\hat{\alpha}_1=\alpha$, see \cite[Example 2.2]{McClan95}.  
    \end{examp}

    \begin{examp}[Partial actions on commutative $\Cast$-algebras]\label{ex: abelian-partial-ac} Let $X$ be a locally compact Hausdorff space. Partial actions of~$G$ on $\mathrm{C}_0(X)$ are in one-to-one correspondence with topological partial actions of~$G$ on~$X$, i.e. pairs $\theta=(\{\mathrm{U}_g\}_{g\in G},\{\theta_g\}_{g\in G})$, where $\mathrm{U}_g\subseteq X$ is open and $\theta_g: \mathrm{U}_{g^\inv}\to \mathrm{U}_g$ is a homeomorphism for all $g\in G$, satisfying axioms analogous to those in Definition \ref{def: partial action}. This correspondence sends a topological partial action $\theta=(\{\mathrm{U}_g\}_{g\in G},\{\theta_g\}_{g\in G})$ on $X$ to $\theta^*=(\{\mathrm{C}_0(\mathrm{U}_g)\}_{g\in G},\{\theta^*_g\}_{g\in G})$, where $\theta_g^*(f)=f\circ\theta_{g^\inv}$, for $f\in \mathrm{C}_0(\mathrm{U}_{g^\inv}) $ and $g\in G$.
         \end{examp}

\begin{defi}\label{def: G-equivariant-ccp} Let $A$ and $B$ be $\Cast$-algebras and $\alpha=\partacg{A}{\alpha}$ and $\beta=\partacg{B}{\beta}$ be partial actions on~$A$ and~$B$, respectively. We will say that a contractive completely positive (ccp) map $\phi: A\to B$ is \emph{$G$-equivariant} if $\phi(A_g)\subseteq B_g$ and $\phi\circ\alpha_g=\beta_g\circ\phi$ for all $g\in G$.
    \end{defi}

The partial action $\alpha=\partacg{A}{\alpha}$ on $A$ extends uniquely to a partial action $\alpha^\ddual=(\{A^\ddual_g\}_{g\in G},\{\alpha_g^\ddual\}_{g\in G})$ on the bidual $A^\ddual$ of $A$, where $A^\ddual_g$ is viewed as a $\mathrm{W}^*$-ideal of $A^\ddual$ under the natural inclusion, see \cite[Example 2.3]{Busetal22}. This partial action is unital, and the inclusion $A\subseteq A^\ddual$ is $G$-equivariant. 

In this paper we will require an extension of~$\alpha=\partacg{A}{\alpha}$ to a (unital) partial action on the injective envelope of~$A$. Given an ideal $J\ideal A$ we will regard in what follows $I(J)$ as an ideal in $I(A)$ under the canonical identification, see Subsection~\ref{subsec: injective-alg}. As for global actions, we show next that a partial action on~$A$ naturally extends to a partial action on its injective envelope $I(A)$. 

\begin{prop}\label{prop: partac on I(A)}
Let $A$ be a $\Cast$-algebra and $I(A)$ its injective envelope. Let $\alpha=\partacg{A}{\alpha}$ be a partial action of~$G$ on $A$. Then there exists a unique partial action $I(\alpha)=\partacg{I(A)}{I(\alpha)}$ of~$G$ on~$I(A)$ such that $I(A)_g=I(A_g)$ for all $g\in G$ and the inclusion $A\subseteq I(A)$ is $G$-equivariant. In particular, $I(\alpha)$ is unital.
\end{prop}

\begin{proof} Fix $g\in G$. By injectivity of $I(A_g)$ there exists a ucp map $I(\alpha)_g \colon I(A_{g^\inv}) \to I(A_g)$ extending the ${}^*$-isomorphism $\alpha_g \colon A_{g^\inv} \to A_g$. This is surjective because $I(\alpha)_g \circ I(\alpha)_{g^\inv}=\id_{I(A_g)}$ by rigidity of the inclusion $A_g\subseteq I(A_g)$, and by essentiality $I(\alpha)_g $ is completely isometric. It follows that $I(\alpha)_g $ is a ${}^*$-isomorphism since it is a surjective complete isometry between $\Cast$-algebras. Uniqueness of $I(\alpha)_g$ follows because if $\beta_g:I(A_{g^\inv}) \to I(A_g)$ is another ucp map extending $\alpha_g$, then by rigidity both $I(\alpha)_g $ and $\beta_g$ are inverses of $I(\alpha)_{g^\inv}$, which implies that $\beta_g=I(\alpha)_g $. In particular, $I(\alpha)_g$ is a ${}^*$-isomorphism. By uniqueness of the central projection $p_g\in I(A)$ with $I(A_g)=p_gI(A)$ for $g\in G$ the pair $I(\alpha)=\partacg{I(A)}{I(\alpha)}$ satisfies the required axioms of a partial action, since $\alpha=\partacg{A}{\alpha}$ does. 
\end{proof}

\subsection{Crossed products by partial actions}\label{sec: partial crossed product}
As for global $\Cast$-dynamical systems, there are notions of full and reduced crossed products associated to a partial action of a discrete group on a $\Cast$-algebra. We briefly recall the main aspects of these constructions here, and refer the reader to \cite{Exel17, QuiRae97} for further details.  For the construction of the reduced crossed associated to $\alpha=\partacg{A}{\alpha}$ we will mainly follow the approach of Quigg and Raeburn in~\cite{QuiRae97}. For an equivalent construction of a reduced crossed product using the more general framework of Fell bundles we refer to the textbook of Exel~\cite{Exel17}.

Fix a partial $\Cast$-dynamical system $(A, G, \alpha)$. For each $g\in G$ and $a\in A_g$ let $a\delta_g\in \mathrm{C}_c(G, A)$ be the function given by  \begin{equation*}\label{eq: induced-expectation}
       a\delta_g(t)=  \begin{cases}
                    a & \text{if }\,  t=g,   \\
                    0 & \text{otherwise. }
                 \end{cases}\end{equation*}
Let $A\rtimes_{\mathrm{alg}}G$ be the subspace of $\mathrm{C}_c(G, A)$ spanned by $\{a\delta_g\,\vert\, g\in G, a\in A_g\}$. Notice that $b\in A\rtimes_{\mathrm{alg}}G $ has a unique expression as a finite sum $b=\sum_{g\in G}a_g\delta_g$ with $a_g\in A_g$ for all $g\in G$. Let $A\rtimes_{\mathrm{alg}}G$ be endowed with the multiplication and involution operations given on spanning elements by
\begin{align*}
    a\delta_g b\delta_h &:= \alpha_g( \alpha_{g^\inv}(a) b) \delta_{gh},\\
    (a\delta_g)^* &:= \alpha_{g^\inv}(a^*) \delta_{g^\inv}
\end{align*} for all $g,h\in G$, $a\in A_g, b\in A_h$. With these operations $A\rtimes_{\mathrm{alg}}G$ is a ${}^*$-algebra. Moreover, if $p\colon A\rtimes_{\mathrm{alg}}G\to \mathbb{R}_+ $ is a $\Cast$-seminorm on $A\rtimes_{\mathrm{alg}}G$, then $$p\big(\sum_{g\in G}a_g\delta_g\big)\leq \sum_{g\in G}\Vert a_g\Vert<\infty. $$ 

 \begin{defi} The \emph{full crossed product} of~$A$ by~$G$ associated to the partial action $\alpha=\partacg{A}{\alpha}$, denoted by $A\rtimes G$, or $A\rtimes_\alpha G$, is the enveloping $\Cast$-algebra of $(A\rtimes_{\mathrm{alg}}G, \Vert\cdot\Vert_{\max})$, where $\Vert\cdot\Vert_{\max}\colon A\rtimes_{\mathrm{alg}}G\to \mathbb{R}_+ $ is the $\Cast$-seminorm defined for $b\in A\rtimes_{\mathrm{alg}}G$ by $$\Vert b\Vert_{\max}:=\sup\{p(b)\mid p\text{ is a $\Cast$-seminorm on } A\rtimes_{\mathrm{alg}}G\}. $$ We will see in what follows that the $\Cast$-seminorm $\Vert\cdot\Vert_{\max}$ is in fact a $\Cast$-norm. 
 \end{defi}

A \emph{partial ${}^*$-representation}  of the group~$G$ in a unital $\Cast$-algebra $B$ is given by a map $u\colon G \to B$ such that $u_e = 1$, $u_{g^\inv} = u_g^*$ and $u_g u_h u_{h^\inv} = u_{gh} u_{h^\inv}$ for all $g,h\in G$.  From the definition we see that~$u_g$ is a partial isometry for all~$g\in G$.  A \emph{covariant representation} of the partial $\Cast$-dynamical system $(A, G, \alpha)$ in a unital $\Cast$-algebra~$B$ is a pair $(\pi,u)$, where $\pi \colon A \to B$ is a ${}^*$-homomorphism and $u\colon G\to B$ is a partial ${}^*$-representation such that
\begin{equation*}
u_g \pi(a)u_{g^\inv} = \pi(\alpha_g(a)) 
\end{equation*}
for all $g\in G, a\in A$. A covariant representation $(\pi, u)$ of $(A, G,\alpha)$ in $B$ induces a ${}^*$-homomorphism $\pi\times u\colon A\rtimes G\to B$ that maps a spanning element $\iota(a\delta_g)$ to $(\pi\times u)(\iota(a\delta_g))=\pi(a)u_g$ for $g\in G$, $a\in A_g$, where $\iota\colon A\rtimes_{\mathrm{alg}}G\to A\rtimes G$ is the canonical map. Conversely, if $\rho$ is a non-degenerate ${}^*$-representation of the full crossed product $A\rtimes G$ on a Hilbert space $\hilb$, then $\rho=\pi\times u$ for a unique covariant representation $(\pi, u)$ of $(A, G, \alpha)$ in $B(\hilb)$ such that~$\pi$ is a non-degenerate ${}^*$-representation of~$A$ on $\hilb$, and $u\colon G\to B(\hilb)$ is a partial ${}^*$-representation of~$G$ with $u_gu_{g^{-1}}$ the orthogonal projection onto $\pi(A_g)\hilb$ for all $g\in G$, see \cite[Theorem~13.2]{Exel17}. 

It will be very helpful for us in this paper to view covariant representations of $(A, G,\alpha)$ in terms of normal representations of $(A\rtimes G)^\ddual$. For each $g\in G$ let $q_g$ denote the unit of $A_g^\ddual$. As shown in \cite{QuiRae97} there exists a partial ${}^*$-representation $m\colon G\to (A\rtimes G)^\ddual$ such that $m_gm_{g^{-1}}=\iota_A^\ddual(q_g)$ for all $g\in G$, where $\iota_A\colon A\to A\rtimes G $ is the canonical ${}^*$-homomorphism (in fact, an inclusion) that sends $a\in A$ to $\iota_A(a\delta_e)$. Moreover, the pair $(\iota_A, m)$ is a covariant representation of $(A, G,\alpha)$ in $(A\rtimes G)^\ddual$ with $\iota(a\delta_g)=\iota_A(a) m_g$ for all $g\in G$ and $a\in A_g$. If $\rho\colon A\rtimes G\to B(\hilb)$ is a non-degenerate ${}^*$-representation and $\rho^\ddual\colon (A\rtimes G)^\ddual\to B(\hilb) $ is the unique normal extension of~$\rho$, then the covariant pair $(\pi, u)$ as above satisfying $\rho=\pi\times u$ is given by $u=\rho\circ m$ and $\pi=\rho\circ\iota_A$. More explicitly, it follows from \cite[Proposition~2.4]{EchRae95} that if $(u_\mu)_\mu$ is an approximate identity for~$A_g$, then $u_g=\rho(m_g)=\lim_\mu\rho(\iota(u_\mu\delta_g))$, where the limit is taken in the strong operator topology of $B(\hilb)$ (see also \cite[p. 293]{EchRae95}).

For the construction of the reduced crossed product associated to~$\alpha$, let $\pi\colon A\to B(\hilb)$ be a faithful non-degenerate representation of $A$. Let $\pi^\ddual$ be the normal extension of $\pi$ and $\alpha^\ddual$ the partial action on $A^\ddual$ extending~$\alpha$. Define a ${}^*$-homomorphism $\pi^r\colon A\to B(\hilb\otimes\ell^2(G))$ by setting $$\pi^r(a):=\sum_{t\in G}\pi^\ddual(\alpha^\ddual_{t^{-1}}(aq_t))\otimes\chi_t,$$ where $\chi_t$ denotes the characteristic function of $\{t\}$, and the sum is taken under the strong${}^*$ operator topology of $B(\hilb\otimes \ell^2(G))$. Let $\lambda\colon G\to B(\ell^2(G))$, $g\mapsto \lambda_g$ be the left-regular representation of~$G$ and define a partial ${}^*$-representation $\nu\colon G\to B(\hilb\otimes \ell^2(G)) $ by putting $$\nu_g:=\sum_{t\in G}\pi^\ddual(q_{t^\inv}q_{t^\inv g})\otimes\chi_t\lambda_g,\qquad g\in G.$$ Then $(\pi^r,\nu)$ is a covariant representation of $(A,G, \alpha)$ such that $\pi^r$ is faithful. 
 \begin{defi} The \emph{reduced crossed product} of~$A$ by~$G$ associated to the partial action $\alpha=\partacg{A}{\alpha}$, denoted by $A\crossed G$, or $A\rtimes_{\alpha, r} G$, is the range of the ${}^*$-representation $\pi^r\times\nu$ of $A\rtimes G$ induced by the covariant pair $(\pi^r, \nu)$. 
 \end{defi}

 The reduced crossed product does not depend on the choice of the faithful representation of~$A$. We will identify~an element $a\in A$ with its image in~$A\crossed G$ under $\pi^r$, so for $g\in G$ and $a\in A_g$ we will simply write $a\nu_g$ for the spanning element $\pi^r(a)\nu_g$. Observe that $\nu_g\in A\crossed G$ for all $g\in G$ if and only if $\alpha$ is unital. The compression by the orthogonal projection of $\hilb\otimes \ell^2(G)$ onto $\hilb\otimes\chi_e$ induces a faithful conditional expectation $E_A\colon A\crossed G\to A$ satisfying $E_A(a\nu_g)=\delta_{e,g}a\nu_g$, where $\delta_{e,g}$ is the Kronecker delta. The existence of a faithful conditional expectation onto~$A$ with this property characterizes the reduced crossed product.

\section{Injective envelopes for unital partial \texorpdfstring{$\Cast$}{C*}-dynamical systems}\label{sec: inj envelopes partial}

In this section we will extend the construction of injective envelopes for $\Cast$-dynamical systems \cite{Ham85} to the setting of unital partial actions. The non-unital setting will be treated in \cref{sec: non-unital}.

\subsection{The category of generalized unital partial \texorpdfstring{$\Cast$}{C*}-dynamical systems}

We will require a notion of a partial action via hereditary subalgebras.

\begin{defi}\label{def: generalized partial action}
   A \emph{generalized unital partial action} of~$G$ on a $\Cast$-algebra $A$ is given by a pair $\alpha=\partacg{A}{\alpha}$, where $\{A_g\}_{g\in G}$ is a family of unital hereditary subalgebras of~$A$ and $\alpha_g \colon A_{g^\inv} \to A_g$ is a ${}^*$-isomorphism for all $g\in G$, such that 
   \begin{enumerate}
       \item $A_e = A$, and $\alpha_e=\id$,
       \item $\alpha_g(A_{g^\inv} \cap A_h) = A_{gh} \cap A_g$ for all $g,h\in G$,
       \item $\alpha_g(\alpha_h(x)) = \alpha_{gh}(x)$ for all $x\in A_{h^\inv}\cap A_{(gh)^\inv}$ for all $g,h\in G$,

              \item\label{item: commuting projection} $\{p_g\}_{g\in G}$ are commuting projections in~$A$, where $p_g$ is the unit of $A_g$ for all $g\in G$.
 \end{enumerate}
 We call $(A,G, \{A_g\}_{g\in G}, \{\alpha_g\}_{g\in G})$, or simply $(A, G, \alpha)$, a \emph{generalized unital partial $\Cast$-dynamical  system}. 
\end{defi}

Note that \cref{item: commuting projection} implies that the family of units forms a commutative subspace lattice \cite{Arv74}. Also, \cref{def: generalized partial action} naturally extends the notion of a (unital) partial action since ideals are, in particular, hereditary subalgebras, and \cref{item: commuting projection} holds automatically in this case. In case $A=\mathrm{C}(X)$ is commutative, a generalized unital partial action on $\mathrm{C}(X)$ is simply a unital partial action as in \cref{def: partial action}, i.e. a topological partial action on $X$ via clopen sets, so that the two definitions coincide in this case. Hence \cref{def: generalized partial action} is only relevant when $A$ is noncommutative. 

A natural class of generalized unital partial actions arises from restrictions of global actions to certain hereditary subalgebras. We will discuss further key examples later in this section.

\begin{examp}\label{ex: restriction-hereditary} Let $A$ be a unital $\Cast$-algebra and $(A,G,\alpha)$ a global $\Cast$-dynamical system. Let $B\subseteq A$ be a hereditary subalgebra with unit $p\in B$ and suppose that $\{\alpha_g(p)\}_{g\in G}$ are commuting projections in $A$. For each $g\in G$ set $p_g:=p\alpha_g(p)$. Then $\{p_g\}_{g\in G}$ are commuting projections, and as for partial actions in the usual sense we obtain a generalized unital partial action $\beta=\partacg{B}{\beta}$ on $B$ with $B_g = B \cap \alpha_g(B)=p_g Ap_g$ and $\beta_g = \alpha_g \vert_{B_{g^\inv}}$ for $g\in G$. We call this the \emph{restriction} of~$\alpha$ to~$B$.
\end{examp}

The $G$-equivariant ccp maps between generalized unital partial actions are defined as for partial actions, see \cref{def: G-equivariant-ccp}. We will consider the category of generalized unital partial actions with a special class of $G$-equivariant maps as morphisms, which we define below. 

\begin{defi}\label{def: g-morphism}
   Let $(A,G,\alpha)$ and $(B, G,\beta)$ be generalized unital partial $\Cast$-dynamical systems with corresponding commuting families of projections $\{p_g\}_{g\in G}$ and $\{q_g\}_{g\in G}$, respectively.  A ucp map $\phi \colon A \to B$ is called a \emph{$G$-morphism} if
    \begin{enumerate}
        \item $\phi$ is \emph{$G$-unital}, i.e. $\phi(p_g) = q_g$ for all $g\in G$,
        \item $\phi$ is \emph{$G$-equivariant}, i.e. $\phi(\alpha_g(x)) = \beta_g(\phi(x))$ for all $g\in G$ and $x\in A_{g^\inv}$.
    \end{enumerate}
    If $\phi$ is in addition completely isometric, we call $\phi$ a \emph{$G$-embedding}. If $\phi$ is a surjective $G$-embedding, we call $\phi$ a \emph{$G$-isomorphism}. Notice that $\phi$ is necessarily a ${}^*$-isomorphism in this case. 
    
    The category of generalized unital partial $\Cast$-dynamical systems has generalized unital partial $\Cast$-dynamical systems as its objects and $G$-morphisms as morphisms.
\end{defi}

 \begin{rem} Observe that Axiom (ii) of \cref{def: g-morphism} requires that $\phi(A_g) \subseteq B_g$ for all $g\in G$. To see that this inclusion is automatic, let $\phi\colon A\to B$ be a ucp map between $\Cast$-algebras $A$ and $B$ and recall that the \emph{multiplicative domain $\dom_\phi$ of $\phi$} is defined as 
\begin{align*}
\dom_\phi &:= \{ a\in A \; \vert \; \phi(a^*a) = \phi(a)^*\phi(a) \text{ and } \phi(aa^*)= \phi(a) \phi(a)^*\}.
\end{align*}
An element $a\in A$ lies in  $\dom_\phi$ if and only if $\phi(ab)= \phi(a)\phi(b)$ and $\phi(ba) = \phi(b)\phi(a)$ for all $b\in A$. If $\phi$ is as in \cref{def: g-morphism}, it follows that each projection~$p_g$ lies in the multiplicative domain of $\phi$ because $\phi$ is $G$-unital. Hence $\phi(A_g)=\phi(p_g Ap_g)\subseteq q_g B q_g$ for all $g\in G$.  
\end{rem}

Our main motivating example of a generalized unital partial action comes from partial ${}^*$-representations. An example of a $G$-morphism that will be very useful in this paper arises from covariant representations.

\begin{examp}\label{ex: partial action on B(H)} Let $u: G\to B(\hilb)$, $g\mapsto u_g$ be a partial ${}^*$-representation. Then $u$ induces a generalized unital partial action $\beta^u=(\{B_g^u\}_{g\in G},\{\beta_g^u\}_{g\in G})$ on $B(\hilb)$ as follows. For each $g\in G$ set $q^u_g:= u_gu_{g^{-1}}$, so that $q^u_g$ is the range projection of $u_g$. Then $\{q_g\}_{g\in G}$ are commuting projections in $B(\hilb)$ \cite[Proposition 9.8]{Exel17}. The unital hereditary subalgebra~$B^u_g$ is defined by $B^u_g:= q^u_g B(\hilb) q^u_g $ for all $g\in G$. The ${}^*$-isomorphism $\beta^u_g: B^u_{g^{-1}}\to B^u_g$ is given by $\beta^u_g(b) = u_g b u_g^*$ for $g\in G$ and $b\in B^u_{g^{-1}}$.  Notice that $B^u_g$ is an ideal in~$B(\hilb)$ if and only if $q^u_g=1$. 

 Let $\alpha=\partacg{A}{\alpha}$ be a partial action and $(\pi,u)$ a covariant representation of $(A, G,\alpha)$ on $\hilb$. As above this induces a generalized unital partial action $\beta^u=(\{B_g^u\}_{g\in G},\{\beta_g^u\}_{g\in G})$ on~$B(\hilb)$. Moreover, observe that $\pi$ is a $G$-equivariant ${}^*$-homomorphism. If $\alpha$ is unital, then $\pi$ is  $G$-unital and hence a $G$-morphism. 
\end{examp}

Another class of examples of $G$-morphisms that will be important to us comes from crossed products. Recall that in the case of a global $\Cast$-dynamical system $(A, G, \alpha)$ with~$A$ unital, the action of~$G$ on~$A$ extends to an inner action on the crossed product $A\crossed G$ with the ${}^*$-automorphisms being implemented by the family of unitaries $\{\lambda_g\}_{g\in G}$ in $A\crossed G$, where $g\mapsto \lambda_g$ is the canonical unitary representation of $G$ in $A\crossed G$. The analogous statement is not true in general for unital partial actions. However, we observe below that a unital partial action of $G$ on a $\Cast$-algebra naturally induces a generalized unital partial action on the corresponding crossed product. 

\begin{prop}\label{prop: gen part ac on crossed product}
    Let $(A, G, \alpha)$ be a unital partial $\Cast$-dynamical system. Then there exists a generalized unital partial action $\beta=\partacg{B}{\beta}$ on $ A \crossed G$ such that the inclusion $A\subseteq A \crossed G $ is a $G$-morphism. 
    
    More specifically, if $\nu: G\to  A \crossed G$ denotes the canonical partial ${}^*$-representation of $G$ in $A \crossed G$ and $p_g$ is the unit of the ideal $A_g$, then $B_g=p_g(A \crossed G)p_g$ and the ${}^*$-homomorphism $\beta_g: B_{g^\inv}\to B_g$ is implemented by the partial isometry~$\nu_g$ for all $g\in G$.
    \end{prop}

\begin{proof} Notice that the canonical partial ${}^*$-representation $\nu: G\to  A \crossed G$ satisfies $\nu_g\nu_{g^{-1}}=p_g$ and $\nu_g a \nu_{g^{-1}}=\alpha_g(p_{g^\inv} a)$ for all $g\in G$ and $a\in A$ since $\nu_g$ implements $\alpha_g$ on $A_{g^\inv}=p_{g^\inv}A$. In particular, the map $\beta_g: B_{g^\inv}\to B_g$, $b\mapsto \nu_g b \nu_{g^{-1}}$ is a ${}^*$-isomorphism that restricts to $\alpha_g$ on $A_{g^\inv}$. Axiom (ii) in \cref{def: generalized partial action} follows because $\alpha_g(p_{g^{-1}}p_h)=p_gp_{gh}\leq p_{gh}$ and Axiom (iii) follows since $\nu_g \nu_h \nu_{h^{-1}}=\nu_{gh} \nu_{h^{-1}}$ for all $g, h\in G$. Also, from above the inclusion $A\subseteq A \crossed G$ is a $G$-morphism. 
\end{proof}

\begin{rem}\label{rem: crossed-action-no-units} Observe that for each $g\in G$ the hereditary subalgebra $B_g$ of $A \crossed G$ in Proposition~\ref{prop: gen part ac on crossed product} can be naturally identified with the reduced crossed product $A_g\crossed G$ of the induced partial action of~$G$ on~$A_g$. More precisely, the underlying family of ideals of this partial action is $\{A_{g,h}\}_{h\in G}$ where $A_{g,h}:=A_g\cap A_h\cap A_{hg}$ with the natural collection of ${}^*$-isomorphisms $A_{g,h^\inv}\cong A_{g,h}$ given by restriction. Thus we have $$B_g=\overline{\lincomb}\{ a \nu_h \, \vert \, a \in A_g\cap A_h\cap A_{hg}\}.$$ This description does not explicitly require the family of units, and will be used in \cref{sec: non-unital} when we discuss non-unital partial actions.
\end{rem}


Another important example of a $G$-morphism in the setting of partial actions arises from the canonical conditional expectation on the reduced crossed product.

\begin{examp}
    Let $(A, G, \alpha)$ be a unital partial $\Cast$-dynamical system and consider the canonical generalized unital partial action on the reduced crossed product $A\crossed G$ as in Proposition~\ref{prop: gen part ac on crossed product}. Let $E_A \colon A\crossed G\to A$ be the canonical conditional expectation. Then~$E_A$ is a $G$-morphism since the inclusion $A\subseteq A\crossed G$ is.
\end{examp}

\subsection{\texorpdfstring{$G$}{G}-Injectivity and \texorpdfstring{$G$}{G}-Essentiality}

 An \emph{extension} of a generalized unital partial $\Cast$-dynamical system $(A, G, \alpha)$ is a pair $((B, G, \beta),\kappa)$, where $(B, G, \beta)$ is a generalized unital partial $\Cast$-dynamical system and $\kappa\colon A \to B$ is a $G$-embedding. When the $G$-embedding $\kappa$ is understood, we simply say that $(B, G, \beta)$ is an extension of $(A, G, \alpha)$.

We will need the following analog of Proposition \ref{prop: partac on I(A)} in the setting of partial actions via unital hereditary subalgebras. 

\begin{prop}\label{prop: injective-p-action-hereditary}
Let $A$ be a $\Cast$-algebra and $I(A)$ its injective envelope. Let $\alpha=\partacg{A}{\alpha}$ be a generalized unital partial action of~$G$ on~$A$ with corresponding family of commuting projections $\{p_g\}_{g\in G}$. Then there exists a unique generalized unital partial action $I(\alpha)=\partacg{I(A)}{I(\alpha)}$ of~$G$ on~$I(A)$ such that $I(A)_g=I(A_g)$ for all $g\in G$ and the inclusion $A\subseteq I(A)$ is a $G$-morphism. In particular, $(I(A), G, I(\alpha))$ is an extension of $(A, G,\alpha)$.
\begin{proof} The argument in the proof of Proposition~\ref{prop: partac on I(A)} shows that there exists a unique generalized unital partial action $I(\alpha)=\partacg{I(A)}{I(\alpha)}$ on $I(A)$ with $I(A)_g=I(A_g)$ for all $g\in G$ and such that the canonical inclusion $A\subseteq I(A)$ is $G$-equivariant. By \cite[Proposition 6.3]{Ham82_tensor} we have $I(A_g)=p_gI(A)p_g$ for all $g\in G$ and hence the inclusion $A\subseteq I(A)$ is also $G$-unital, i.e. a $G$-morphism, so that $(I(A), G, I(\alpha))$ is an extension of $(A, G,\alpha)$.
    \end{proof}
\end{prop}

\begin{defi}\label{def: partial-g-injectivity}
We will say that generalized unital  partial $\Cast$-dynamical system $(B,G,\beta)$ is \emph{injective} if for any pair of generalized unital partial actions $\gamma=\partacg{C}{\gamma}$ and $\alpha=\partacg{A}{\alpha}$ with a $G$-embedding $\kappa \colon A \to C$ and $G$-morphism $\phi \colon A \to B$ there exists a $G$-morphism $\Hat{\phi}\colon C \to B$ such that $\hat{\phi}\circ \kappa = \phi$. In other words, the diagram \begin{align*}
    \begin{tikzcd}[ampersand replacement=\&]
         A  \arrow[d,hook, swap, "{\kappa}"] \arrow[r, "\phi"] \& B  \\
        C \arrow[ru, dashed, swap, "{\hat{{\phi}}}"] \&
        \end{tikzcd}
\end{align*}
commutes. When there is no chance of confusion we will omit the generalized unital partial action~$\beta$, and simply say that $B$ is $G$-injective. 
\end{defi}

We call an extension $((B, G,\beta),\kappa)$ of $(A,G,\alpha)$ \emph{injective} if $(B, G,\beta)$ is injective. An extension $((B, G,\beta),\kappa)$ of $(A,G,\alpha)$ is called \emph{essential} if given a generalized unital partial action $\gamma=\partacg{C}{\gamma}$ and a $G$-morphism $\phi\colon B \to C $ such that $\phi \circ \kappa$ is a $G$-embedding, then~$\phi$ is a $G$-embedding. 

\begin{rem}\label{rem: global-G-morphism} Note that if $A$ is a unital $\Cast$-algebra and $(A, G,\alpha)$ is a global $\Cast$-dynamical system, i.e. $A_g = A$ for all $g\in G$, then $(A, G,\alpha)$ is injective as in \cref{def: partial-g-injectivity} if and only if $A$ is $G$-injective in the sense of Hamana~\cite{Ham85}. This follows because if $\beta=\partacg{B}{\beta}$ is a generalized unital partial action and $\phi\colon A\to B$ is a $G$-morphism, then $\beta$ must be a global action. 
    \end{rem}

Our next immediate goal is to show that every generalized unital partial $\Cast$-dynamical system admits an injective extension, following the main steps in the construction of injective envelopes in~\cite{Ham79, Ham85}.

Recall that for a $\Cast$-algebra $A$ we can equip $\ell^\infty(G,A)$ with the canonical left translation action~$\tau$ given by $\tau_g(f)(t) = f(g^\inv t)$ for $g,t\in G$ and $f\in \ell^\infty(G,A)$. Furthermore, \cite[Lemma 2.2]{Ham85} shows that if $A$ is injective, then $\ell^\infty(G,A)$ is $G$-injective. In our setting we will require a construction that utilizes a corner of~$\ell^\infty(G,A)$ and the corresponding restriction of the left translation action. 

Let $\alpha=\partacg{A}{\alpha}$ be a generalized unital partial action with corresponding family of commuting projections $\{p_g\}_{g\in G}$, i.e. $A_g = p_gAp_g$ for all $g\in G$. Define a projection $p\in \ell^\infty(G,A) $ by $p(t) := p_{t^\inv}$ for $t\in G$. Then $\{\tau_g(p)\}_{g\in G}$ are commuting projections in $\ell^\infty(G,A)$ and hence the restriction of the left translation action as in Example~\ref{ex: restriction-hereditary} is a generalized unital partial action on the corner $p\ell^\infty(G,A)p$ whose corresponding family of commuting projections is $\{p\tau_g(p)\}_{g\in G}$.  
In particular, we have
\begin{equation*}\label{eq: def of Bg}
\begin{aligned}
    p\tau_g(p) \ell^\infty(G,A)p\tau_g(p)&= \{ f\in \ell^\infty(G,A) \, \vert \, f(t) \in A_{t^\inv}\cap A_{t^\inv g}\}\\
   & = \{ f\in \ell^\infty(G,A) \, \vert \, f(t) \in p_{t^{-1}}p_{t^{-1}g}Ap_{t^{-1}}p_{t^{-1}g}\}.
    \end{aligned}
\end{equation*}
Observe that the map $\iota \colon A \to p\ell^\infty(G,A)p$ defined by $\iota(a)(t) = \alpha_{t^\inv} (p_t a p_t)$ for $t\in G$ and $a\in A$ is a ${}^*$-homomorphism if $\alpha=\partacg{A}{\alpha}$ is a partial action in the usual sense, which has already appeared before in the setting of enveloping actions of partial actions. See, for instance, the proofs of \cite[Theorem 4.5]{DokEx05} and \cite[Proposition~2.7]{Abetal22}. We will show next that it is in general a $G$-embedding.

\begin{prop}\label{prop: G-embed into B}
    Let $\alpha=\partacg{A}{\alpha}$ be a generalized unital partial action of~$G$ on a $\Cast$-algebra~$A$ with corresponding family of commuting projections $\{p_g\}_{g\in G}$. Consider the generalized unital partial action on $p\ell^\infty(G,A)p$ given by the restriction of the left translation action~$\tau$ as in Example~\ref{ex: restriction-hereditary}, where $p(t)=p_{t^{-1}}$ for all $t\in G$. Then the map  $\iota \colon A \to p\ell^\infty(G,A)p$ defined by $\iota(a)(t) = \alpha_{t^\inv} (p_t a p_t)$ for $t\in G$ and $a\in A$ is a $G$-embedding.
\end{prop}

\begin{proof} That $\iota$ is completely isometric follows by setting~$t=e$.
 To see that $\iota$ is a $G$-morphism, first observe that
 $$\iota(p_g)(t) = \alpha_{t^\inv}(p_t p_g p_t) = \alpha_{t^\inv}(p_t p_g)\alpha_{t^\inv}(p_t) = p_{t^\inv} p_{t^\inv g}p_{t^\inv} = p_{t^\inv} p_{t^\inv g} = p(t) \tau_g(p)(t).$$
We used above that $\alpha_g(p_{g^\inv}p_h) = p_gp_{gh}$ for $g,h\in G$. This shows that $\iota$ preserves the units and hence is $G$-unital.

It remains to show that $\beta_g(\iota(a)) = \iota(\alpha_g(a))$ for all $g\in G$ and $a\in A$. To see this let $a\in A_{g^\inv}$. Then for $t\in G$ we obtain
\begin{align*}
\beta_g(\iota(a))(t) &= \iota(a)(g^\inv t) = \alpha_{t^\inv g} (p_{g^\inv t} a p_{g^\inv t}) \\&= \alpha_{t^\inv}(\alpha_{g}( p_{g^\inv t}p_{g^\inv} a p_{g^\inv t}p_{g^\inv}))\\
&= \alpha_{t^\inv}(\alpha_{g}( p_{g^\inv t}p_{g^\inv} ) \alpha_{g}( a) \alpha_{g}( p_{g^\inv t}p_{g^\inv} ) ) \\&= \alpha_{t^\inv} (p_t p_g \alpha_g(a)p_t p_g) \\
&= \alpha_{t^\inv} (p_t \alpha_g(a)p_t) = \iota(\alpha_g(a))(t),
\end{align*}  
showing that $\beta_g(\iota(a)) = \iota(\alpha_g(a))$ as needed.
\end{proof}

We will show below that $p\ell^\infty(G,A)p$ is $G$-injective whenever $A$ is injective (in the category of operator systems).

\begin{prop}\label{prop: inj G-ext}
Let $\alpha=\partacg{A}{\alpha}$ be a generalized unital partial action of a discrete group~$G$ on a $\Cast$-algebra~$A$ with corresponding family of commuting projections $\{p_g\}_{g\in G}$. Suppose that $A$ is injective and consider $p\ell^\infty(G,A)p$ with the generalized unital partial action given by the restriction of the left translation action~$\tau$ as in Example \ref{ex: restriction-hereditary}, where $p(t)=p_{t^{-1}}$ for $t\in G$. Then $p\ell^\infty(G,A)p$ is $G$-injective. 
\end{prop}

\begin{proof}
To show that $p\ell^\infty(G,A)p$ is $G$-injective, we proceed as in \cite[Lemma~2.2]{Ham85}. Let $(C,G,\gamma)$ be a generalized unital partial $\Cast$-dynamical system and $\phi \colon C \to p\ell^\infty(G,A)p$ a $G$-morphism. Let $((D,G,\delta), \kappa)$ be an extension of $(C,G,\gamma)$. For each $g\in G$ let~$q_g$ and~$r_g$ denote the units of~$C_g$ and~$D_g$, respectively. The map $\psi \colon C \to A \colon x \mapsto \phi(x)(e)$ is a ucp~map and, because $A$ is injective, it extends to a ucp~map $\hat{\psi}\colon D \to A$ such that $\hat{\psi}\circ \kappa = \psi$. Define a map $\hat\phi \colon D \to \ell^\infty(G,A)$ by $$\hat\phi(x)(t) := \hat\psi(\delta_{t^\inv}(r_t x r_t))\qquad\qquad (x\in D, t\in G).$$ We will first prove that $\hat\phi_g(r_g)=p\tau_g(p)$ for all $g\in G$ and, in particular, $\hat\phi(D)\subseteq p\ell^\infty(G,A)p $. To see this notice that $\psi(q_g) = \phi(q_g)(e) = p_g$ for all $g\in G$ because $\phi$ is $G$-unital. Since $\kappa$ is also $G$-unital we have for all $g\in G$ that $$\hat\psi(r_g) = \hat\psi(\kappa(q_g))= \psi(q_g) = p_g.$$ Thus $\hat\psi$ is $G$-unital and it follows immediately from the definition of~$\hat\phi$ that $\hat\phi(r_g)=p\tau_g(p)$ for all $g\in G$. Hence $\hat\phi(D) \subseteq p\ell^\infty(G,A)p$ and $\hat\phi$ is $G$-unital. 

It remains to show that $\hat\phi$ is $G$-equivariant and $\hat\phi \circ \kappa = \phi$, i.e. $\hat\phi$ extends~$\phi$. Towards $G$-equivariance, let $g\in G$ and $x\in D_{g^\inv}$. We compute for $t\in G$
\begin{align*}
\hat\phi(\delta_g(x))(t) &= \hat\psi(\delta_{t^\inv}(r_t \delta_g(x)r_t)) =  \hat\psi(\delta_{t^\inv}(r_t r_g \delta_g(x) r_g r_t))\\
&= \hat\psi(\delta_{t^\inv}(\delta_g (\delta_{g^\inv}(r_t r_g )x\delta_{g^\inv}(r_t r_g )))) \\&= \hat\psi(\delta_{t^\inv g}(\delta_{g^\inv}(r_t r_g )x\delta_{g^\inv}(r_t r_g ))\\
&= \hat\psi(\delta_{t^\inv g}(r_{g^\inv t} r_{g^\inv} x r_{g^\inv t} r_{g^\inv})) \\&=  \hat\psi(\delta_{t^\inv g}(r_{g^\inv t} x r_{g^\inv t})) \\
&=\hat\phi(x)(g^\inv t) = \beta_g(\hat\phi(x))(t).
\end{align*} 
Thus $\hat\phi$ is $G$-equivariant.

To see that $\hat\phi \circ \kappa = \phi$, recall first that the units belong to the multiplicative domain of a $G$-morphism. Now for $x\in C$ and $t\in G$ we obtain 
\begin{align*}
\hat\phi(\kappa(x))(t) &= \hat\psi(\delta_{t^\inv}(r_t \kappa(x)r_t)) = \hat\psi(\delta_{t^\inv}(\kappa(q_t x q_t))) \\
&= \hat\psi(\kappa(\gamma_{t^\inv}(q_t x q_t))) = \psi(\gamma_{t^\inv}(q_t x q_t))\\
&= \phi(\gamma_{t^\inv}(q_t x q_t))(e) = \phi(q_t x q_t)(t)\\ 
&= \phi(q_t)(t)\phi(x)(t)\phi(q_t)(t) =  \phi(x)(t).
\end{align*}
In the last two lines above we used that $\phi\colon C\to p\ell^\infty(G,A)p$ is a $G$-morphism, while in the first line we used that $\kappa: C\to D$ is so.
\end{proof}

The following corollary is a crucial step in the construction of injective envelopes for unital partial $\Cast$-dynamical systems.

\begin{cor} Every generalized unital partial $\Cast$-dynamical system admits an injective extension. 
\begin{proof} This follows from Propositions~\ref{prop: G-embed into B} and \ref{prop: inj G-ext},  since every generalized unital partial $\Cast$-dynamical system admits an extension whose underlying $\Cast$-algebra is injective by Proposition~\ref{prop: injective-p-action-hereditary}.
    \end{proof}
    \end{cor}

We also obtain a characterization of $G$-injectivity as in the setting of global actions.

\begin{prop}\label{prop: alternative char G-inj}
  Let $\alpha=\partacg{A}{\alpha}$ be a generalized unital partial action of~$G$ on a $\Cast$-algebra $A$ with corresponding family of commuting projections $\{p_g\}_{g\in G}$. Let $p\in \ell^\infty(G,A)$ and $\iota\colon A \to p\ell^\infty(G,A)p $ be as in Proposition~\ref{prop: G-embed into B}. Then $(A,G,\alpha)$ is injective if and only if $A$ is injective and there exists a $G$-morphism $\phi \colon p\ell^\infty(G,A)p \to A$ such that $\phi \circ \iota = \id_A$.
\begin{proof} The statement follows by arguing as in \cite[Remark 2.3]{Ham85}
\end{proof}
\end{prop}

By \cite[Proposition~5.1]{Ham85}, corners of injective objects in the category of $\Cast$-dy\-nam\-ic\-al systems are again $G$-injective. We show next that a unital partial action coming from the restriction of an injective $\Cast$-dynamical system gives rise to an injective object in the category of generalized unital partial $\Cast$-dynamical systems. Our proof will require uniqueness of the enveloping action of a partial action, see \cite{Aba03} and also \cref{sec: enveloping action}.

\begin{cor}\label{cor: restriction is partially G-inj}
    Let $(A,G,\alpha)$ be a global $\Cast$-dynamical system. Suppose that $A$ is $G$-injective and let $q\in A$ be a central projection. Then the unital partial $\Cast$-dynamical system obtained by restricting $\alpha$ to~$Aq$ is injective.  
\end{cor}

\begin{proof} Set $C:=Aq$. Since $A$ is injective and~$C$ is a corner of~$A$, it follows that $C$ is itself injective. Let $\gamma=\partacg{C}{\gamma}$ denote the unital partial $\Cast$-dynamical system given by the restriction of~$\alpha$ to~$C$.  For each $g\in G$ set $p_g:=q\alpha_g(q)$ and let $p\in  \ell^\infty(G, C)$ and $\iota\colon C\to\ell^\infty(G, C)$ be as in Proposition~\ref{prop: inj G-ext}. Let $\mathrm{Orb}_G^\tau(\iota(C))$ denote the $\Cast$-subalgebra of $\ell^\infty(G, C)$ spanned by the orbits of $\iota(C)$ under the left translation action. Similarly, let $\mathrm{Orb}_G^\alpha(C)$ be the $\Cast$-subalgebra of $A$ spanned the orbits of $C$ under~$\alpha$. By uniqueness of enveloping actions, there is a unique $G$-equivariant ${}^*$-isomorphism $\rho\colon\mathrm{Orb}_G^\tau(\iota(C))\to \mathrm{Orb}_G^\alpha(C)$ such that $\rho\circ \iota=\id_C$. 

Since $A$ is $G$-injective and $\rho$ is $G$-equivariant, there is a $G$-equivariant ucp map $\hat\rho\colon\allowbreak \ell^\infty(G, C)\to A $ satisfying $\hat\rho\circ\iota=\id_{C}$. In particular, $\hat\rho(p\tau_g(p))=q\alpha_g(q)$ for all $g\in G$ and by $G$-equivariance we deduce that the restriction of $\hat\rho$ to the corner $p\ell^\infty(G, C)p$ defines a $G$-morphism $\phi \colon p\ell^\infty(G,C)p \to C$ such that $\phi \circ \iota = \id_C$. Thus $(C, G,\gamma)$ is injective by Proposition~\ref{prop: alternative char G-inj}.
\end{proof}

\subsection{Injective envelopes for unital partial \texorpdfstring{$\Cast$}{C*}-dynamical systems}\label{sec: inj env construction}

In this section we extend work of Hamana \cite{Ham85} and show that every unital partial $\Cast$-dynamical system admits an injective envelope in the category of generalized unital partial $\Cast$-dynamical systems. Our proof utilizes the construction from Proposition~\ref{prop: inj G-ext} and the existence of faithful covariant representations of a partial action on Hilbert space. We will treat non-unital partial actions in \cref{sec: non-unital}.

\begin{defi}
Let $(A, G, \alpha)$ be a unital partial $\Cast$-dynamical system. An \emph{injective envelope} of $(A, G, \alpha)$ is an extension that is injective and essential. 
\end{defi}

As in \cite{Ham79, Ham85} a property of an extension that is closely related to essentiality and is often more convenient to work with is rigidity. An extension $((B,G,\beta),\kappa)$ of $(A,G,\alpha)$ is called \emph{rigid} if whenever $\phi\colon B \to B$ is a $G$-morphism such that $\phi \circ \kappa = \kappa$, then $\phi = \id_B$. The next lemma shows that rigidity is equivalent to essentiality in case of injective extensions for unital partial actions. For the proof we adapt the argument in \cite[Lemma~3.7]{Ham79} slightly to address the fact that the objects in our category come from partial actions on $\Cast$-algebras rather than on operator systems. 

\begin{lem}\label{lem: rigid implies essential}
Let  $\alpha=\partacg{A}{\alpha}$ be a unital partial action and $((B,G,\beta),\kappa)$ an injective extension of $(A,G,\alpha)$. Then $((B,G,\beta),\kappa)$ is rigid if and only if it is essential. 
\end{lem}
\begin{proof} Suppose first that $((B,G,\beta),\kappa)$ is a rigid extension of~$(A,G,\alpha)$. Let $(C,G,\gamma)$ be a generalized unital partial $\Cast$-dynamical system with corresponding family of commuting projections $\{r_g\}_{g\in G}$ and $\phi\colon B\to C$ a $G$-morphism such that $\phi\circ\kappa$ is completely isometric. Since $\phi\circ \kappa\colon A\to C$ is a $G$-morphism, we have $r_g= \phi(\kappa(p_g))$ for all $g\in G$, where~$p_g$ is the unit of~$A_g$. A multiplicative domain argument shows that for all $g\in G$ the projection~$r_g$ commutes with $\operatorname{im}(\phi\circ\kappa)$ and, moreover, $r_g\Cast(\phi(\kappa(A)))=r_g\Cast(\phi(\kappa(A_g)))$. Because $\phi\circ\kappa$ is a $G$-morphism, this implies in particular that $\Cast(\phi(\kappa(A)))$ is $\gamma$-invariant, i.e. $\gamma_g\big(\Cast(\phi(\kappa(A)))\cap r_gCr_g\big)\subseteq \Cast(\phi(\kappa(A)))$ for all $g\in G$. Hence $\gamma$ restricts to a unital partial action on $\Cast(\phi(\kappa(A)))$ whose corresponding family of ideals is $\{\Cast(\phi(\kappa(A_g)))\}_{g\in G}$.

Let $\rho\colon \Cast(\phi(\kappa(A)))\to A $ be a ${}^*$-homomorphism such that $\rho\circ\phi\circ\kappa=\id_A$, where we are using here that~$A$ is the $\Cast$-envelope of~$A$ regarded as an operator system. Then~$\rho$ is a $G$-morphism. Consider the $G$-morphism $\kappa\circ\rho\colon \Cast(\phi(\kappa(A)))\to B $. Since~$B$ is $G$-injective, we can find a $G$-morphism $\psi\colon C\to B$ such that $\psi\vert_{\Cast(\phi(\kappa(A)))}=\kappa\circ\rho$. Then the composition $\psi\circ \phi$ satisfies $\psi\circ\phi\circ\kappa=\kappa$, and by rigidity of $((B,G,\beta),\kappa)$ we deduce that $\psi\circ\phi=\id_B$. It follows that~$\phi$ is completely isometric as needed.

For the converse, let $\phi\colon B\to B $ be a $G$-morphism such that $\phi\circ\kappa=\kappa$. Then the assumption implies that $\phi$ is completely isometric. Arguing as above we see that $\Cast(\phi)$ is $\beta$-invariant and there exists a $G$-morphism $\rho\colon \Cast(\phi(B))\to B$ such that $\rho\circ\phi=\id_B $. By $G$-injectivity of~$B$ we get a $G$-morphism $\hat{\rho}\colon B\to B$ extending $\rho$. Since $\hat{\rho}\circ\phi\circ\kappa=\kappa$ we obtain that $\hat{\rho}$ is completely isometric again by essentiality of $((B,G,\beta),\kappa)$. In particular, $\rho$ is a faithful ${}^*$-homomorphism, which implies that $\phi$ is a ${}^*$-homomorphism, and $\phi\colon B\to B$ is surjective since $\hat{\rho}$ is completely isometric. This gives that $\phi$ is a $G$-isomorphism. Repeating this argument as in \cite[Lemma~3.7]{Ham79}  with the $G$-morphism $\psi:=\frac{\id_B+\phi}{2}$, which satisfies $\psi\circ \kappa=\kappa$, gives that $\psi$ is a ${}^*$-automorphism of~$B$, and hence $\psi=\id_B=\phi$ by extremality of ${}^*$-automorphisms among ucp maps. 
\end{proof}

In what follows we will write $(A, G, \alpha)\subseteq (B,G,\beta) $ if there is a $G$-embedding $\kappa\colon A\to B$ that is a ${}^*$-homomorphism, i.e.  an inclusion of generalized unital partial $\Cast$-dynamical systems.
The following lemma will be needed to show that the injective envelope of a partial $\Cast$-dynamical system remains in the category of partial actions. 

\begin{lem}\label{lem:isomorphism-envelopes-AB} Let $\alpha=\partacg{A}{\alpha}$ and $\beta=\partacg{B}{\beta}$ be unital partial actions. Suppose that there are inclusions $$(A, G,\alpha)\subseteq (B,G,\beta)\subseteq (I(A),G, I(\alpha))$$ and that $(I_G(A), G,I_G(\alpha))$, $(I_G(B),G,I_G(\beta))$ and $(I_G(I(A)), G, I_G(I(\alpha)))$ are their respective injective envelopes. Then the following inclusions hold:
\begin{enumerate}
    \item[\rm{(1)}] $(A, G,\alpha) \subseteq (B, G,\beta)\subseteq (I_G(A),G, I_G(\alpha))$; in particular, $I_G(A)=I_G(B)=I_G(I(A))$.

    \item[\rm{(2)}] $Z(A)\subseteq Z(B)\subseteq Z(I_G(A))$, and $I_G(A)$ is commutative if~$A$ is.
\end{enumerate}
\begin{proof} For part (1), notice that the inclusions $A\subseteq B\subseteq I(A)$ are $G$-embeddings. We will show that the inclusion $A\subseteq I_G(A)$ induces a $G$-embedding of $\Cast$-algebras $B\subseteq I_G(A)$. By $G$-injectivity of $I_G(A)$ there exists a $G$-morphism $\psi\colon I_G(B)\to I_G(A)$ extending the inclusion $A\subseteq I_G(A) $. Since this is completely isometric on~$A$, the essentiality of the inclusion $A\subseteq B$ implies that $\psi$ is completely isometric on~$B$. Hence $\psi$ is completely isometric by $G$-essentiality of $B\subseteq I_G(B)$. Applying now $G$-injectivity of~$I_G(B)$ we get a $G$-morphism $\phi\colon I_G(A)\to I_G(B)$ satisfying $\phi\circ\psi=\id_{I_G(B)}$. This is completely isometric by $G$-essentiality since its restriction to~$A$ is so. Because $\phi\circ\psi=\id_{I_G(B)}$, it follows that $\phi$ is also surjective and hence a $G$-isomorphism, giving the required inclusions. 

For part (2), observe that the inclusion $Z(A)\subseteq Z(I(A))$ holds by \cite[Corollary 4.3]{Ham79_CStar} and $I(A)$ is abelian if $A$ is. In particular, $Z(A)\subseteq Z(B)$ since $B\subseteq I(A)$. Let $\iota\colon I(A)\to p\ell^\infty(G,I(A))p$ be the $G$-embedding from Proposition \ref{prop: G-embed into B}. Then $\iota$ is a ${}^*$-homomorphism. Proposition~\ref{prop: inj G-ext} shows that $p\ell^\infty(G,I(A))p$ is $G$-injective and hence there is a $G$-morphism $\phi\colon I_G(A)\to p\ell^\infty(G,I(A))p$ such that $\phi\circ\hat{\kappa}=\iota,$ where we denote by $\kappa$ the $G$-embedding $I(A)\subseteq I_G(A)$ as in part (1). Now $\phi$ is a $G$-embedding by essentiality of the inclusion $(A, G,\alpha)\subseteq (I_G(A), G,I_G(\alpha))$ and we conclude that $\phi\colon I_G(A)\to \operatorname{im}\phi$ is a $G$-isomorphism, where $\operatorname{im}\phi$ is equipped with the Choi--Effros product. To see that $Z(I(A))\subseteq Z(I_G(A)) $, observe that $\{p_g\}_{g\in G}$ are central projections in $I(A)$, and hence $\iota\big(Z(I(A))\big)\subseteq Z(\ell^\infty(G, I(A)))$. Thus from the definition of the Choi-Effros product on $\operatorname{im}\phi$ we see that $Z(I(A))\subseteq Z(I_G(A))$ and that $I_G(A)$ is commutative if $A$ is, proving part (2).
    \end{proof}
    \end{lem} 

We are now ready to show that every unital partial action has a $G$-injective envelope. 

\begin{thm}\label{thm: existence inj env}
Let $(A, G, \alpha)$ be a unital partial $\Cast$-dynamical system. Then $(A, G, \alpha)$ has an injective envelope $((I_G(A), G,I_G(\alpha)),\kappa)$. Moreover, $I_G(\alpha)=\partacg{I_G(A)}{I_G(\alpha)}$ is a unital partial action, and $(I_G(A), G,I_G(\alpha))$ is unique in the sense that if $((B, G,\beta), \kappa_B)$ is an injective envelope of $(A, G, \alpha)$, then there is a unique $G$-isomorphism $\psi\colon I_G(A)\to B$ such that $\psi\circ\kappa=\kappa_B$.
\end{thm}

\begin{proof} We will follow the approach in \cite[Theorem 4.9]{Ken+21}. By Lemma~\ref{lem: rigid implies essential} it suffices to show that $(A, G, \alpha)$ admits an injective extension that is rigid. Let $(\pi, u)$ be a faithful non-degenerate covariant representation of $(A, G, \alpha)$ on a Hilbert space $\hilb$. Consider the generalized unital partial action on $B(\hilb)$ induced by~$u$ as in Example~\ref{ex: partial action on B(H)} and consider also the generalized unital partial action on the corner $p\ell^\infty(G,B(\hilb))p$ obtained by restricting the left translation action on $\ell^\infty(G,B(\hilb))$ as in Proposition~\ref{prop: inj G-ext}. Here $p$ is the projection $p(t)=\pi(p_{t^{-1}})\in\ell^\infty(G,B(\hilb))$, where $p_g$ denotes the unit of $A_g$ for all $g\in G$. Then Proposition~\ref{prop: inj G-ext} implies that $p\ell^\infty(G,B(\hilb))p$ is $G$-injective. Let $\kappa\colon A \to p\ell^\infty(G,B(\hilb))p$ be the composition $\iota\circ \pi$, where $\iota\colon B(\hilb) \to p\ell^\infty(G,B(\hilb))p$ is the canonical $G$-embedding from Proposition~\ref{prop: G-embed into B}. Then $\kappa$ is a $G$-embedding because $\pi$ also is. Notice that $\kappa$ is in addition a ${}^*$-homomorphism. Also, observe that $p\ell^\infty(G,B(\hilb))p$ is a von Neumann algebra. Define
\begin{equation*}\label{eq: def of semigr}
S := \{\phi \colon p\ell^\infty(G,B(\hilb))p \to p\ell^\infty(G,B(\hilb))p \; \vert \; \phi \text{ is a } G\text{-morphism and } \phi\circ\kappa = \kappa\}.
\end{equation*}
This is a compact convex right topological semigroup under composition when equipped with the relative point-weak${}^*$ topology. Repeating the arguments in the proof of \cite[Theorem 4.9]{Ken+21} we obtain a minimal left ideal $L\subseteq S$. This is necessarily closed, and contains an idempotent, see \cite[Corollary 2.6]{HindStrau12}. The arguments in \cite[Theorem 4.9]{Ken+21} also show that $L$ is in fact a left zero semigroup, meaning that $\phi_1\circ \phi_2=\phi_1$ for all $\phi_1,\phi_2\in L$.

Now let $\psi\in L$ be an idempotent. Since $p\ell^\infty(G,B(\hilb))p$ is injective, it follows that $\operatorname{im}\psi$ is injective in the category of operator systems. Thus by \cite{ChoiEf77} it is isomorphic to a $\Cast$-algebra when endowed with Choi--Effros product $a\cdot b := \psi(ab)$. 
Since $\psi \in S$,  $\psi$ is a $G$-morphism with $\psi \circ \kappa = \kappa$. Hence the generalized unital partial action on $p\ell^\infty(G,B(\hilb))p$ restricts to a generalized unital partial action on $\operatorname{im}\psi$ such that the inclusion $\kappa\colon A\to \operatorname{im}\psi$ is a $G$-embedding. Again since $\psi$ is an idempotent, $p\ell^\infty(G,B(\hilb))p$ is $G$-injective and $\psi\colon p\ell^\infty(G,B(\hilb))p\to \operatorname{im}\psi$ is a $G$-morphism, it follows that $\operatorname{im}\psi$ is $G$-injective too.

Write $I_G(A):=\operatorname{im}\psi$ and let $I_G(\alpha)=\partacg{I_G(A)}{I_G(\alpha)}$ denote the generalized unital partial action on $I_G(A)$ from above. Then $I_G(A)$ is $G$-injective and $((I_G(A), G,I_G(\alpha)),\kappa)$ is an injective extension of $(A, G, \alpha)$. We will show now that this extension is rigid, and hence an injective envelope of~$(A, G, \alpha)$ by Lemma~\ref{lem: rigid implies essential}. Let $\phi \colon I_G(A) \to I_G(A)$ be a $G$-morphism such that $\phi \circ \kappa = \id_{\kappa(A)}$. View $\phi$ as a $G$-morphism $\phi\colon I_G(A)\to p\ell^\infty(G,B(\hilb))p$ using the inclusion $I_G(A)\subseteq p\ell^\infty(G,B(\hilb))p$. Then $\phi\circ\psi \in S$ and since $\psi\in L$ is an idempotent, we actually have $\phi\circ\psi=\phi\circ\psi\circ\psi\in L$. Because $L$ is a left zero semigroup we then obtain that $\psi\circ \phi\circ\psi=\psi$. Thus, using again that $\psi$ is an idempotent and $\operatorname{im}\phi\circ\psi\subseteq\operatorname{im}\psi$ we get $\phi\circ\psi=\psi\circ \phi\circ\psi=\psi$, which shows that $\phi=\id_{I_G(A)}$ as needed. Uniqueness of the injective envelope as in the statement follows as in \cite{Ham79}, see also the argument in \cite[Theorem 4.9]{Ken+21}. 

Finally we claim that $I_G(\alpha)=\partacg{I_G(A)}{I_G(\alpha)}$ is a unital partial action, i.e. $I_G(A)_g \ideal I_G(A)$ for all $g\in G$. For this notice that since $A_g$ is an ideal in~$A$, the unit $p_g\in A_g$ belongs to the center $Z(A)$ of $A$, and since $\kappa$ is $G$-unital, the unit of $I_G(A)_g$ is $\kappa(p_g)$. Thus the claim follows immediately from Lemma~\ref{lem:isomorphism-envelopes-AB}.
\end{proof}

When there is no chance of confusion,  we will simply say that $I_G(A)$ is the \emph{$G$-injective envelope} of~$A$. Moreover, for convenience we will often suppress the map $\kappa$ and simply identify an element $a\in A$ with its image in~$I_G(A)$.

\begin{rem} Notice that if~$A$ is unital, then the injective envelope of a global $\Cast$-dy\-nam\-ic\-al system $(A, G,\alpha)$ as in Theorem~\ref{thm: existence inj env} coincides with the $G$-injective envelope of~$A$ in the sense of Hamana \cite{Ham81}, see Remark~\ref{rem: global-G-morphism}.
    \end{rem}

We will be frequently using the following important technical observation.
\begin{prop}\label{prop: inclusion into I_G(A) normal}
    Let $(A, G,\alpha)$ be a unital partial $\Cast$-dynamical system and consider the inclusion $\kappa\colon A \hookrightarrow I_G(A)$ of~$A$ into its $G$-injective envelope. Then $\kappa$ is normal, i.e. if $(x_\lambda)_\lambda$ is an increasing net in~$A$ with $\sup_\lambda x_\lambda = x \in A$, then $\sup_\lambda \kappa(x_\lambda) = \kappa(x)$.
\end{prop}
\begin{proof} The argument follows as in \cite[Lemma~3.1]{Ham85}(i), see also Lemma~\ref{lem:isomorphism-envelopes-AB}.\end{proof}

\subsection{Injective envelopes, pseudo-expectations and the intersection property}\label{subsec: pseudo-exp}

We begin by recalling the ideal intersection property for inclusions of $\Cast$-algebras and partial $\Cast$-dynamical systems.

\begin{defi} Let $A$ and $B$ be $\Cast$-algebras with an inclusion of $\Cast$-algebras $A\subseteq B$. We say that the inclusion $A\subseteq B$ has the \emph{ideal intersection property} if $J \cap A = \trivial$ for an ideal $J\ideal B$ implies that $J = \trivial$. We will say that a partial $\Cast$-dynamical system $(A, G,\alpha)$ has the \emph{ideal intersection property} if the inclusion $A\subseteq A\crossed G$ has the ideal intersection property.
\end{defi}

Notice that $A\subseteq A\crossed G$ has the ideal intersection property if and only if a ${}^*$-ho\-mo\-mor\-phism $\rho\colon A\crossed G\to B(\hilb)$ is faithful whenever $\rho\vert_A$ is.

We will require the notion of a pseudo-expectation in the context of partial actions. This was introduced by Pitts in the context of general inclusions of $\Cast$-algebras \cite{Pitts17} (see also \cite{PitZar15}), and its equivariant version has proven to be a very useful technical tool to study the ideal intersection property for $\Cast$-dynamical systems, see for instance \cite{KenScha19}, \cite{GefUrs23}, \cite{KKS25}. Recall from Proposition~\ref{prop: gen part ac on crossed product} that the reduced crossed product $A\crossed G$ associated to a unital partial action on~$A$ carries a canonical generalized unital partial action of~$G$ so that the inclusion $A\subseteq A\crossed G $ is a $G$-morphism. We will use this in what follows. 

\begin{defi}
    Let $(A, G,\alpha)$ be a unital partial $\Cast$-dynamical system. A ucp map $\phi \colon A\rtimes_r G \to I_G(A)$ is called a \emph{pseudo-expectation} for $(A,G,\alpha)$ if~$\phi$ is a $G$-mor\-phism and $\phi\vert_A = \id_A$.
\end{defi}

\begin{rem} As for global $\Cast$-dynamical systems, every $G$-equivariant conditional expectation from $A\crossed G$ onto~$A$ is a pseudo-expectation for $(A,G,\alpha)$ since $A \subseteq I_G(A)$. In particular, the canonical conditional expectation $E_A\colon A\crossed G\to A $ is a pseudo-expectation and hence pseudo-expectations always exist. See \cite[Remark 6.2]{KenScha19}. 
    \end{rem}

In general, pseudo-expectations for $(A, G, \alpha)$  correspond to $G$-equivariant conditional expectations for $(I_G(A), G, I_G(\alpha))$
\begin{prop}\label{prop: 1-1 pseudo-exp with cond exp}
    Let $(A, G, \alpha)$ be a unital partial $\Cast$-dynamical system. Then there is a one-to-one correspondence between pseudo-expectations $\phi \colon A \crossed G \to I_G(A)$ for $(A, G, \alpha)$ and $G$-equivariant conditional expectations $\Phi \colon I_G(A) \crossed G \to I_G(A)$. 
\end{prop}

\begin{proof} 
    Clearly, if $\Phi \colon I_G(A) \crossed G \to I_G(A)$ is a $G$-equivariant conditional expectation, then its restriction to $A \crossed G$ is a pseudo-expectation.  Conversely, suppose that $\phi \colon A \crossed G \to I_G(A)$ is a pseudo-expectation for $(A, G,\alpha)$. Then by $G$-injectivity it extends to a $G$-morphism $\Phi \colon I_G(A) \crossed G \to I_G(A)$ with $\Phi\vert_A = \id_A$. By $G$-rigidity of the inclusion $A\subseteq I_G(A)$ it follows that $\Phi\vert_{I_G(A)}=\id_{I_G(A)}$, and so $\Phi$ is indeed a $G$-equivariant conditional expectation. Uniqueness of the extension~$\Phi$ follows because every $G$-equivariant conditional expectation $\psi\colon I_G(A) \crossed G \to I_G(A)$ contains $I_G(A)$ in its multiplicative domain, and hence is determined by the values $\{\psi(\nu_g)\}_{g\in G}$, where $g\mapsto \nu_g$ is the canonical partial ${}^*$-representation of~$G$ in $A\crossed G$. 
\end{proof}

Pseudo-expectations naturally give rise to ideals in the reduced crossed product.

\begin{lem}\label{lem: faithful kernel of pseudo-exp}
    Let $(A, G, \alpha)$ be a unital partial $\Cast$-dynamical system and $\phi \colon A \crossed G \to I_G(A)$ a pseudo-expectation. Define 
    \begin{equation*}
        J_\phi := \{x \in A \crossed G \, \vert \, \phi(x^*x) = 0\}.
    \end{equation*}
    Then $  J_\phi $ is an ideal in $A\crossed G$ with $J_\phi \cap A = \trivial$. 
\end{lem}

\begin{proof} The proof follows as in \cite[Lemma~6.5]{KenScha19} with the partial ${}^*$-representation $g\mapsto\nu_g$ of~$G$ in $A\crossed G$ playing the role of the unitary representation in the global setting.
\end{proof}

The next corollary characterizes the ideal intersection property in terms of pseudo-expectations.

\begin{cor}\label{cor: pseudo-and-IIP} Let $(A, G, \alpha)$ be a unital partial $\Cast$-dynamical system. Then $(A, G, \alpha)$ has the ideal intersection property if and only if every pseudo-expectation for $(A, G, \alpha)$ is faithful.
\begin{proof} The proof is similar to that of \cite[Theorem~6.6]{KenScha19}. The forward implication follows from Lemma~\ref{lem: faithful kernel of pseudo-exp}. Suppose that every pseudo-expectation for $(A,G,\alpha)$ is faithful and let $J\ideal A\crossed G$ with $J\cap A=\trivial$. Equip the quotient $A\crossed G/J$ with the natural generalized unital partial action induced by the partial ${}^*$-representation $g\mapsto \pi(\nu_g)$. Then the quotient map $\pi\colon A\crossed G\to A\crossed G/J$ is a $G$-morphism that restricts to a $G$-embedding of~$A$, and hence by $G$-injectivity we can find a $G$-morphism $\psi\colon A\crossed G/J\to I_G(A) $ such that $\psi\circ \pi\vert_A=\id_A$. The composition $\phi:=\psi\circ \pi$ is a pseudo-expectation for $(A, G,\alpha)$, and so the assumption gives that $J_\phi=\trivial$. Since $J\subseteq J_\phi$ we deduce that $J=\trivial$ as needed. 
    \end{proof}
    \end{cor}
 
In the next theorem we extend a result of Bryder \cite[Theorem~3.2]{Bryd22} for $\Cast$-dynamical systems to the setting of unital partial actions. We will use this later in \cref{sec: non-unital} to establish the analogous statement for arbitrary partial $\Cast$-dynamical systems.
\begin{thm}\label{thm: IIP-equiv-unital-pa}  Let $(A, G,\alpha)$ be a unital partial $\Cast$-dynamical system. Then the following are equivalent:
\begin{enumerate}
    \item[\rm{(1)}] $(A, G,\alpha)$ has the ideal intersection property;
\item[\rm{(2)}] every unital partial $\Cast$-dynamical system $(B, G,\beta)$ with $$(A, G,\alpha)\subseteq (B,G,\beta)\subseteq (I_G(A),G,I_G(\alpha))$$ has the ideal intersection property;

\item[\rm{(3)}] the unital partial $\Cast$-dynamical system $(I_G(A),G,I_G(\alpha))$ has the ideal intersection property.
\end{enumerate}
\begin{proof} We will prove (1)$\Rightarrow$(3)$\Rightarrow$(2)$\Rightarrow$(1). Suppose that $(A, G,\alpha)$ has the ideal intersection property and let $J\ideal I_G(A)\crossed G$ with $J\cap I_G(A)=\trivial$. Let $\pi\colon I_G(A)\crossed G\to I_G(A)\crossed G/ J$ be the quotient map and equip $I_G(A)\crossed G/ J$ with the generalized unital partial action induced by the partial ${}^*$-representation $g\mapsto \pi(\nu_g)$, so that $\pi$ is a $G$-morphism. Since $A\cap J=\trivial$, the assumption yields that~$\pi$ restricts to a faithful ${}^*$-homomorphism from $A\crossed G$ into $I_G(A)\crossed G/ J$. Hence the canonical conditional expectation on $A\crossed G$ induces a $G$-morphism $E_{\pi(A)}\colon \pi(A\crossed G)\to I_G(A)$ satisfying for $E_{\pi(A)}(\pi(a\nu_g))=\delta_{e,g} a\nu_g$ for $g\in G$ and $a\in A_g$, where $\delta_{e,g}$ is the Kronecker delta. In particular, $E_{\pi(A)}(\nu_g)=0$ if $g\neq e$. By $G$-injectivity this extends to a $G$-morphism $\hat{E}_{\pi(A)}\colon I_G(A)\crossed G/ J\to I_G(A)$. We must have that $\hat{E}_{\pi(A)}\circ \pi\vert_{I_G(A)}=\id_{I_G(A)}$ by $G$-rigidity, and hence a multiplicative domain argument using that $A\subseteq I_G(A)$ is $G$-unital shows that $\hat{E}_{\pi(A)}\circ \pi$ coincides with the canonical conditional expectation $E_{I_G(A)}\colon I_G(A)\crossed G\to I_G(A)$. Since this is faithful we must have $J=\trivial$. 

For the implication (3)$\Rightarrow$(2), our argument is now similar to that in the proof of \cite[Theorem~3.2]{Bryd22}, except that we use kernels of pseudo-expectations, see \cite[Theorem~6.1]{Ken+21}. Let $(B, G,\beta)$ be a unital partial $\Cast$-dynamical system with $G$-embeddings of $\Cast$-algebras $A\subseteq B\subseteq I_G(A)$. Then the inclusion $(B, G,\beta)\subseteq (I_G(A),G, \alpha)$ is essential, which implies that $I_G(A)=I_G(B)$. Let $J\ideal B\crossed G$ with $J\cap B=\trivial$. Reasoning as in the proof of Corollary~\ref{cor: pseudo-and-IIP} we see that the quotient map $\pi\colon B\crossed G\to B\crossed G/J$ induces a pseudo-expectation $\phi\colon B\crossed G\to I_G(A)$ for $(B, G,\beta)$ such that $J\subseteq J_\phi$. Let $\Phi\colon I_G(A)\crossed G\to I_G(A)$ be the unique $G$-equivariant conditional expectation extending~$\phi$. Then $J_{\Phi}$ is an ideal of $I_G(A)\crossed G$ with $J_\Phi\cap I_G(A)=\trivial$. The assumption that $(I_G(A),G,I_G(\alpha))$ has the ideal intersection property implies that $J_{\Phi}=\trivial$, and hence $J=\trivial$ since $J\subseteq J_{\Phi}$. The implication (2)$\Rightarrow$(1) is clear. 
\end{proof}
    \end{thm}

\section{Injective envelopes and enveloping actions}\label{sec: enveloping action}

In this section, we establish a natural relationship between injective envelopes of unital partial $\Cast$-dynamical systems and of their enveloping actions. For a detailed account on enveloping actions of partial actions, we refer to the original work of Abadie~\cite{Aba03}.

We recall the definition of an enveloping action of a partial action in the category of $\Cast$-algebras, see \cite[Definition~2.3]{Aba03}. 

\begin{defi}\label{def: enveloping action}
    Let $(A, G, \alpha)$ be a partial $\Cast$-dynamical system. We call  $(\cC^A, \zeta^A)$ an \emph{enveloping action} for $\alpha=\partacg{A}{\alpha}$ if
    $(\cC^A, G, \zeta^A)$ is a global $\Cast$-dynamical system, $A\ideal \cC_A$, $\alpha$ arises as the restriction of~$\zeta^A$ to~$A$, and $\lincomb\{\zeta_{t}^A(x) \, \vert \, x\in A, t\in G\}$ is dense in $\cC^A$.
\end{defi}
The enveloping action is unique up to canonical ${}^*$-isomorphism \cite[Theorem 2.1]{Aba03}. By \cite[Corollary 4.8]{Fer18} a partial action on a unital $\Cast$-algebra has an enveloping action if and only if it is a unital partial action. In fact if $\alpha=\partacg{A}{\alpha}$ is a unital partial action on $A$, its enveloping action is precisely the $\Cast$-algebra spanned by the orbits of $\iota(A)$ under the left translation action on $\ell^\infty(G,A)$, where $\iota\colon A\to p\ell^\infty(G,A)p$ is the canonical $G$-embedding from Proposition~\ref{prop: G-embed into B}. More precisely, here $\iota(a)(t)=\alpha_{t^{-1}}(p_ta)$ for $a\in A$, $t\in G$, the $\Cast$-algebra $\cC^A$ is $$\cC^A=\overline{\operatorname{span}}\{\tau_g(\iota(a)): a\in A, g\in G\},$$ and the action~$\zeta^A$ on $\cC^A$ is simply the restriction of the left translation action~$\tau$ on $\ell^\infty(G,A)$. Then $\iota(A)$ is an ideal in $\cC^A$, and $\alpha=\partacg{A}{\alpha}$ arises as the restriction of~$\zeta^A$ to the copy of~$A$.  See the proof of \cite[Theorem 4.5]{DokEx05} for further details.

Let $\alpha=\partacg{A}{\alpha}$ be a unital partial action with enveloping action~$(\cC^A, \zeta^A)$ and let $1_A$ denote the unit of~$A$. Since $A = \cC^A 1_A\ideal \cC^A $ we have $I(A) = I(\cC^A)1_A$. Recall that the $\Cast$-dynamical system $(\cC^A, \zeta^A)$ extends uniquely to a $\Cast$-dynamical system $(I(\cC^A), I(\zeta^A))$. Its restriction to $I(\cC^A)1_A$ is a partial action extending~$\alpha$, and hence by uniqueness it coincides with the extension $I(\alpha)=\partacg{I(A)}{I(\alpha)}$ of $\alpha$ as in Proposition~\ref{prop: partac on I(A)}. The next theorem shows that the $G$-injective envelopes of $A$ and of $\cC^A$ exhibit a similar relationship.

\begin{thm}\label{thm: corner of enveloping action}
    Let $(A, G, \alpha)$ be a unital partial $\Cast$-dynamical system with corresponding family of central projections $\{p_g\}_{g\in G}$ and let $(\cC^A, \zeta^A)$ be the enveloping action of~$\alpha$. Let $(I_G(\cC^A), \allowbreak G, I_G(\zeta^A))$ be the injective envelope of $(\cC^A, G, \zeta^A)$. Then $I_G(A) = I_G(\cC^A)1_A $ with the partial action given by the restriction of $(I_G(\cC^A), I_G(\zeta^A))$.
\end{thm}

\begin{proof} Set $B:=I_G(\cC^A)1_A$ and let $\beta$ be the partial action given by the restriction of $I_G(\zeta^A)$ to $B$. Let $(B, G,\beta)$ be the corresponding unital partial $\Cast$-dynamical system. Since $A\ideal \cC^A$ and $\cC^A \subseteq I_G(\cC^A)$, we obtain an inclusion of $\Cast$-algebras $A\subseteq B$. To see that this embedding is $G$-unital notice that $I_G(\zeta^A)$ extends~$\zeta^A$, and the partial action on~$A$ is the restriction of~$\zeta^A$. Hence the unit of $B_g$ is $\zeta^A_g(1_A)1_A=p_g,$ showing that the inclusion $A\subseteq B$ is $G$-unital. A similar reasoning gives $G$-equivariance.  
Since $I_G(\cC^A)$ is $G$-injective in the category of $\Cast$-dynamical systems and $B$ is the corner determined by a central projection in~$I_G(\cC^A)$, Corollary~\ref{cor: restriction is partially G-inj} yields that $(B, G,\beta)$ is injective in the category of generalized unital partial $\Cast$-dynamical systems. Hence from above it is an injective extension of $(A, G, \alpha)$. We will prove that this extension is also rigid. 

   Let $\varphi \colon B \to B$ be a $G$-morphism with $\varphi\vert_A = \id_A$. Consider the ucp map $\Tilde{\varphi} \colon \ell^\infty(G,B) \allowbreak\to \ell^\infty(G,B)$ defined by $\Tilde{\varphi}(f)(t) := \varphi(f(t))$ for $f\in \ell^\infty(G,B)$ and $t\in G$. Then $\tilde{\varphi}$ is a $G$-morphism since for $f\in \ell^\infty(G,B)$ and $t\in G$ we have \begin{align*}
    \Tilde{\varphi}(\tau_g(f))(t) = \varphi(f(g^\inv t)) = \Tilde{\varphi}(f)(g^\inv t) = \tau_g(\Tilde{\varphi}(f))(t).
    \end{align*}  Let $\iota \colon B \to p\ell^\infty(G,B)p$ be the canonical $G$-embedding. So $\iota(b)(t) = \beta_{t^\inv}(b p_{t})$ for $b\in B$ and $t\in G$, and we see that $\Tilde{\varphi}\circ\iota =\iota\circ\varphi$ because $\varphi$ is a $G$-morphism. In particular, $\Tilde{\varphi}\vert_{\iota(A)} = \id_{\iota(A)}$. Combining this with $G$-equivariance we deduce that $\Tilde{\varphi}$ restricts to the identity map also on the $\Cast$-algebra $\operatorname{Orb}_G(\iota(A))$ spanned by the orbits of~$\iota(A)$ under the left translation action. 

  Now let $\operatorname{Orb}_G(B)$ and $\operatorname{Orb}_G(\iota(B))$ be the $\Cast$-algebras spanned by the orbits of~$B$ in $I_G(\cC^A)$ and of $\iota(B)$ in $\ell^\infty(G, B)$, respectively. By uniqueness of the enveloping action of~$\beta$, there exists a unique ${}^*$-isomorphism $\phi\colon\operatorname{Orb}_G(B)\to\operatorname{Orb}_G(\iota(B))$ satisfying $\phi\vert_B=\iota$ and $\phi\circ I_G(\zeta^A)_g=\tau_g\circ\phi$ for all $g\in G$. This extends to a $G$-equivariant ucp map $\hat{\phi}\colon I_G(\cC^A)\to \ell^\infty(G,B)$ because $\ell^\infty(G,B)$ is $G$-injective. By $G$-essentiality $\hat{\phi}$ is completely isometric since $\hat{\phi}\vert_{\cC^A}$ is. Hence applying $G$-injectivity of $I_G(\cC^A)$ we obtain a $G$-morphism $\psi \colon \ell^\infty(G,B) \to I_G(\cC^A)$ with $\psi\circ\hat{\phi}=\id_{I_G(\cC^A)}.$ In particular, we have $\psi\circ\iota=\id_B$.

It follows from above that the composition $\psi\circ\Tilde{\varphi}\circ\hat{\phi}$ is also a $G$-morphism on~$ I_G(\cC^A)$ satisfying $\psi\circ\Tilde{\varphi}\circ\hat{\phi}\vert_{\cC^A}=\id_{\cC^A}$. We deduce that $\psi\circ\Tilde{\varphi}\circ\hat{\phi}=\id_{I_G(\cC^A)}$ by $G$-rigidity of the inclusion $\cC^A\subseteq I_G(\cC^A)$. Then for $b\in B$ we obtain $$\varphi(b)=\psi(\iota(\varphi(b)))=\psi(\Tilde{\varphi}(\iota(b)))=b,$$ proving that $\varphi=\id_B$ as needed.
\end{proof}

Theorem~\ref{thm: corner of enveloping action} can be very helpful to deduce important results for injective envelopes of partial $\Cast$-dynamical systems from the machinery developed for global actions. We illustrate this in the next corollary with a partial generalization of \cite[Theorem 3.4]{Ham85}. We will establish the full statement of \cite[Theorem 3.4]{Ham85} for arbitrary partial $\Cast$-dynamical systems in Theorem~\ref{thm: G-essential crossed prod in I(A cross G)-non-unital}.

\begin{cor}\label{cor: G-essential crossed prod in I(A cross G)}
 Let $(A,G,\alpha)$ and $(B, G,\beta)$ be unital partial $\Cast$-dy\-na\-mi\-cal  systems with a $G$-embedding of $\Cast$-algebras $A\subseteq B$. Suppose that $(A,G,\alpha)\subseteq (B, G,\beta)\subseteq(I_G(A),G,I_G(\alpha))$. Then $A\crossed G \subseteq B\crossed G \subseteq I(A\crossed  G)$. In particular, $I(A\crossed G)=I(I(A)\crossed G)=I(I_G(A)\crossed G)$.
\end{cor}

\begin{proof}
Suppose that $A \subseteq B \subseteq I_G(A)$. Note that it suffices to show that $I_G(A) \crossed G \subseteq I(A\crossed G)$ as then $A \crossed G \subseteq B\crossed G \subseteq I_G(A)\crossed G \subseteq I(A\crossed G)$. Let $(\cC^A, \zeta^A)$ be the enveloping action of $\alpha=\partacg{A}{\alpha}$. By Theorem~\ref{thm: corner of enveloping action} we have $I_G(A) = I_G(\cC^A)1_A$ and thus $I_G(A) \crossed G = 1_A(I_G(\cC^A) \crossed G)1_A$. Using now that $I_G(\cC^A)\crossed G\subseteq I(\cC^A\crossed G)$ by \cite[Theorem 3.4]{Ham85} we obtain the inclusion $1_A(I_G(\cC^A) \crossed G)1_A \subseteq 1_A I(\cC^A \crossed G)1_A = I(1_A (\cC^A \crossed G)1_A) = I(A\crossed G)$, giving $I_G(A) \crossed G \subseteq I(A\crossed G)$ as needed.
\end{proof}

We will end this section with the equivalence between the ideal intersection property for a unital partial action and for its enveloping action. We believe this is known since the corresponding reduced crossed products are Morita equivalent by a theorem of Abadie~\cite[Theorem 3.3]{Aba03}. Since we have not found the explicit statement in the literature we include the argument below, which utilizes the bijective correspondence between the lattice of ideals for Morita equivalent $\Cast$-algebras.

\begin{prop}\label{prop: enveloping IIP iff partac IIP}
    Let $\alpha=\partacg{A}{\alpha}$ be a unital partial action and $(\cC^A, \zeta^A)$ its enveloping action. Then the partial $\Cast$-dynamical system $(A, G,\alpha)$ has the ideal intersection property if and only if $(\cC^A, G, \zeta^A)$ does. 
\end{prop}

\begin{proof}
By \cite[Theorem 28.8]{Exel17} the crossed product $A\crossed G$ is a full hereditary $\Cast$-subalgebra of $\cC^A\crossed G$, and so $1_A$ is a full projection in $\cC^A\crossed G$. Now let $I\ideal \cC^A \crossed G $ be an ideal.  Observe that $I\cap (A\crossed G)=1_AI1_A$ and since $1_A$ is a full projection, it follows that $I$ coincides with the ideal of $\cC^A \crossed G$ generated by $1_AI1_A$ because $$I=\langle 1_A\rangle I\langle 1_A \rangle\subseteq \langle 1_A I1_A\rangle\subseteq I,$$ where we use the notation $\langle S\rangle$ for the ideal in $\cC^A \crossed G$ generated by $S\subseteq \cC^A \crossed G $.   In particular, it follows that $I\neq \trivial$ if and only if $I\cap (A\crossed G)\neq \trivial$.

 Suppose now that $(\cC^A,\zeta^A)$ has the ideal intersection property and let $J \ideal A\crossed G$ with $J \cap A = \trivial$. Set $I := \langle J \rangle \ideal \cC^A \crossed G$. From above we have $1_AI1_A=J$. We claim that $I \cap \cC^A = \trivial$. Indeed, looking for a contradiction suppose $0 \neq a \in I \cap \cC^A$. Since the orbits of $A$ span a dense subalgebra of~$\cC^A$ we can find $g\in G$ such that~$a\zeta^A_g(1_A)\neq 0$. Since $I$ is an ideal, it is invariant under the canonical inner action of~$G$ on $ \cC^A \crossed G$ extending $\zeta^A$, and hence we may in fact assume that $a\in I\cap  \cC^A $ satisfies $a1_A\neq 0$. Then $0\neq 1_Aa^*a1_A\in 1_AI1_A\cap A=J\cap A$, a contradiction. Thus we must have $I\cap \cC^A=\trivial $. By assumption $I=\trivial$, giving $J=1_AI1_A=\trivial$ as needed.%

Conversely, suppose that $(A, G,\alpha)$ has the ideal intersection property and let $I\ideal \cC^A \crossed G$ be such that $I \cap \cC^A = \trivial$. Then  $J:=1_AI1_A$ satisfies $J\cap A=\trivial$. By assumption this implies $J=\trivial$, and hence $I=\trivial$.
\end{proof}

\section{Injective envelopes for non-unital partial actions}\label{sec: non-unital}

In this section we introduce injective envelopes for arbitrary partial $\Cast$-dynamical systems. The main idea is to consider an \emph{essential unitization} of a partial action $\alpha=\partacg{A}{\alpha}$, meaning a unital partial $\Cast$-dynamical system $(B, G,\beta)$ with $G$-equivariant inclusions of $\Cast$-algebras $A\subseteq B\subseteq I(A)$. In particular, we define the $G$-injective envelope $I_G(A)$ of~$A$ to be the $G$-injective envelope of $I(A)$ in the category of generalized unital partial $\Cast$-dynamical systems, and establish several desirable properties for the inclusion $A\subseteq I_G(A)$ in analogy to the unital setting.

\subsection{Injective envelopes and the ideal intersection property}

Recall from Proposition~\ref{prop: partac on I(A)} that a not necessarily unital partial action $\alpha=\partacg{A}{\alpha}$ extends uniquely to a unital partial action $I(\alpha)=\partacg{I(A)}{I(\alpha)}$ on $I(A)$ such that $I(A)_g=I(A_g)$ and the inclusion $A\subseteq I(A)$ is $G$-equivariant.
\begin{defi}\label{def: G-envelope-non-unital}
    Let $(A, G,\alpha)$ be a partial $\Cast$-dynamical system and $I(\alpha)$ the corresponding unital partial action on the injective envelope of~$A$. We define the \emph{injective envelope} of $(A, G, \alpha)$ to be the injective envelope of $(I(A), G, I(\alpha))$ in the category of generalized unital partial $\Cast$-dynamical systems, i.e. $I_G(A) := I_G(I(A))$ and $I_G(\alpha):=I_G(I(\alpha))$.
\end{defi}

\begin{rem} It follows from Lemma~\ref{lem:isomorphism-envelopes-AB} that every unital partial $\Cast$-dynamical system $(B, G, \beta)$ with $(A, G, \alpha)\subseteq (B, G, \beta)\subseteq (I(A), G, I(\alpha)) $ satisfies $I_G(B)=I_G(I(A))$. Hence there is nothing special in our choice of the essential unitization $I(\alpha)=\partacg{I(A)}{I(\alpha)}$ in Definition~\ref{def: G-envelope-non-unital}.
    \end{rem}
    
 Our next goal is to show that for an arbitrary partial action $\alpha$, the partial $\Cast$-dynamical system $(A, G, \alpha)$ has the ideal intersection property if and only if $(I_G(A), G, I_G(\alpha))$ does so. The next proposition combined with Theorem~\ref{thm: IIP-equiv-unital-pa} gives the forward implication.

\begin{prop}\label{prop: A IIP implies I(A) IIP}
    Let $(A, G, \alpha)$ be a partial $\Cast$-dynamical system. If $(A, G, \alpha)$ has the ideal intersection property, then so does $(I(A), G, I(\alpha))$.
\end{prop}

\begin{proof} The argument is similar to that in the proof of (1)$\Rightarrow$(3) in Theorem~\ref{thm: IIP-equiv-unital-pa}, except that $A\crossed G$ does not necessarily contain the family of partial isometries $\{\nu_g\}_{g\in G}$ in $I(A)\crossed G$. Suppose $(A, G, \alpha)$ has the ideal intersection property and take $J \ideal I(A) \crossed G$ with $I(A) \cap J = \trivial$. Then $J \cap A = \trivial$ and thus the assumption yields $J \cap (A \crossed G) = \trivial$. Hence the quotient map $\pi\colon I(A)\crossed G\to I(A)\crossed G/J$ restricts to a faithful ${}^*$-homomorphism from $A\crossed G$ to $I(A)\crossed G/J$.

It follows that there is a conditional expectation $E_{\pi(A)}\colon \pi(A\crossed G)\to A$ satisfying $E_{\pi(A)}(\pi(a\nu_g))=\delta_{e,g}a$ for $g\in G$ and $a\in A$. Since $I(A)$ is injective, there is a ucp map $\hat{E}_{\pi(A)}\colon I(A)\crossed G/J\to I(A)$ extending~$E_{\pi(A)}$. We will show that the composition $\phi:=\hat{E}_{\pi(A)}\circ\pi$ coincides with the canonical conditional expectation $E_{I(A)}\colon I(A)\crossed G\to I(A)$.

First, notice that $\phi\vert_{I(A)}=\id_{I(A)}$ by rigidity of the inclusion $A\subseteq I(A)$. In particular, $I(A)$ is contained in the multiplicative domain of $\phi$. Thus $\phi (\nu_g) = p_g \phi(\nu_g)$ for all $g\in G$, where $p_g$ is the unit of~$I(A)_g$. Let $(u_\mu)_\mu$ be an increasing approximate identity for~$A_g$. Then $\sup_\mu u_\mu = p_g$ in $I(A)$ and we obtain 
    \begin{align*}
        \phi (\nu_g)^* \phi(\nu_g) = \phi(\nu_g)^* (\sup_\mu u_\mu) \phi (\nu_g) =  \sup_\mu \big(\phi(\nu_g)^* u_\mu \phi(\nu_g)\big) =  \sup_\mu \big(\phi (\nu_g)^* \phi(u_\mu\nu_g)\big).
    \end{align*}
 This shows that $\phi(\nu_g) = \delta_{e,g}$ for all $g\in G$ because $\phi(u_\mu \nu_g)=0$ for every~$\mu$ if $g\neq e$. Thus $\phi$ is the canonical conditional expectation, which implies that $J=\trivial$. 
\end{proof}

Recall that if $(A, G, \alpha)$ is a $\Cast$-dynamical system and $A$ is non-unital, then the action on~$A$ extends uniquely to an action on the minimal unitization~$\widetilde{A}$ of~$A$, and $A\crossed G$ is an essential ideal of $\widetilde{A}\crossed G$, see \cite[Lemma 7.1]{GefUrs23}. In particular, $A\subseteq A\crossed G$ has the ideal intersection property if and only if $\widetilde{A}\subseteq \widetilde{A}\crossed G$ does so. It follows from results in the unital setting that the ideal intersection property for $(A, G, \alpha)$ is equivalent to that of $(I(A), G, I(\alpha))$, and hence also of $(I_G(A), G, I_G(\alpha))$, see \cite[Theorem~3.2]{Bryd22}.

For a non-unital partial action $\alpha=\partacg{A}{\alpha}$ the situation is more delicate, and it will take us a considerable effort to show that the ideal intersection property for $I(A)\subseteq I(A)\crossed G$ indeed implies the ideal intersection property for $A\subseteq A\crossed G$.

It will be crucial for us to work with unitizations of $\alpha=\partacg{A}{\alpha}$ that are small, in the sense that the underlying $\Cast$-algebra is generated by orbits of~$A$ (or of $\widetilde{A}$ when $A$ is non-unital) as a $\Cast$-algebra. The following lemma clarifies what we mean by orbits in this setting.

\begin{lem}\label{lem:partial-action-on-orbits} Let $A$ be a unital $\Cast$-algebra and $\alpha=\partacg{A}{\alpha}$ a partial action on~$A$. Suppose that $(C, G, \gamma)$ is a unital partial $\Cast$-dynamical system with corresponding family of central projections $\{p_g\}_{g\in G}$. Let $\pi\colon A\to C$ be a $G$-equivariant ${}^*$-homomorphism and set $$\operatorname{Orb}^\gamma_G(\pi(A)):=\Cast(\{\gamma_g(\pi(a)p_{g^{-1}})\,\vert\,a\in A, g\in G\}).$$ Then $\operatorname{Orb}^\gamma_G(\pi(A))$ is $\gamma$-invariant and hence $\gamma=\partacg{C}{\gamma}$ restricts to a unital partial action on $\operatorname{Orb}^\gamma_G(\pi(A))$ such that $\pi\colon A\to \operatorname{Orb}^\gamma_G(\pi(A)) $ is $G$-equivariant.

\begin{proof} Observe that $ \operatorname{Orb}^\gamma_G(\pi(A))\cap C_g=p_g\operatorname{Orb}^\gamma_G(\pi(A))$, and hence all we need to show is that $\gamma_g(\operatorname{Orb}^\gamma_G(\pi(A))p_{g^{-1}})\subseteq \operatorname{Orb}^\gamma_G(\pi(A))p_g$ for all $g\in G$.

Let $h\in G$ and $a\in A.$ For $g\in G$ we compute \begin{align*}
    \gamma_g(\gamma_h(\pi(a)p_{h^\inv})p_{g^\inv}) &= \gamma_g(\gamma_h(\pi(a)p_{h^\inv})p_h p_{g^\inv}) p_g \\ &= \gamma_g(\gamma_h(\pi(a)p_{h^\inv})\gamma_h (p_{h^\inv} p_{h^\inv g^\inv}) ) p_g \\& = \gamma_{gh} (\pi(a) p_{h^\inv} p_{(gh)^\inv}) p_g \\&= \gamma_{gh} (\pi(a)p_{(gh)^\inv})\gamma_{gh}( p_{h^\inv} p_{(gh)^\inv}) p_g \\&=\gamma_{gh}(\pi(a)p_{(gh)^\inv})p_g\in \operatorname{Orb}_G(\pi(A))p_g.
    \end{align*}
Since $\operatorname{Orb}^\gamma_G(\pi(A))$ is the closed linear span of products of this form and the $p_g$s are central projections in~$C$, we conclude that $\operatorname{Orb}^\gamma_G(\pi(A))$ is $\gamma$-invariant as needed.
\end{proof}
    \end{lem}
    
We will require the following lemma.

\begin{lem}\label{lem: Mult map for inclusion}
   Let $A$ be a unital $\Cast$-algebra and $\alpha=\partacg{A}{\alpha}$ a partial action on~$A$. Suppose that $(C, G, \gamma)$ is a unital partial $\Cast$-dynamical system with corresponding family of central projections $\{p_g\}_{g\in G}$ and a $G$-equivariant inclusion of $\Cast$-algebras $A\subseteq C$. Assume that $C=\operatorname{Orb}^\gamma_G(A)$ and let $\Pi \colon (C\crossed G)^\ddual \to B(\hilb)$ be a ucp map that is multiplicative on $A\crossed G$. Let $g\mapsto \nu_g$ be the canonical partial ${}^*$-representation of~$G$ in~$C\crossed G$ implementing~$\gamma$ and suppose that $\{q_g \nu_g \vert g\in G\}$ is in the multiplicative domain of~$\Pi$, where $q_g$ denotes the unit of $A_g^\ddual$. If $\Pi(p_g) = \Pi(q_g)$ for all $g\in G$, then~$\Pi$ is multiplicative on~$C\crossed G$. 
\end{lem}

\begin{proof} We begin by observing that $q_g\leq p_g$ for all $g\in G$ since $A_g\subseteq Cp_g$, and in particular the partial isometries $\{\nu_g\}_{g\in G}$ implement the partial action on $A^\ddual$. Now to see that~$\Pi$ is multiplicative on $C\crossed G$, observe first that the assumption $C=\operatorname{Orb}^\gamma_G(\pi(A))$ implies that $C\crossed G$ is generated as a $\Cast$-algebra by~$A$ and the family of partial isometries $\{\nu_g\}_{g\in G}$. Since $A$ is contained in $\dom_\Pi$ by assumption and $\dom_\Pi$ is a $\Cast$-algebra, all we need to show to establish the lemma is that~$\nu_g$ belongs to~$\dom_\Pi$ for all $g\in G$. 
    
  By assumption the partial isometry $q_g \nu_g$ belongs to $\dom_\Pi$. Also, the assumption that $\Pi(p_g) = \Pi(q_g)$ yields that $p_g\in \dom_\Pi$. In particular $\Pi(\nu_g)=\Pi(p_g)\Pi(\nu_g)=\Pi(q_g)\Pi(\nu_g)=\Pi(q_g\nu_g)$ and hence
 $$\Pi(\nu_g)^*\Pi(\nu_g)=\Pi(q_g\nu_g)^*\Pi(q_g\nu_g)=\Pi(q_{g^{-1}})=\Pi(p_{g^{-1}})=\Pi(\nu_g^*\nu_g).$$
   Applying this to $g^\inv$ yields $\Pi(\nu_g)\Pi(\nu_g)^*=\Pi(p_g)=\Pi(\nu_g\nu_g^*)$, and it follows that $\nu_g\in \dom_\Pi$ as required.
\end{proof}

The next proposition will be the main ingredient in the proof that the ideal intersection property for $(I_G(A), G, I_G(\alpha))$ implies the ideal intersection property for $(A, G, \alpha)$ also for non-unital partial actions.

\begin{prop}\label{pro: IIP-equiv. for unitization} Let $A$ be a unital $\Cast$-algebra and $\alpha=\partacg{A}{\alpha}$ a partial action on~$A$. Consider the canonical inclusion $A\subseteq I(A)$ and set $C:=\operatorname{Orb}^{I(\alpha)}_G(A)$. Let $\gamma=\partacg{C}{\gamma}$ be the unital partial action on~$C$ as in Lemma~\ref{lem:partial-action-on-orbits}. Let $\pi\colon C\crossed G\to B(\hilb)$ be a ucp map that is multiplicative on~$A\crossed G$ and suppose that $\pi\vert_A$ is faithful. Then there exist a unital $\Cast$-algebra~$D$ and a ucp map $\phi\colon \operatorname{im}\pi\to D$ such that $\phi\circ \pi\colon C\crossed G\to D$ is a ${}^*$-homomorphism and the restriction~$\phi\circ \pi\vert_{C}$ is faithful.
\begin{proof} For each $g\in G$ let $p_g$ be the unit of $I(A_g)=I(A)_g$, so that $C_g=Cp_g$. Let~$\Pi:= \pi^\ddual$ be the unique normal extension $\Pi \colon (C\crossed G)^\ddual \to B(\hilb)$ of~$\pi$. Because $\pi$ is multiplicative on~$A\crossed G$ it follows that $(A\crossed G)^\ddual$ is contained in $\dom_\Pi$.

 Let $g\mapsto \nu_g$ be the canonical partial ${}^*$-representation of $G$ in $C\crossed G$, so that $v_gv_{g^{-1}}=p_g$ for all $g\in G$, and for each $g\in G$ let $q_g \in A_g^\ddual \subseteq (A\crossed G)^\ddual$ be the unit of~$A_g^\ddual$. Notice that $\Pi(q_gv_g)$ implements the partial action on $\pi(A)$, and hence also on $\Pi(A^\ddual)=\pi^\ddual(A^\ddual)$. Let $\operatorname{Orb}^{\alpha^\ddual}_G(A)$ be the $\Cast$-subalgebra of $A^\ddual$ generated by the orbits of~$A$ in~$A^\ddual$ as in Lemma~\ref{lem:partial-action-on-orbits}. Observe that $$\operatorname{Orb}^{\alpha^\ddual}_G(A)= \Cast(\{q_g\nu_g a (q_g\nu_g)^* \,\vert\, a\in A, g\in G\})\subseteq (C\crossed G)^\ddual.$$ 
 
 Consider the $\Cast$-subalgebra~$\mathcal{B}$ of $\Cast(\Pi)$ defined by $$\mathcal{B} := \Cast\big(\operatorname{im}\pi \cup\{\Pi(b) \Pi(q_g \nu_g) \vert b\in \operatorname{Orb}^{\alpha^\ddual}_G(A), g\in G\}\big).$$ Since $\pi\vert_A$ is faithful and $A\subseteq C\subseteq I(A)$, the restriction $\pi\vert_C$ is completely isometric and thus there exists a ucp map $\rho\colon B(\hilb)\to I(A)$ such that $\rho\circ\pi\vert_C=\id_C$. Define an ideal $I\ideal \mathcal{B}$ by 
    \begin{equation*}
        I := \{d\in \mathcal{B} \,\vert \,\rho(b^* d^*d b) = 0 \text{ for all } b\in \mathcal{B}\},
    \end{equation*} and notice that $I\cap \pi(A)=\trivial$. Let $\psi \colon \mathcal{B} \to \mathcal{B}/I$ be the quotient map. We will prove that $\psi \circ \pi \colon C \crossed G \to \mathcal{B}/I$ is a ${}^*$-homomorphism. To do this, by Lemma~\ref{lem: Mult map for inclusion} it suffices to show that  $(\psi \circ \Pi)(p_g) = (\psi \circ \Pi)(q_g)$ for all $g\in G$ or, equivalently, that $\Pi(p_g) - \Pi(q_g)\in I$ for all $g\in G$.

   We begin by proving a couple of claims.
    
    \vspace{0.5em}
    \textit{\underline{Claim 1:} For all $a \in A$ and $g\in G$ we have $\rho(\Pi(\nu_g q_g a (q_g \nu_g)^*)) = I(\alpha)_g(a p_{g^\inv})$. In particular, $\Pi(\operatorname{Orb}^{\alpha^\ddual}_G(A))$ is contained in the multiplicative domain of~$\rho$.}

    \begin{proof}[Proof of Claim 1]\noindent To see the first assertion assume $a = c^* c$ for some $c\in A$ and let $(u_\mu)_\mu \subseteq A_{g^\inv}$ be an increasing approximate identity for $A_{g^\inv}$. Then
    \begin{align*}
        \Pi (I(\alpha)_g(c^*c p_{g^\inv})) = \Pi (I(\alpha)_g(c^* p_{g^\inv} c)) &\geq \Pi (I(\alpha)_g(c^* u_\mu c)) \\
        &= \pi (\alpha_g(c^* u_\mu c)) = \Pi ( q_g \nu_g c^* u_\mu c (q_g \nu_g)^*)
    \end{align*}
    for all $\mu$ since $q_g \nu_g$ implements the partial action on~$A_g$. It follows by taking the supremum over~$\mu$ that $\Pi (I(\alpha)_g(c^*c p_{g^\inv}))\geq \Pi ( q_g \nu_g c^* c (q_g \nu_g)^*)$ since $\Pi$ is normal. Applying~$\rho$ yields
    \begin{align*}
        I(\alpha)_g(c^*c p_{g^\inv})\geq  \rho(\Pi (q_g \nu_g c^* c (q_g \nu_g)^*)).
    \end{align*}
   For the reverse inequality notice that for all~$\mu$ we have
    \begin{align*}
        \rho(\Pi(q_g \nu_g c^*c (q_g \nu_g)^*)) \geq \rho(\pi(\alpha_g(c^*u_\mu c))) = \alpha_g(c^*u_\mu c).
    \end{align*}
    This gives 
    \begin{equation*}
      \rho(\Pi(q_g \nu_g c^*c (q_g \nu_g)^*))  \geq \sup_\mu I(\alpha)_g(c^*u_\mu c) = I(\alpha)_g\big(\sup_\mu (c^*u_\mu c )\big)=I(\alpha)_g(c^*cp_{g^{-1}} ),
    \end{equation*} where in the first equality above we used that, being a ${}^*$-isomorphism between monotone complete $\Cast$-algebras, $I(\alpha)_g$ is normal. This shows that $\rho(\Pi(q_g \nu_g c^*c (q_g \nu_g)^*))=I(\alpha)_g(c^*c p_{g^\inv})$ for all $c\in A$. The first assertion in the claim then follows because~$A$ is spanned by positive elements.  
    
    To see that $\Pi(\operatorname{Orb}^{\alpha^\ddual}_G(A))\subseteq \dom_\rho$ it suffices to show that $\Pi(q_g \nu_g a (q_g \nu_g)^*)$ belongs to $\dom_\rho$ for all $a\in A$ and $g\in G$ since $\Pi$ is multiplicative on~$\operatorname{Orb}^{\alpha^\ddual}_G(A)$. This follows because for  $y = \Pi(\nu_g^* q_g a q_g \nu_g)$  we have $\rho(y^*y) = \rho(\Pi(\nu_g^* q_g a^*a q_g \nu_g) = I(\alpha)_{g^\inv}(a^*ap_{g}) = \rho(y)^* \rho(y)$ and similarly $\rho(yy^*) = \rho(y)\rho(y)^*$. This completes the proof of the \emph{Claim 1}.\renewcommand{\qedsymbol}{$\diamond$}
    \end{proof}
    \vspace{0.5em}

    \textit{\underline{Claim 2:} Let  $a_1, \ldots, a_n \in A$, $t_1, \ldots, t_n\in G$ and consider the corresponding elements $b = I(\alpha)_{t_1}(a_1 p_{t_1^\inv}) \cdots I(\alpha)_{t_n}(a_n p_{t_n^\inv})$ and $b_q = q_{t_1} \nu_{t_1} a_1 (q_{t_1} \nu_{t_1})^* \cdots \allowbreak q_{t_n} \nu_{t_n} a_n (q_{t_n} \nu_{t_n})^* $ of~$C$ and of $\operatorname{Orb}^{\alpha^\ddual}_G(A)$, respectively. Then for all~$d \in \mathcal{B}$ and $g\in G$ we have that $\rho(\pi(b \nu_g)^* d^* d \allowbreak\pi(b\nu_g)) = 0$ if and only if $\rho(\Pi(b_q q_g \nu_g)^* d^* d \Pi(b_q q_g\nu_g)) = 0$.}

    \begin{proof}[Proof of Claim 2]\noindent Write $x_g := \Pi(b\nu_g) - \Pi(b_q q_g \nu_g) = \pi(b\nu_g) - \Pi (b_q q_g \nu_g)$. We will first show that $\rho(x_g^*x_g) = 0$. We have
    \begin{equation}\label{eq: mixed-terms}
    \begin{aligned}
        \rho(x_g^*x_g) &= \rho( \pi(b\nu_g)^* \pi(b\nu_g)) - \rho(\pi(b\nu_g)^*\Pi(b_q q_g \nu_g)) \\
        &\qquad - \rho(\Pi(b_q q_g \nu_g)^*\pi(b \nu_g)) + \rho(\Pi((b_q q_g \nu_g)^* b_q q_g \nu_g)) \\
        &\leq \rho(\pi((b\nu_g)^*b \nu_g)) - \rho(\Pi((b\nu_g)^* b_q q_g \nu_g))\\
        &\qquad - \rho(\Pi((b_q q_g \nu_g)^* b \nu_g)) + \rho(\Pi(\nu_{g^\inv} q_g b_q^* b_q q_g \nu_g))\\
        &= 2 I(\alpha)_{g^\inv}(b^* b p_g)  - \rho(\Pi((b\nu_g)^* b_q q_g \nu_g)) - \rho(\Pi((b_q q_g \nu_g)^* b \nu_g)),
    \end{aligned}
    \end{equation}
    where the inequality follows from the Cauchy--Schwarz inequality and the fact that $b_q q_g \nu_g \in \dom_\Pi$. The second equality follows by \emph{Claim~1} and because $\rho\circ\pi\vert_C=\id_C$. For the mixed terms after the last equality in~\eqref{eq: mixed-terms} observe first that $\pi(C)$ is contained in the multiplicative domain of~$\rho$. Now recall that $\nu_g\nu_g^* = p_g$ and $q_g \leq p_g$, and thus $q_g p_g = q_g$. Using this we compute 
    \begin{align*}
        \rho(\Pi((b\nu_g)^* b_q q_g \nu_g)) &= \rho(\Pi((b\nu_g)^* p_g q_g b_q q_g \nu_g)) = \rho(\Pi( \nu_g^* b^* \nu_g \nu_g^* q_g b_q q_g  \nu_g)) \\
        &= \rho(\pi(\nu_g^* b^* \nu_g) \Pi(\nu_g^* q_g b_q q_g  \nu_g)) = I(\alpha)_{g^\inv}(b^* p_g) \rho( \Pi(\nu_g^* q_g b_q q_g  \nu_g))\\
        &= I(\alpha)_{g^\inv}(b^* p_g) I(\alpha)_{g^\inv}(b p_g) = I(\alpha)_{g^\inv}(b^*b p_g).
    \end{align*}
 Taking adjoints also gives $\rho(\Pi((b_q q_g \nu_g)^* b \nu_g)) = I(\alpha)_{g^\inv}(b^*b p_g)$. Substituting these equalities in the last line of \eqref{eq: mixed-terms} yields
    \begin{equation*}
        0 \leq \rho(x_g^* x_g) \leq 2 I(\alpha)_{g^\inv}(b^* b p_g)  -  2 I(\alpha)_{g^\inv}(b^* b p_g) = 0,
    \end{equation*}
    showing that $\rho(x_g^* x_g)=0$ as wanted. 
    
  The above implies that $\rho(x_g^* d^* d x_g) = 0$ for all $d\in \mathcal{B}$ as $\rho(x_g^* d^* d x_g) \leq \Vert d \Vert^2 \rho(x_g^* x_g)$. Expanding this gives
    \begin{equation}\label{eq: Claim 2 trick}
    \begin{aligned}
       0 = \rho(x_g^* d^* d x_g) &= \rho(\pi(b \nu_g)^* d^*d \pi(b \nu_g)) - \rho(\pi(b\nu_g)^* d^* d \Pi(b_q q_g \nu_g))\\
        &\quad - \rho(\Pi(b_q q_g \nu_g)^* d^* d \pi(b\nu_g)) + \rho(\Pi(b_q q_g \nu_g)^* d^* d \Pi(b_q q_g \nu_g)).
    \end{aligned}
    \end{equation}
    Now suppose $\rho(\pi(b \nu_g)^* d^*d \pi(b\nu_g)) = 0$. Then since $\Vert \rho(a^*c) \Vert^2 \leq \Vert \rho(a^* a) \Vert \Vert c\Vert^2$ for all $a,c\in \mathcal{B}$ we obtain $$\rho(\pi(b\nu_g)^* d^* d \Pi(b_q q_g \nu_g)) = 0 = \rho(\Pi(b_q q_g \nu_g)^* d^* d \pi(b\nu_g))$$ by setting $a = d\pi(b\nu_g) $ and $c = d \Pi(b_q q_g \nu_g)$. It follows from \eqref{eq: Claim 2 trick} that we must have $\rho(\Pi(b_q q_g \nu_g)^* d^* d \Pi(b_q q_g \nu_g)) = 0$. The other direction follows by the exact same reasoning, proving \emph{Claim~2}.\renewcommand{\qedsymbol}{$\diamond$}
\end{proof}
    
    We proceed to showing that $\pi(p_g) - \Pi(q_g)\in I$. To lighten notation, write $r_g=\pi(p_g) - \Pi(q_g)$. By definition of $I$, we need to show that $\rho(b^* r_g^* r_g b) = 0$ for all $b\in \mathcal{B}$. By Cauchy--Schwarz and the definition of~$\mathcal{B}$ it suffices to show that $\rho( B_n^* r_g^* r_g B_n) = 0$ for all $n\geq 1$ and $B_n\in\mathcal{B}$ of the form 
    \begin{equation}\label{eq: induction defi C_n}
        B_n = \prod_{k=1}^n x_k \qquad \text{ for } x_k = \pi(b_k\nu_{h_k}) \text{ or } x_k = \Pi(b_{k,q} q_{h_k}\nu_{g_k}),\, k=1,\ldots, n,
    \end{equation} where $b_k \in C$, $b_{k,q}\in \operatorname{Orb}_G^{\alpha^\ddual}(A)$ in turn are spanning elements as in the statement of~\emph{Claim~2}. Observe in addition that since $(A\crossed G)^\ddual \subseteq \dom_\Pi$, we can assume that for each $1\leq k\leq n-1$ the consecutive factors~$x_k$ and~$x_{k+1}$ are not both of the form $\Pi(b_q\nu_{g})$ for some $g\in G$ and $b_{q}\in \operatorname{Orb}_G^{\alpha^\ddual}(A)$ since the elements of this form span a subalgebra of $(A\crossed G)^\ddual$.

    Now we will show that $\rho( B_n^* r_g^* r_g B_n) = 0$ by induction on~$n$. Suppose~$n=1$, so that $C_1 = \pi(b\nu_g)$ for $b\in C, g\in G$ or $C_1 = \Pi(b_q q_g \nu_g)$ for $b_q \in \operatorname{Orb}_G^{\alpha^\ddual}(A),g\in G$. By \emph{Claim 2} we can in fact assume the latter, i.e. $C_1=\Pi(b_q q_g \nu_g)$. Using that $b_q q_g \nu_g \subseteq \dom_\Pi$ and that $q_gp_g=q_g$ we get
    \begin{align*}
        \rho(\Pi(b_q q_g \nu_g)^*r_g^* r_g \Pi(b_q q_g \nu_g)) &= \rho(\Pi( p_g b_q q_g \nu_g)^* \Pi(p_g b_q q_g \nu_g)) \\
        &\quad \, - \rho(\Pi( p_g b_q q_g \nu_g)^* \Pi(b_q q_g \nu_g)) \\
        &\quad \, - \rho(\Pi( b_q q_g \nu_g)^* \Pi(p_g b_q q_g \nu_g)) \\
        &\quad \,+ \rho(\Pi(( b_q q_g \nu_g)^*b_q q_g \nu_g))= 0
    \end{align*} since each the summand coincides with~$I(\alpha)_g(b^*bp_g)$ or $-I(\alpha)_g(b^*bp_g)$ according to its sign. This gives the case $n=1$.
    
    Now assume $\rho(B_n^* r_g^* r_g B_n) = 0$ for all $B_n$ as in \eqref{eq: induction defi C_n} for a fixed $n\in \mathbb{N}$. Let $B_{n+1}$ be as in \eqref{eq: induction defi C_n} with $n+1$ factors. We will show $\rho(B_{n+1}^* r_g^* r_g B_{n+1}) = 0$. Again by \emph{Claim 2} we can assume that $B_{n+1} = B_n \Pi(b_{n+1, q} q_{h_{n+1}} \nu_{h_{n+1}})$. Furthermore, by the induction hypothesis and the observation below~\eqref{eq: induction defi C_n} we can assume that $x_n = \pi(b_n \nu_{h_n})$ for $b_n \in C$ a spanning element as in \emph{Claim~1} and $h_n\in G$. Appealing to the fact that $b_{n+1, q} q_{h_{n+1}} \nu_{h_{n+1}}$ is in the multiplicative domain of~$\Pi$ and using that $q_{h_{n+1}}=q_{h_{n+1}}p_{h_{n+1}}$ we obtain
    \begin{equation*}\label{eq: rewrite of C_n}
        \begin{aligned}
            B_{n+1} &= B_{n-1} x_n x_{n+1} = B_{n-1} \Pi(b_n \nu_{h_n} b_{n+1, q} q_{h_{n+1}} \nu_{h_{n+1}}) \\
            &= B_{n-1} \Pi(b_n \nu_{h_n} \nu_{h_{n+1}}\nu_{h_{n+1}}^* q_{h_{n+1}} b_{n+1, q} q_{h_{n+1}} \nu_{h_{n+1}}) \\
            &= B_{n-1} \pi(b_n p_{h_n} \nu_{h_n h_{n+1}}) \Pi(\nu_{h_{n+1}}^* q_{h_{n+1}} b_{n+1, q} q_{h_{n+1}} \nu_{h_{n+1}})\\
            &= B'_n \Pi(\nu_{h_{n+1}}^* q_{h_{n+1}} b_{n+1, q} q_{h_{n+1}} \nu_{h_{n+1}}),
        \end{aligned}
    \end{equation*}
    where $B'_n=B_{n-1}\pi(b_n p_{h_n} \nu_{h_n h_{n+1}})$. Since $\Pi(\operatorname{Orb}_G^{\alpha^\ddual}(A))\subseteq \dom_\rho$ by \emph{Claim~1} this gives
    \begin{align*}
        \rho(B_{n+1}^* r_g^* r_g B_{n+1}) &= \rho\Big(\Pi(\nu_{h_{n+1}}^* q_{h_{n+1}} b_{n+1, q} q_{h_{n+1}} \nu_{h_{n+1}})^* {B'}_n^* r_g^*\\
        &\qquad\,\, r_g B'_n \Pi(\nu_{h_{n+1}}^* q_{h_{n+1}} b_{n+1, q} q_{h_{n+1}} \nu_{h_{n+1}})\Big) \\
        &= I(\alpha)_{h_{n+1}^\inv} (b_{n+1}^* p_{h_{n+1}}) \rho({B'}_n^* r_g^* r_g B'_n )I(\alpha)_{h_{n+1}^\inv} (b_{n+1} p_{h_{n+1}}) = 0
    \end{align*}
    by the induction hypothesis. This shows that $r_g \in I$ for all $g\in G$. 
    
    The above implies that $\psi(\pi(p_g))=\psi(\Pi(q_g))$ for all $g\in G$. We deduce from Lemma~\ref{lem: Mult map for inclusion} that $\psi\circ\pi\colon C\crossed G\to \mathcal{B}/I$ is a ${}^*$-homomorphism. Setting $D:=\mathcal{B}/I$ and $\phi:=\psi\vert_{\operatorname{im}\pi}$ we see that $\phi\circ\pi\colon C\crossed G\to D$ is a ${}^*$-homomorphism as required. Finally, that $\phi\circ\pi\vert_{C}$ is faithful follows because $I\cap A=\trivial$ and $A\subseteq C\subseteq I(A)$.
\end{proof}
\end{prop}

The next lemma implies that the partial action on the $\Cast$-algebra generated by the orbits of~$A$ in $I(A)$ is in a precise sense the smallest unitization of~$\alpha=\partacg{A}{\alpha}$.

\begin{lem}\label{lem:co-universal-unitization} Let $A$ be a unital $\Cast$-algebra and $\alpha=\partacg{A}{\alpha}$ a partial action on~$A$. Let $(B, G, \beta)$ be a unital partial $\Cast$-dynamical system with corresponding family of central projections $\{r_g\}_{g\in G}$ and suppose that there exists a $G$-equivariant inclusion of $\Cast$-algebras $A\subseteq B$ such that $B=\operatorname{Orb}_G^\beta(A)$. Suppose in addition that $r_gA_g^\perp=\trivial$ for all $g\in G$, where $A_g^\perp=\{a\in A\mid aA_g=\trivial\}$ is the annihilator of~$A_g$ in~$A$. Then there exists a $G$-morphism $\Psi\colon B\to \operatorname{Orb}_G^{I(\alpha)}(A)$ such that $\Psi\vert_A=\id$. 
\begin{proof} We will follow the notation in the proof of Proposition~\ref{pro: IIP-equiv. for unitization}. Write $C=\operatorname{Orb}_G^{I(\alpha)}(A)$. Let $\varphi\colon B\crossed G\to B(\hilb)$ be a representation such that $\varphi\vert_B$ is faithful and let $\pi\colon C\crossed G \to B(\hilb)$ be a ucp map extending the restriction $\varphi\vert_{A\crossed G}$. Let $\Pi\colon (C\crossed G)^\ddual\to B(\hilb)$ be the unique normal extension of~$\pi$ and similarly consider the unique normal extension $\varphi^\ddual\colon (B\crossed G)^\ddual\to B(\hilb)$ of~$\varphi$ to~$(B\crossed G)^\ddual$. Observe that $\varphi^\ddual\vert_{(A\crossed G)^\ddual}$ coincides with~$\Pi$. 

Let $\rho\colon B(\hilb)\to I(A)$ be a ucp map such that $\rho\circ\pi\vert_C=\id_C$. As in \emph{Claim 1} in the proof of Proposition~\ref{pro: IIP-equiv. for unitization} we have that $\rho(\Pi(\nu_g q_g a (q_g \nu_g)^*)) = I(\alpha)_g(a p_{g^\inv})$ for all $a\in A$ and $g\in G$, and $\Pi(\operatorname{Orb}^{\alpha^\ddual}_G(A))$ in contained in the multiplicative domain of~$\rho$. 
We claim that $\rho(\varphi(r_g))=p_g$ for all $g\in G$. Indeed, since $A_g\subseteq B_g=r_gB$ we see that $\rho(\varphi(r_g))\geq p_g$. Since~$A$ is contained in the multiplicative domain of~$\rho\circ\varphi$ and $r_gA_g^\perp=\trivial$ by assumption, we must have $\rho(\varphi(r_g))A_g^\perp=\trivial$, and hence $\rho(\varphi(r_g))(1-p_g)=0$ since the inclusion $I(A)\subseteq I_G(A)$ is normal by Proposition~\ref{prop: inclusion into I_G(A) normal}. This gives the reverse inequality $\rho(\varphi(r_g))\leq p_g$, proving the claim.

It follows that $\varphi(r_g)\in \dom_\rho$. Now let $g\mapsto u_g$ be the canonical partial ${}^*$-representation of~$G$ in $B\crossed G$ implementing~$\beta$, so that $u_gu_g^*=r_g$ for all $g\in G$. Then observing that $r_g\geq q_g$ for all $g\in G$ we obtain for all $a\in A$ 
\begin{equation*}
    \begin{aligned}
\rho(\varphi(\beta_g(ar_{g^{-1}})))&=\rho(\varphi(r_g))\rho(\varphi(\beta_g(ar_{g^{-1}})))\rho(\varphi(r_g))\\&=\rho(\varphi^\ddual(q_g))\rho(\varphi^\ddual(\beta_g(ar_{g^{-1}})))\rho(\varphi^\ddual(q_g))\\&=\rho(\varphi^\ddual(q_gu_gar_{g^{-1}}u_{g}^*q_g)))=\rho(\varphi^\ddual(q_gu_gaq_{g^{-1}}u_{g}^*q_g)))\\&=\rho(\Pi(q_g\nu_gaq_{g^{-1}}\nu_{g}^*q_g))=I(\alpha)_g(ap_{g^{-1}}).
         \end{aligned}
\end{equation*} The same reasoning employed in the proof of Proposition~\ref{pro: IIP-equiv. for unitization} shows that $\varphi(B)$ is contained in the multiplicative domain of $\rho$ because $B=\operatorname{Orb}_G^\beta(A)$. We conclude that the composition $\Psi:=\rho\circ\varphi\vert_B$ is a ${}^*$-homomorphism with $\Psi\vert_A=\id_A$ and satisfies $\Psi(r_g)=p_g$ for all $g\in G$, i.e. $\Psi$ is $G$-unital. To see that $\Psi$ is also $G$-equivariant, let $a\in A$ and $h\in G$. For $g\in G$ recall that $\beta_g(\beta_h(ar_{h^{-1}})r_{g^\inv})=\beta_{gh}(ar_{(gh)^\inv})r_g$ and thus \begin{equation*}
\begin{aligned}
    \Psi(\beta_g(\beta_h(ar_{h^\inv})r_{g^\inv}))&=\Psi(\beta_{gh}(ar_{(gh)^\inv})r_g)=I(\alpha)_{gh}(ap_{(gh)^\inv})p_g\\&=I(\alpha)_g(I(\alpha)_h(ap_{h^\inv})p_{g^\inv})=I(\alpha)_g(\Psi(\beta_h(ar_{h^\inv})r_{g^\inv})),
    \end{aligned}
\end{equation*} proving that $\Psi$ is indeed a $G$-morphism.
\end{proof}
    \end{lem}

The following is an important consequence of Lemma~\ref{lem:co-universal-unitization}.
    \begin{cor}\label{cor:unitization-isomorphism} Let $A$ be a unital $\Cast$-algebra and $\alpha=\partacg{A}{\alpha}$ a partial action on~$A$. Suppose that $\beta=\partacg{B}{\beta}$ is a partial action on a unital $\Cast$-algebra~$B$ with $G$-equivariant inclusions of $\Cast$-algebras $A\subseteq B\subseteq I_G(A)$. Then there is a $G$-isomorphism $\Psi\colon \operatorname{Orb}_G^{I_G(\alpha)}(B)\to \operatorname{Orb}_G^{I(\beta)}(B) $ such that $\Psi\vert_B=\id_B$. In particular, $I_G(A)=I_G(B)$.
    \begin{proof} Let $\{p_g\}_{g\in G}$ be the central projections in $I_G(A)$ corresponding to~$I_G(\alpha)=\partacg{I_G(A)}{I_G(\alpha)}$. Observe that $p_gB_g^\perp=\trivial$ for all $g\in G$ since $A_g\subseteq B_g\subseteq p_gI_G(A)$ and hence $B_g^\perp\subseteq A_g^\perp\subseteq (1-p_g)I_G(A)$.  Applying Lemma~\ref{lem:co-universal-unitization} with $\beta=\partacg{B}{\beta}$ playing the role of~$\alpha=\partacg{A}{\alpha}$ gives a $G$-morphism $\Psi\colon \operatorname{Orb}_G^{I_G(\alpha)}(B)\to \operatorname{Orb}_G^{I(\beta)}(B)$ such that $\Psi\vert_B=\id_B$. Notice that $\Psi$ is a surjective ${}^*$-homomorphism. Also, because $\Psi\vert_A=\id_A$ we see that~$\Psi$ is completely isometric when restricted to~$\operatorname{Orb}^{I(\alpha)}_G(A)$.  By Lemma~\ref{lem:isomorphism-envelopes-AB} we have $I_G(\operatorname{Orb}^{I(\alpha)}_G(A))=I_G(A)$, and hence the inclusion $\operatorname{Orb}^{I(\alpha)}_G(A)\subseteq  \operatorname{Orb}_G^{I_G(\alpha)}(B)$ is $G$-essential, which implies that $\Psi$ is a ${}^*$-isomorphism as needed. That $I_G(A)=I_G(B)$ follows because the inclusion $ \operatorname{Orb}_G^{I_G(\alpha)}(B)\subseteq I_G(A)$ is $G$-essential.
        \end{proof}
        \end{cor}

\begin{rem}\label{rem:partial-on-unitization} Let $(A, G, \alpha)$ be a partial action on~$A$ and suppose that $A$ is non-unital. Then $\alpha$ naturally induces a partial action on the minimal unitization~$\widetilde{A}$ of~$A$ by setting $\widetilde{A}_g:= A_g$ and $\widetilde{\alpha}_g:=\alpha_g$ if $g\neq e$, and $\widetilde{A}_e:=A$. As for global actions, $A\crossed G$ is an essential ideal in $\tilde{A}\crossed G$, and it follows that $(A, G, \alpha)$ has the ideal intersection property if and only if $(\widetilde{A}, G, \widetilde{\alpha})$ does so. See the arguments in Lemma 7.1 and Proposition 7.2 of~\cite{GefUrs23}.
    \end{rem}

We are ready to establish the analog of Theorem~\ref{thm: IIP-equiv-unital-pa} for arbitrary partial actions. 

\begin{thm}\label{thm: Equivalent IIP A to I(A)}
    Let $(A, G, \alpha)$ be a partial $\Cast$-dynamical system. Then the following are equivalent:
    \begin{enumerate}
        \item[\rm{(1)}] $(A, G, \alpha)$ has the ideal intersection property;
        \item[\rm{(2)}] every partial $\Cast$-dynamical system $(B, G, \beta)$ with $$(A, G, \alpha)\subseteq  (B, G, \beta) \subseteq (I_G(A), G, I_G(\alpha))$$ has the ideal intersection property;
        \item[\rm{(3)}] the unital partial $\Cast$-dynamical system $(I_G(A), G, I_G(\alpha))$ has the ideal intersection property.
    \end{enumerate}
\end{thm}

\begin{proof}
    We will show (1) $\Rightarrow$ (3) $\Rightarrow$ (2) $\Rightarrow$ (1). The first implication follows from Proposition~\ref{prop: A IIP implies I(A) IIP} and Theorem~\ref{thm: IIP-equiv-unital-pa}. For the implication (3) $\Rightarrow$ (2), let $(B, G, \beta)$ be a partial $\Cast$-dynamical system with $G$-equivariant inclusions $A\subseteq B\subseteq I_G(A)$ and let $J\ideal B\crossed G$ with $J\cap B=\trivial$. By Remark~\ref{rem:partial-on-unitization} we may assume that~$B$ is unital. Let $\pi\colon B\crossed G\to B(\hilb)$ be a unital representation of $B\crossed G$ with $\ker\pi=J$ and let $\hat{\pi}$ be a ucp map from  $\operatorname{Orb}_G^{I_G(\alpha)}(B)\crossed G$ into $ B(\hilb)$ extending~$\pi$. Identifying $\operatorname{Orb}_G^{I_G(\alpha)}(B)$ with the unitization $\operatorname{Orb}_G^{I(\beta)}(B)$ of $\beta=\partacg{B}{\beta}$ in~$I(B)$ using Corollary~\ref{cor:unitization-isomorphism} and applying Proposition~\ref{pro: IIP-equiv. for unitization} yields a $\Cast$-algebra~$D$ and a ucp map $\phi\colon \operatorname{im}\hat{\pi}\to D$ such that $\phi\circ\hat{\pi}$ is a ${}^*$-homomorphism and $\phi\circ\hat{\pi}$ is faithful on~$\operatorname{Orb}_G^{I_G(\alpha)}(B)$.

    Set $I:=\ker(\phi\circ\hat{\pi})$. Then $I$ is an ideal in~$ \operatorname{Orb}_G^{I_G(\alpha)}(B)\crossed G $ containing~$ J$ and satisfying $I\cap \operatorname{Orb}_G^{I_G(\alpha)}(B)=\trivial $. Since the $G$-injective envelope of $\operatorname{Orb}_G^{I_G(\alpha)}(B)$ is~$I_G(A)$, the assumption that $(I_G(A), G, I_G(\alpha))$ has the ideal intersection property combined with Theorem~\ref{thm: IIP-equiv-unital-pa} implies that $I=\trivial$. This gives that $J=\trivial$ as needed, establishing the implication (3) $\Rightarrow$ (2). The implication (2) $\Rightarrow$ (1) is clear.
  \end{proof}

    \begin{rem} The main technical difficulty in the proof of the implication  (3) $\Rightarrow$ (2) in Theorem~\ref{thm: Equivalent IIP A to I(A)} in comparison with the $G$-unital case (Theorem~\ref{thm: IIP-equiv-unital-pa}) is that if $\hat{\pi}\colon I_G(A)\crossed G\to B(\hilb)$ is a ucp map extending a ${}^*$-homomorphism $\pi\colon A\crossed G\to B(\hilb)$, then there is no clear reason why the family of units $\{p_g\}_{g\in G}$ should be in the multiplicative domain of~$\hat{\pi}$. In particular, $G$-injectivity of $I_G(A)$ does not naturally apply, in contrast to Theorem~\ref{thm: IIP-equiv-unital-pa}. However, it is also possible to prove  (3) $\Rightarrow$ (2) in Theorem~\ref{thm: Equivalent IIP A to I(A)} using pseudo-expectations as in Theorem~\ref{thm: IIP-equiv-unital-pa} as follows. By Lemma~\ref{lem:co-universal-unitization} there is a $G$-morphism $\Psi\colon \pi^\ddual(\operatorname{Orb}_G^{\alpha^\ddual}(A))\to I_G(A)$ such that $\Psi\circ\pi\vert_A=\id_A$. Equipping $B(\hilb)$ with the generalized unital partial action induced by the partial isometries in $\{\pi^\ddual(q_g\nu_g)\}_{g\in G}$ in $\pi^\ddual((A\crossed G)^\ddual)$ we obtain a $G$-morphism $\hat{\Psi}\colon B(\hilb)\to I_G(A)$ by $G$-injectivity of $I_G(A)$. Using an argument as in the proof of Theorem~\ref{thm: essentiality-equivariant-maps} (see also Proposition~\ref{pro: pseudo-expec=cond-expec-nonunital}) one can show that the composition $\hat{\Psi}\circ\hat{\pi}\colon I_G(A)\crossed G\to I_G(A)$ is actually a $G$-equivariant conditional expectation, and satisfies $J_{\hat{\Psi}\circ\hat{\pi}}\supseteq \ker\pi$. 
        \end{rem}

\subsection{\texorpdfstring{$G$}{G}-rigidity and \texorpdfstring{$G$}{G}-essentiality}

In this subsection we will discuss $G$-rigidity and $G$-essentiality for the inclusion $A\subseteq I_G(A)$. We begin by showing that this inclusion is always rigid with respect to $G$-equivariant ucp maps, and essential with respect to $G$-morphisms.

\begin{prop}\label{pro: rigidity-non-unital} Let $(A, G, \alpha)$ be a partial $\Cast$-dynamical system. \begin{enumerate}
    \item If $\phi\colon I_G(A)\to I_G(A)$ is a $G$-equivariant ucp map such that $\phi\vert_A=\id_A$, then $\phi$ is necessarily a $G$-morphism and $\phi=\id_{I_G(A)}$.
    \item  If $(B, G, \beta)$ is a generalized unital partial $\Cast$-dynamical system and $\phi\colon I_G(A)\to B$ is a $G$-morphism such that $\phi\vert_A$ is completely isometric, then~$\phi$ is completely isometric. 
\end{enumerate}
\begin{proof} We may assume that~$A$ is unital. For \emph{(i)}, let $\phi\colon I_G(A)\to I_G(A)$ be an equivariant map with $\phi\vert_A=\id_A$. We claim that~$\phi$ is $G$-unital. Indeed, let $\{p_g\}_{ g\in G}$ be the family of central projections in $I(A)$ corresponding to $I(\alpha)$. Let $(u_\mu)_{\mu}$ be an increasing approximate identity for~$A_g$. Then $p_g=\sup_{\mu}u_\mu$ in $I_G(A)$ since the inclusion $I(A)\subseteq I_G(A)$ is normal. In particular, this implies $\phi(p_g)\geq \sup_{\mu}\phi(u_\mu)=p_g $ because $\phi\vert_A=\id_A$ by assumption. Since $p_gA_g^\perp=\trivial$ and $A\subseteq \dom_\phi$, we deduce as in the proof of Lemma~\ref{lem:co-universal-unitization} that $\phi(p_g)=p_g$.

It follows that $\phi$ is a $G$-morphism. By equivariance we see that~$\phi$ is the identity map when restricted to~$\operatorname{Orb}_G^{I(\alpha)}(A)$, and hence $\phi=\id_{I_G(A)}$ by $G$-rigidity of the inclusion of $\operatorname{Orb}_G^{I(\alpha)}(A)$ in~$I_G(A)$.

For \emph{(ii)}, let $\phi\colon I_G(A)\to B$ be a $G$-morphism such that $\phi\vert_A$ is completely isometric. Then the essentiality of the inclusion $A\subseteq I(A)$ implies that $\phi\vert_{I(A)}$ is completely isometric, and hence the assertion follows by $G$-essentiality of the inclusion $I(A)\subseteq I_G(A)$.
    \end{proof}
    \end{prop}

Next we discuss the inclusion $A\subseteq I_G(A)$ in relation to $G$-equivariant maps that are not necessarily $G$-unital. The next theorem shows that~$A\subseteq I_G(A)$ is essential with respect to $G$-equivariant ucp maps that are multiplicative on~$A$. In addition, the next theorem implies that if $\beta=\partacg{B}{\beta}$ is any generalized unital partial action and $A\subseteq B$ is a $G$-equivariant inclusion of $\Cast$-algebras, then this inclusion extends to a $G$-morphism from~$B$ into the $G$-injective envelope~$I_G(A)$ of~$A$. 

\begin{thm}\label{thm: essentiality-equivariant-maps} Let~$(A, G,\alpha)$ be a partial $\Cast$-dynamical system.  Let $(B, G,\beta)$ be a generalized unital partial $\Cast$-dynamical system with corresponding family of commuting projections $\{r_g\}_{g\in G}$. Suppose that there exists a $G$-equivariant inclusion of $\Cast$-algebras $A\subseteq B$. Then:

\begin{enumerate}
    \item If $\phi\colon I_G(A)\to B$ is a $G$-equivariant (not necessarily $G$-unital) ucp map extending the inclusion $A\subseteq B$, then~$\phi$ is completely isometric.

    \item If $r_g(A_g^\perp\cap A)=\trivial$ for all $g\in G$, then there is a $G$-morphism $\psi\colon B\to I_G(A)$ such that $\psi\vert_A=\id_A$. 
\end{enumerate}
\begin{proof} We may assume that $A$ is unital. For \emph{(i)}, notice that as in \cite[Example 2.3]{Busetal22} $\beta=\partacg{B}{\beta}$ extends uniquely to a generalized unital partial action $\beta^\ddual=(\{B_g^\ddual\}_{g\in G},\{\beta_g^\ddual\}_{g\in G})$ on the bidual $B^\ddual$ with $B_g^\ddual=r_gB^\ddual r_g$, the bidual of~$B_g$, and $\beta_g^\ddual\colon B_{g^\inv}^\ddual\to B_g^\ddual$ the unique normal extension of~$\beta_g$. For each $g\in G$ let~$q_g$ be the unit of $A_g^\ddual$, and notice that $q_gB^\ddual q_g\subseteq B_g^\ddual$ for all $g\in G$. We define another generalized unital partial action $\gamma=\partacg{C}{\gamma}$ on $B^\ddual$ by setting $C_g:=q_gB^\ddual q_g$ and $\gamma_g:=\beta^\ddual_g\vert_{C_g}$ for all $g\in G$. 

Let $\phi\colon I_G(A)\to B$ be a $G$-equivariant ucp map extending the inclusion $A\subseteq B$. By Lemma~\ref{lem:co-universal-unitization} there is a $G$-morphism $\Psi\colon \operatorname{Orb}_G^{\alpha^\ddual}(A)\to \operatorname{Orb}_G^{I(\alpha)}(A) $ such that $\Psi\vert_A=\id_A$. Since the inclusion $\operatorname{Orb}_G^{\alpha^\ddual}(A)\subseteq B^\ddual$ is a $G$-embedding with respect to $\gamma=\partacg{C}{\gamma}$, by $G$-injectivity of $I_G(A)$ we can find a $G$-morphism $\hat\Psi\colon B^\ddual\to I_G(A)$ extending~$\Psi$ and we define a ucp map $\rho\colon I_G(A)\to I_G(A)$ as the composition $\rho:=\hat{\Psi}\circ \phi$. We will show that $\rho$ is a $G$-morphism.

As usual let $p_g\in I_G(A)$ be the unit of~$I_G(A)_g$. Since $A\subseteq \dom_\rho$ we must have $\rho(p_g)=p_g$ because $p_gA_g^\perp=\trivial$. It follows that~$p_g$ lies in the multiplicative domain of~$\rho$. Using in addition that $q_g\in \dom_{\hat\Psi}$ for all $g\in G$ we compute for $c\in I_G(A)$ 
\begin{equation*}
    \begin{aligned}
        \rho(I_G(\alpha)_g(cp_{g^\inv}))&=\rho(p_g) \rho(I_G(\alpha)_g(cp_{g^\inv})) \rho(p_g)\\&=\hat{\Psi}(q_g)\hat{\Psi}\big(\phi(I_G(\alpha)_g(cp_{g^\inv}))\big)\hat{\Psi}(q_g)\\&=\hat{\Psi}(q_g\beta_g(\phi(cp_{g^\inv}))q_g)\\&=\hat{\Psi}\big(\beta^\ddual_g(q_{g^\inv}\phi(cp_{g^\inv})q_{g^\inv})\big)\\&=I_G(\alpha)_g\big(\hat{\Psi}(q_{g^\inv}\phi(cp_{g^\inv})q_{g^\inv}))\\&=I_G(\alpha)_g(\rho(cp_{g^\inv})),
    \end{aligned}
\end{equation*} where in the third equality above we used that $\phi$ is $G$-equivariant, in the second last equality we used that $\hat\Psi$ is a $G$-morphism with respect to $\gamma=\partacg{C}{\gamma}$, and in the last equality we used that $\rho(p_{g^\inv})=p_{g^\inv}$.

We have proved that $\rho\colon I_G(A)\to I_G(A)$ is a $G$-morphism. Since $\rho\vert_A=\id_A$ it follows from  Proposition~\ref{pro: rigidity-non-unital} that $\rho=\id_{I_G(A)}$. In particular $\rho$ is a complete isometry, and hence so is~$\phi$, establishing the first assertion in the theorem. For the second assertion, suppose that  $r_g(A_g^\perp\cap A) =\trivial$ for all $g\in G$. Then $\hat\Psi(r_g)=p_g$ for all $g\in G$, and a computation similar to the above shows that $\psi:=\hat\Psi\vert_B$ is in fact a $G$-morphism from~$B$ into $I_G(A)$ with respect to $\beta=\partacg{B}{\beta}$.
\end{proof}
\end{thm}

The following corollary is an immediate consequence of Theorem~\ref{thm: essentiality-equivariant-maps}. Recall that an ideal $J\ideal A$ is invariant under a partial action $\alpha=\partacg{A}{\alpha}$ if $\alpha_g(J\cap A_{g^\inv})\subseteq J$ for all $g\in G$.
\begin{cor} Let~$A$ be a unital $\Cast$-algebra and $\alpha=\partacg{A}{\alpha}$ a partial action on~$A$. If~$A$ has no non-trivial $\alpha$-invariant ideals, then $I_G(A)$ has no non-trivial $I_G(\alpha)$-invariant ideals. 
\begin{proof} Suppose~$A$ has no non-trivial $\alpha$-invariant ideals and let $J$ be a proper $I_G(\alpha)$-invariant ideal of~$I_G(A)$. Then $J\cap A$ is a proper $\alpha$-invariant ideal of~$A$ because~$A$ is unital. By assumption we must have $J\cap A=\trivial$, and it follows that the quotient map $\pi\colon I_G(A)\to I_G(A)/J$ is faithful on~$A$. Because $\pi$ is $G$-equivariant with respect to the partial action on $I_G(A)/J$ induced by $I_G(\alpha)$ we deduce from Theorem~\ref{thm: essentiality-equivariant-maps} that~$\pi$ is faithful, giving that $J=\trivial$ as required.
    \end{proof}

\end{cor}

\subsection{Injective envelopes of reduced crossed products}

In this subsection our main purpose is to show that the inclusion $A\crossed G\subseteq I_G(A)\crossed G$ is essential for an arbitrary partial action~$\alpha=\partacg{A}{\alpha}$, which implies the identity $I(A\crossed G)=I(I_G(A)\crossed G)$. This is a crucial tool when studying the ideal structure of crossed products. We will require a canonical generalized unital partial action on the injective envelope of $A\crossed G$ extending $\alpha$. In the unital setting this follows from Proposition~\ref{prop: gen part ac on crossed product}.

Let $\alpha=\partacg{A}{\alpha}$ be a partial action. Recall from Remark~\ref{rem: crossed-action-no-units} that for each~$g\in G$ the partial action~$\alpha$ restricts to a partial action on $A_g$ in a natural way, and we can identify the reduced partial crossed product $A_g\crossed G$ with the hereditary $\Cast$-subalgebra of $A\crossed G$ $$A_g\crossed G=\overline{\lincomb}\{ a \nu_h \, \vert \, a \in A_g\cap A_h\cap A_{hg}\}.$$ In particular, $I(A_g\crossed G)$ is a unital hereditary subalgebra of~$I(A\crossed G)$.

\begin{prop}\label{prop: gen part ac on crossed product-non-unital}
    Let $(A, G,\alpha)$ be a partial $\Cast$-dynamical system. Then there exists a generalized unital partial action $\hat{\alpha}=\partacg{I(A\crossed G)}{\hat{\alpha}}$ on $I(A\crossed G)$ such that $I(A\crossed G)_g=I(A_g\crossed G)$ for all $g\in G$ and the inclusion $A\subseteq I(A \crossed G) $ is $G$-equivariant. Moreover, there exists a faithful $G$-morphism $\Phi\colon I(A\crossed G)\to I_G(A)$ that restricts to the canonical conditional expectation $E_A\colon A\crossed G\to A$ on $A\crossed G$.
    \begin{proof} Fix $g\in G$. For $a\in  A_{g^{-1}}\cap A_h\cap A_{hg^{-1}}$ we see that $\alpha_g(a)\in A_g\cap A_{gh}\cap A_{ghg^{-1}}$. Hence the ${}^*$-isomorphism $\alpha_g\colon A_{g^{-1}}\to A_g$ extends to a ${}^*$-isomorphism $A_{g^{-1}}\crossed G\cong A_g\crossed G$ that maps a spanning element $a\nu_h\in A_{g^{-1}}\crossed G$ to $\alpha_g(a)\nu_{ghg^{-1}}\in A_g\crossed G$. This extends uniquely to a ${}^*$-isomorphism $\hat{\alpha}_g\colon I(A_{g^{-1}}\crossed G)\to I(A_g\crossed G)$ by an argument as in the proof of Proposition~\ref{prop: injective-p-action-hereditary}.

    Let $(u_\mu)_{\mu}$ be an increasing approximate identity for $A_g$. Then $(u_\mu)_{\mu}$ is an approximate identity for $A_g\crossed G$ and setting $r_g:=\sup_{\mu}u_\mu\in I(A\crossed G)$ gives the identification $I(A_g\crossed G)=r_gI(A\crossed G) r_g$. That the pair $\hat{\alpha}=(\{I(A_g\crossed G)\}_{g\in G}, \{\hat{\alpha}_g\}_{g\in G})$ satisfies the axioms of a generalized unital partial action follows because if $q_g$ denotes the unit of~$A_g^\ddual$ for all $g\in G$ and $\phi\colon (A\crossed G)^\ddual\to I(A\crossed G)$ is a ucp map extending the inclusion $A\crossed G\subseteq I(A\crossed G)$, then we must have $\phi(q_g)=r_g$. From this we see that $\{r_g\}_{g\in G}$ is a commuting family of projections in $I(A\crossed G)$ and satisfies $\hat{\alpha}_g(r_{g^\inv}r_h)=r_gr_{gh}$ for all $g,h\in G$ since $\alpha=\partacg{A}{\alpha}$ is a partial action. This proves the first part of the statement.

    For the second part, assume first that $A$ is $G$-injective, i.e. $A=I_G(A)$ and in particular~$\alpha$ is $G$-unital. Then the canonical conditional expectation $E_A\colon A\crossed G \to A$ is a $G$-morphism and hence it extends to a $G$-morphism $\Phi\colon I(A\crossed G)\to A$ by $G$-injectivity. To see that $\Phi$ is faithful, we argue as in \cite[Theorem~3.2]{Bryd22}. If $J:=\{d\in I(A\crossed G)\,\vert\, \Phi(d^*d)=0\},$ then $J\subseteq I(A\crossed G)$ is a $\Cast$-algebra that is also an $(A\crossed G)$-module satisfying $J\cap (A\crossed G)=\trivial$. By \cite[Lemma~1.2]{Ham82_centre} we must have $J=\trivial$, giving that~$\Phi$ is faithful.  
    
    In case $A$ is not $G$-injective, let $C:=\operatorname{Orb}_G^{I(\alpha)}(A)$ be the $\Cast$-algebra generated by the orbits of~$A$ in~$I(A)$ and let $\gamma=\partacg{C}{\gamma}$ be the unital partial action induced by~$I(\alpha)$. Let $\phi\colon C\crossed G\to I(A\crossed G)$ be a ucp map extending the inclusion $A\crossed G\subseteq I(A\crossed G)$. We claim that $\phi$ is $G$-unital. To see this, observe first that $\phi(p_g)\geq r_g$ for all $g\in G$ since $A_g\subseteq C_g=p_gC$. For the reverse inequality, let $\psi\colon I(A\crossed G)\to I(I_G(A)\crossed G)$ be a ucp map extending the inclusion $A\crossed G\subseteq I_G(A)\crossed G$. If $(u_\mu)_\mu$ is an increasing approximate identity for~$A_g$, then $p_g=\sup_{\mu}u_\mu$ in $I(I_G(A)\crossed G))$ by an application of the faithful conditional expectation $\Phi\colon I_G(A)\crossed G\to I_G(A)$ from above, or alternatively by \cite[Lemma~3.1]{Ham85} and Corollary~\ref{cor: G-essential crossed prod in I(A cross G)}. Hence $\psi(r_g)\geq p_g$. Let $\hat{\phi}\colon I(I_G(A)\crossed G)\to I(A\crossed G)$ be a ucp map extending $\phi\colon C\crossed G\to I(A\crossed G)$. Then we have $\hat{\phi}\circ\psi=\id_{I(A\crossed G)}$ by rigidity, giving that $\phi(p_g)\leq \hat{\phi}(\psi(r_g))=r_g$, and so $\phi(p_g)=r_g$.
    
    The above shows that $\phi$ is $G$-unital. To see that $\phi$ is in fact a ${}^*$-homomorphism, notice that on one hand $\phi(\nu_g)^*\phi(\nu_g)\leq \phi(p_{g^{-1}})=r_{g^\inv}$ and on the other hand, $$\phi(\nu_g)^*\phi(\nu_g)\geq \sup_{\mu}(\phi(\nu_g)^*u_\mu\phi(\nu_g))=\sup_{\mu}\alpha_{g^\inv}(u_\mu)=r_{g^\inv}.$$ Thus $\phi(\nu_g)^*\phi(\nu_g)=\phi(p_g)$ for all $g\in G$, which implies that each $\nu_g$ lies in the multiplicative domain of~$\phi$. We conclude that $\phi$ is a ${}^*$-homomorphism. An argument similar to that in the proof of Proposition~\ref{prop: A IIP implies I(A) IIP} shows that~$\phi$ is in addition faithful, giving the inclusion of $\Cast$-algebras $A\crossed G \subseteq C\crossed G\subseteq I(C\crossed G)$.

    The inclusion $C\crossed G\subseteq I(A\crossed G)$ established above is a $G$-embedding with respect to the generalized unital partial action on $I(A\crossed G)$ constructed in the first part of the proof. Now let $E_C\colon C\crossed G\to C$ be the canonical conditional expectation. Then $E_C$ is a $G$-morphism, and hence by $G$-injectivity of $I_G(A)$ there exists a $G$-morphism $\Phi\colon I(A\crossed G)\to I_G(A)$ such that $\Phi\vert_{C\crossed G}= E_C$. In particular, $\Phi$ restricts to the canonical conditional expectation on~$A\crossed G$. That $\Phi$ is faithful follows by the same argument applied for injective partial $\Cast$-dynamical systems. 
\end{proof}
    \end{prop}

We can now establish the analog of \cite[Theorem 3.4]{Ham85} in the context of partial $\Cast$-dynamical systems.

\begin{thm}\label{thm: G-essential crossed prod in I(A cross G)-non-unital}
    Let $(A,G,\alpha)$ and $(B, G,\beta)$ be partial $\Cast$-dy\-na\-mi\-cal  systems with a $G$-equi\-var\-iant inclusion of $\Cast$-algebras $A\subseteq B$. Then $$(A,G,\alpha) \subseteq (B, G,\beta) \subseteq (I_G(A), G, I_G(\alpha))$$ if and only if there is a $G$-equivariant inclusion of $\Cast$-algebras $$A\crossed G \subseteq B\crossed G \subseteq I(A\crossed  G).$$ In particular, $I(A\crossed G)=I(I(A)\crossed G)=I(I_G(A)\crossed G)$.
    \begin{proof} For the forward direction, it suffices to show that $A\crossed G\subseteq I_G(A)\crossed G\subseteq I(A\crossed G)$. Let $\rho\colon I_G(A)\crossed G\to I(A\crossed G)$ be a ucp map extending the inclusion $A\crossed G\subseteq I(A\crossed G)$. We will show that $\rho$ is a ${}^*$-homomorphism. The argument in the proof of Proposition~\ref{prop: gen part ac on crossed product-non-unital} shows that $\rho$ is $G$-unital and restricts to a faithful ${}^*$-homomorphism from $C\crossed G$ to $I(A\crossed G)$, where $C=\mathrm{Orb}_G^{I(\alpha)}(A)$. So it suffices to show that $\rho\vert_{I_G(A)}$ is multiplicative. 
    
    Let $\Phi\colon I(A\crossed G)\to I_G(A)$ be the faithful $G$-morphism extending the canonical conditional expectation as in Proposition~\ref{prop: gen part ac on crossed product-non-unital}. Then~$\Phi\circ\rho\vert_{I_G(A)}=\id_{I_G(A)}$ by Proposition~\ref{pro: rigidity-non-unital} and hence for $c\in I_G(A)$ we have $$c^*c=\Phi(\rho(c))^*\Phi(\rho(c))\leq \Phi(\rho(c)^*\rho(c))\leq \Phi(\rho(c^*c))=c^*c,$$ giving that $\Phi(\rho(c)^*\rho(c))=\Phi(\rho(c^*c))$. Since $\rho(c)^*\rho(c)\leq \rho(c^*c)$ and $\Phi$ is faithful this implies $\rho(c)^*\rho(c)=\rho(c^*c)$. Because this holds for arbitrary $c\in I_G(A)$ we deduce that $\rho\vert_{I_G(A)}$ is multiplicative, and hence $\rho$ is a ${}^*$-homomorphism as needed. Since the composition $\Phi\circ\rho\colon I_G(A)\crossed G\to I_G(A)$ coincides with the canonical conditional expectation on $I_G(A)\crossed G$, we conclude that $\rho$ is faithful, giving the desired inclusion.

    Conversely, suppose the inclusion $A\subseteq B$ induces $G$-equivariant inclusions $A\crossed G \subseteq B\crossed G \subseteq I(A\crossed  G)$. Then $A\crossed G\subseteq B\crossed G$ is essential, and it follows that $I(A\crossed G)=I(B\crossed G)$. From the first part we also have inclusions $B\crossed G\subseteq I_G(B)\crossed G\subseteq I(B\crossed G)$. If $\{r_g\}_{g\in G}$ and $\{p_g\}_{g\in G}$ denote the family of central projections corresponding to $I_G(\alpha)$ and $I_G(\beta)$, respectively, then we must have $p_g= r_g$ by $G$-equivariance since $A_g\crossed G\subseteq B_g\crossed G\subseteq p_gI(A\crossed G)p_g$. So let $\Phi_A\colon I(A\crossed G)\to I_G(A)$ be the faithful $G$-morphism extending the canonical conditional expectation~$E_A$ and let $\Phi_B$ be the faithful $G$-morphism $I(A\crossed G)\to I_G(B)$ extending $E_B$. Set $\psi:=\Phi_A\vert_{I_G(B)}$. Then $\psi\colon I_G(B)\to I_G(A)$ is a $G$-morphism extending the inclusion $A\subseteq I_G(A)$, and by Proposition~\ref{pro: rigidity-non-unital} we have $\psi\circ\Phi_B\vert_{I_G(A)}=\id_{I_{G}(A)}$. Since the composition $\psi\circ\Phi_B\colon I(A\crossed G)\to I_G(A)$ is a faithful $G$-morphism as it coincides with the canonical conditional expectation on $A\crossed G$, the same reasoning applied in the first part gives that $\Phi_B\vert_{I_G(A)}$ is a ${}^*$-homomorphism. This gives a $G$-embedding of $\Cast$-algebras $I_G(A)\subseteq I_G(B)$. 
    
    To see that $\Phi_B\vert_{I_G(A)}$ is surjective, we argue as in \cite[Lemma~3.3]{Ham85}. Set $\rho:=\Phi_B\vert_{I_G(A)}$ and take $d\in I_G(B)$.  Then $$0\leq \psi\big((\rho(\psi(d))-d)^*(\rho(\psi(d))-d)\big)=0$$ since $\psi\circ\rho=\id_{I_G(A)}$ and $\operatorname{im}\rho\subseteq \dom_\psi$. This gives $\rho(\psi(d))=d$ as $\psi\circ\Phi_B$ is faithful. So $\rho\colon I_G(A)\to I_G(B)$ is a $G$-isomorphism, giving the desired inclusions $A\subseteq B\subseteq I_G(A)$.  
        \end{proof}
   \end{thm}

   \begin{rem} In the proofs of Proposition~\ref{prop: gen part ac on crossed product-non-unital} and Theorem~\ref{thm: G-essential crossed prod in I(A cross G)-non-unital} we have avoided using monotone crossed products, and hence also Corollary~\ref{cor: G-essential crossed prod in I(A cross G)} and \cite[Theorem~3.4]{Ham85}.
   \end{rem}

The following corollary is an immediate consequence of Theorem~\ref{thm: G-essential crossed prod in I(A cross G)-non-unital} and the general theory of injective envelopes of $\Cast$-algebras, see \cite[Theorems 6.3 and 7.1]{Ham81}.

\begin{cor}\label{cor: primality}
    Let $\partacg{A}{\alpha}$ be a partial $\Cast$-dynamical system. Then the following are equivalent:
    \begin{enumerate}
        \item $A\crossed G$ is prime;
         \item  $I(A)\crossed G$ is prime;
         \item  $I_G(A)\crossed G$ is prime;
         \item $Z(I( A\crossed G))=\mathbb{C}.$
    \end{enumerate}
\end{cor}

\subsection{Pseudo-expectations for non-unital partial actions}

We now consider pseudo-ex\-pec\-ations in the context of partial actions that are not necessarily unital.

\begin{defi}\label{def: pseudo-non-unital} We define a \emph{pseudo-expectation} for a partial $\Cast$-dynamical system $(A, G,\alpha)$ to be a $G$-equivariant ccp map $\phi \colon A\rtimes_r G \to I_G(A)$ such that $\phi\vert_A = \id_A$.
\end{defi}

We verify next that pseudo-expectations for $(A,G,\alpha)$ indeed correspond to $G$-equi\-var\-iant conditional expectations for $(I_G(A), G, I_G(\alpha))$.
\begin{prop}\label{pro: pseudo-expec=cond-expec-nonunital}
    Let $(A, G, \alpha)$ be a partial $\Cast$-dynamical system. Then there is a one-to-one correspondence between pseudo-expectations $\phi \colon A \crossed G \to I_G(A)$ for $(A, G, \alpha)$ and $G$-equivariant conditional expectations $\Phi \colon I_G(A) \crossed G \to I_G(A)$.
    \begin{proof} By Proposition~\ref{prop: 1-1 pseudo-exp with cond exp} it suffices to show that there is a one-to-one correspondence between pseudo-expectations for $(A, G,\alpha)$ and pseudo-expectations for $(C,G,\gamma)$, where $C=\operatorname{Orb}_G^{I(\alpha)}(A)$ and $\gamma=\partacg{C}{\gamma}$ is the unital partial action on~$C$ induced by~$I(\alpha)$. Let $\phi$ be a pseudo-expectation for $(A,G, \alpha)$ and let $\hat{\phi}\colon C\crossed G\to I_G(A)$ be a ucp extension of~$\phi$. Then $\hat{\phi}$ is necessarily $G$-unital since $A\subseteq \dom_{\hat{\phi}}$. We will show that $\hat\phi$ is also $G$-equivariant. 
    
    Since $\hat{\phi}$ is $G$-unital we have that $\hat{\phi}(\nu_h)=p_h\hat{\phi}(\nu_h)$ for all $h\in G$ . For $g, h\in G$ set $d:=p_{g^\inv}\nu_hp_{g^\inv}\in C_{g^\inv}\crossed G$. Then $\hat{\phi}(\hat{\alpha_g}(d))=p_{gh}p_gp_{gh g^\inv}\hat{\phi}(\hat{\alpha_g}(d))$. Let $a\in A_g\cap A_{gh}\cap A_{gh g^\inv}$, so that $\alpha_{g^\inv}(a)\in A_{g^\inv}\cap A_{h}\cap A_{hg^\inv}$, and observe that $\alpha_{g^\inv} (a)d=\alpha_{g^\inv}(a)\nu_h\in A_g\crossed G$. Since $\hat{\phi}$ extends $\phi$ and $\phi$ is $G$-equivariant we get $$a\hat{\phi}(\hat{\alpha}_g(d))=I(\alpha)_g(\phi(\alpha_{g^\inv}(a)d))=I(\alpha)_g(\hat{\phi}(\alpha_{g^\inv}(a)d))=aI_G(\alpha)_g(\hat{\phi}(d)),$$ giving that
    $$(\hat{\phi}(\hat{\alpha}_g(d))-I_G(\alpha)_g(\hat{\phi}(d)))^*a(\hat{\phi}(\hat{\alpha}_g(d))-I_G(\alpha)_g(\hat{\phi}(d)))=0.$$ Since $p_{gh}p_gp_{gh g^\inv}$ satisfies $p_{gh}p_gp_{gh g^\inv}=\sup_\mu u_\mu$ for $(u_\mu)_{\mu}$ an increasing approximate identity for $A_g\cap A_{gh}\cap A_{gh g^\inv}$ we deduce that $\hat{\phi}(\hat{\alpha}_g(d))=I_G(\alpha)_g(\hat{\phi}(d))$. A similar argument shows that $\hat{\phi}(I(\alpha)_g(ap_{g^\inv}))=I_G(\alpha)_g(\hat{\phi}(ap_{g^{-1}}))$ for all $g\in G$ and $a\in A$ and hence $\hat{\phi}$ is a $G$-morphism as needed. Uniqueness of $\hat{\phi}$ follows similarly.
         \end{proof}
\end{prop}

We also have a characterization of the ideal intersection property in terms of pseudo-expectation as in the unital setting.

\begin{cor}\label{cor: pseudo-IIP} Let $(A, G, \alpha)$ be a partial $\Cast$-dynamical system. Then $(A, G, \alpha)$ has the ideal intersection property if and only if every pseudo-expectation for $(A, G, \alpha)$ is faithful.
\end{cor}

\section{Injective envelopes for abelian partial \texorpdfstring{$\Cast$}{C*}-dynamical systems}\label{sec: abelian-groupoid}

 In this section we fix a unital commutative $\Cast$-algebra $A=\mathrm{C}(X)$ for $X$ a compact Hausdorff space and a partial action $\alpha_g=(\{\mathrm{U}_g\}_{g\in G}, \{\alpha_g\}_{g\in G})$ of $G$ on $X$. This gives rise to an abelian partial $\Cast$-dynamical system $\alpha^*=(\{\mathrm{C}_0(\mathrm{U}_{g})\}_{g\in G},\{\alpha^*_g\}_{g\in G})$ (see Example~\ref{ex: abelian-partial-ac}). In order to lighten notation we will simply denote by $\alpha_g$ instead of $\alpha_g^*$ the ${}^*$-isomorphism $\mathrm{C}_0(\mathrm{U}_{g^\inv})\cong \mathrm{C}_0(\mathrm{U}_g)$ induced by $\alpha_g$. In addition, in what follows we will write $gx$ instead of $\alpha_g(x)$ for the image of $x\in \mathrm{U}_{g^{-1}}$ in $\mathrm{U}_g$ under~$\alpha_g$.
 
The \emph{transformation groupoid} associated to the partial action $\alpha=(\{\mathrm{U}_g\}_{g\in G}, \{\alpha_g\}_{g\in G})$, denoted by $\mathcal{G}_\alpha\ltimes X$, is $$\mathcal{G}_\alpha\ltimes X=\{(g,x)\,\vert \, x\in \mathrm{U}_{g^{-1}}\}$$ equipped with the relative product topology, range and source maps $s,r$ given by~$s(g,x)=x$ and $r(g,x)=gx$, composition defined by $(g, hx)(h,x)=(gh,x)$ and inverse $(g,x)^{-1}=(g^{-1}, gx)$. This is a locally compact \'{e}tale Hausdorff groupoid with unit space~$X$. By~\cite{Aba04} the reduced partial crossed product $\mathrm{C}(X)\crossed G$ is canonically isomorphic to the reduced groupoid $\Cast$-algebra of $\mathcal{G}_\alpha\ltimes X$. We will see that the injective envelope of $\mathrm{C}(X)$ in the category of generalized unital partial $\Cast$-dynamical systems coincides with the injective envelope of $\mathrm{C}(X)$ in the category of unital $(\mathcal{G}_\alpha\ltimes X)$-$\Cast$-algebras as constructed in \cite{Bor20, Ken+21}.

We will follow the terminology and notation from~\cite{Ken+21}. Let $B$ be a unital $\Cast$-algebra and $\iota_B\colon \mathrm{C}(X)\hookrightarrow B$ a unital inclusion of $\Cast$-algebras. For an open subset $V\subseteq X$ let $B_V$ be the hereditary subalgebra of~$B$ defined by $$B_{V}:=\overline{\iota_B(\mathrm{C}_0(V)) B\iota_B(\mathrm{C}_0(V))}.$$ Notice that $B_V$ is an ideal of $B$ if $B$ is commutative.%

\begin{lem}\label{lem: partial-action-groupoid} Let $B$ be a unital commutative $\Cast$-algebra with a unital inclusion $\iota_B\colon \mathrm{C}(X)\hookrightarrow B$ of $\Cast$-algebras. Then $B$ is a $(\mathcal{G}_\alpha\ltimes X)$-$\Cast$-algebra as in \textup{\cite[Definition~3.3]{Ken+21}} if and only if for each $g\in G$ there exists a ${}^*$-isomorphism $\beta_g\colon B_{\mathrm{U}_{g^\inv}}\to  B_{\mathrm{U}_{g}}$ such that the pair $\beta=(\{B_{\mathrm{U}_g}\}_{g\in G},\{\beta_g\}_{g\in G})$ is a partial action on~$B$ and the inclusion $\iota_B$ is $G$-equivariant. 
 \begin{proof} Suppose that $B$ is a $(\mathcal{G}_\alpha\ltimes X)$-$\Cast$-algebra. For $g\in G$ we have that $(g,\mathrm{U}_{g^\inv})$ is an open bisection of $(\mathcal{G}_\alpha\ltimes X)$, i.e. an open subset of $\mathcal{G}_\alpha\ltimes X$ on which the range and source maps are homeomorphisms onto their images, and thus it follows from \cite[Definition~3.3]{Ken+21} that there exists a ${}^*$-isomorphism $\beta_g\colon B_{\mathrm{U}_{g^\inv}}\to  B_{\mathrm{U}_{g}}$. By definition of a $(\mathcal{G}_\alpha\ltimes X)$-$\Cast$-algebra the inclusion $\iota_B\colon \mathrm{C}(X)\hookrightarrow B$ is equivariant with respect to the $(\mathcal{G}_\alpha\ltimes X)$-$\Cast$-algebra structure of~$B$, and hence we must have that $\beta_g(B_{\mathrm{U}_{g^\inv}}\cap B_{\mathrm{U}_h})\subseteq B_{\mathrm{U}_{gh}}$. Axiom (iii) from the definition of a partial action follows from the compatibility condition in~\cite[Definition~3.3]{Ken+21}. This shows that $\beta=\partacg{B}{\beta}$ is a partial action for which the inclusion $\iota_B\colon \mathrm{C}(X)\hookrightarrow B$ is $G$-equivariant.

 Conversely, suppose that for each $g\in G$ there exists a ${}^*$-isomorphism $\beta_g\colon B_{\mathrm{U}_{g^\inv}}\to  B_{\mathrm{U}_{g}}$ such that the pair $\beta=(\{B_{\mathrm{U}_g}\}_{g\in G},\{\beta_g\}_{g\in G})$ is a partial action on~$B$ and the inclusion~$\iota_B$ is $G$-equivariant. From the groupoid structure of $\mathcal{G}_\alpha\ltimes X$ we see that if~$\gamma$ is a bisection of $\mathcal{G}_\alpha\ltimes X$, then $\gamma=\bigcup_{g\in G}(g,V_{g^{-1}})$, where $\{V_{g^\inv}\}_{g\in G}$ and $\{\alpha_g(V_{g^\inv})\}_{g\in G}$ are families of mutually disjoint open subsets of~$X$. The assumption that $B$ is commutative implies that we can identify $B_{\mathrm{supp}\,\gamma}$ with the $\mathrm{c}_0$-direct sum of the $B_{V_{g^\inv}}$s, where $\mathrm{supp}\,\gamma=s(\gamma)\subseteq X$. Similarly we can identify $B_{\mathrm{im}\,\gamma}$ with the $\mathrm{c}_0$-direct sum of the $B_{\alpha_g(V_{g^\inv})}$s, where $\mathrm{im}\,\gamma=r(\gamma)$. By $G$-equivariance of the inclusion  $\iota_B\colon \mathrm{C}(X)\hookrightarrow B$ we have that $\beta_g$ restricts to a ${}^*$-isomorphism $B_{V_{g^\inv}}\to B_{\alpha_g(V_{g^\inv})}$ for each $g\in G$. By the above observation this induces a ${}^*$-isomorphism $\beta_\gamma\colon B_{\mathrm{supp}\,\gamma}\to  B_{\mathrm{im}\,\gamma}.$ The compatibility axiom from \cite[Definition~3.3]{Ken+21} then follows because $\beta_g\colon B_{\mathrm{U}_{g^\inv}}\to  B_{\mathrm{U}_{g}}$ is a partial action.
\end{proof}   
\end{lem}

By \cite[Theorem~4.9]{Ken+21} the $\Cast$-algebra $\mathrm{C}(X)$ admits an injective envelope in the category of unital $(\mathcal{G}_\alpha\ltimes X)$-$\Cast$-algebras with $(\mathcal{G}_\alpha\ltimes X)$-ucp maps as morphisms. This coincides with the injective envelope of $\mathrm{C}(X)$ in the category of operator systems with $\mathcal{G}_\alpha\ltimes X$ actions as constructed by Borys in \cite{Bor20} (see \cite[Theorem~4.20]{Ken+21}). The injective envelope of the $(\mathcal{G}_\alpha\ltimes X)$-$\Cast$-algebra $\mathrm{C}(X)$ is again a commutative $\Cast$-algebra, whose spectrum we denote by $\partial_H(\mathcal{G}_\alpha\ltimes X)$, following the notation of~\cite{Ken+21}.

\begin{thm}\label{thm: abelian-partial-actions} Let $\mathcal{C}:=\mathrm{C}(\partial_H(\mathcal{G}_\alpha\ltimes X))$ be the injective envelope of the $(\mathcal{G}_\alpha\ltimes X)$-$\Cast$-algebra $\mathrm{C}(X)$ in the category of unital $(\mathcal{G}_\alpha\ltimes X)$-$\Cast$-algebras. Let $\gamma=(\{\mathcal{C}_{\mathrm{U}_g}\}_{g\in G},\{\gamma_g\}_{g\in G})$ be the partial action on~$\mathcal{C}$ from Lemma~\ref{lem: partial-action-groupoid} and let $I(\gamma)=\partacg{\mathcal{C}}{I(\gamma)}$ be the unital partial action on~$\mathcal{C}$ extending $\gamma$ as in Proposition~\ref{prop: partac on I(A)}. Then there exists a $G$-isomorphism $\Psi\colon \mathcal{C}\to I_G(\mathrm{C}(X)) $ such that $\Psi\circ\iota=\id_{\mathrm{C}(X)}$. 
\begin{proof} For each $g\in G$ let $r_g\in\mathcal{C}$ be the unit of the ideal $\mathcal{C}_g$. Since $r_g$ is the supremum of any increasing approximate identity for $\mathrm{C}(\mathrm{U}_{g})$, we must have $r_g( \mathrm{C}(\mathrm{U}_g)^\perp\cap\mathrm{C}(X))=\trivial$. Then by Theorem~\ref{thm: essentiality-equivariant-maps} there exists a $G$-morphism $\Psi\colon \mathcal{C}\to I_G(\mathrm{C}(X)) $ with $\Psi\circ\iota=\id_{\mathrm{C}(X)}$. View $I_G(\mathrm{C}(X))$ with the $(\mathcal{G}_\alpha\ltimes X)$-$\Cast$-algebra structure induced by the partial action $I_G(\alpha)=\partacg{I_G(\mathrm{C}(X))}{I_G(\alpha)}$ as in Lemma~\ref{lem: partial-action-groupoid}. Then~$\Psi$ is a $(\mathcal{G}_\alpha\ltimes X)$-ucp map as in \cite[Definition~3.5]{Ken+21}, and we deduce from essentiality of the inclusion $\iota\colon \mathrm{C}(X)\hookrightarrow \mathcal{C}$ that~$\Psi$ is completely isometric.

To see that $\Psi$ is surjective, we apply $(\mathcal{G}_\alpha\ltimes X)$-injectivity of~$\mathcal{C}$ to get a $(\mathcal{G}_\alpha\ltimes X)$-ucp map $\phi\colon I_G(\mathrm{C}(X))\to \mathcal{C}$ extending the inclusion~$\iota$. We claim that~$\phi$ is $G$-unital. Indeed, the composition $\phi\circ\Psi\colon\mathcal{C}\to\mathcal{C}$ is a $(\mathcal{G}_\alpha\ltimes X)$-ucp map with $\phi\circ\Psi\vert_{\mathrm{C}(X)}=\id_{\mathrm{C}(X)}$, and hence $\phi\circ\Psi=\id_{\mathcal{C}}$ by rigidity of the inclusion of $\iota\colon \mathrm{C}(X)\hookrightarrow \mathcal{C}$. It follows that $r_g=\phi(\Psi(r_g))=\phi(p_g)$ for all $g\in G$, where~$p_g$ is the unit of $I_G(\mathrm{C}(X))_g$, proving that~$\phi$ is $G$-unital. For $G$-equivariance of~$\phi$, take a positive element $c\in I_G(\mathrm{C}(X))_+$ and for $g\in G$ let $(u_\mu)_{\mu}$ be an increasing approximate identity for~$\mathrm{C}(\mathrm{U}_g)$. We compute using that~$\mathcal{C}$ is commutative and~$\phi$ is $G$-unital \begin{equation*}
\begin{aligned}
    \phi(I_G(\alpha)_g(c))&=\phi(I_G(\alpha)_g(c))\sup_\mu \iota(u_\mu)=\sup_\mu \phi(I_G(\alpha)_g(c))\iota(u_\mu)\\&=\sup_\mu \phi(I_G(\alpha)_g(c)u_\mu)=\sup_\mu I(\gamma)_g(\phi(c\gamma_{g^\inv}(u_\mu)))\\&=I(\gamma)_g(\sup_\mu \phi(c\gamma_{g^\inv}(u_\mu)))=I(\gamma)_g(\phi(c) p_{g^\inv})\\&=I(\gamma)_g(\phi(c)),
    \end{aligned}
    \end{equation*}
where in the fourth equality we used that~$\phi$ is a $(\mathcal{G}_\alpha\ltimes X)$-ucp map. This shows that $\phi\colon I_G(\mathrm{C}(X))\to \mathcal{C}$ is a $G$-morphism. Since $\Psi\circ \phi\vert_{\mathrm{C}(X)}=\id_{\mathrm{C}(X)}$ we must have $\Psi\circ \phi=\id_{I_G(\mathrm{C}(X))}$ by Proposition~\ref{pro: rigidity-non-unital}, which implies that $\Psi$ is surjective and hence a $G$-isomorphism as wanted. 
\end{proof}
    
\end{thm}

\begin{rem} Theorem~\ref{thm: abelian-partial-actions} combined with \cite[Theorem~4.2.6]{Kroell25} implies that the characterization of the ideal intersection property from \cite[Theorem~7.2]{Ken+21} applies to the transformation groupoid associated to an arbitrary partial action of~$G$ on a compact space. 
    \end{rem}

\addcontentsline{toc}{section}{References}
\printbibliography

\end{document}

%% file: commands.tex
\theoremstyle{definition}
\newtheorem{defi}{Definition}[section]

\theoremstyle{plain}
\newtheorem{thm}[defi]{Theorem}
\newtheorem{cor}[defi]{Corollary}
\newtheorem{prop}[defi]{Proposition}
\newtheorem{lem}[defi]{Lemma}

\newtheorem{introthm}{Theorem}
\newtheorem{introcor}[introthm]{Corollary}

\theoremstyle{remark}
\newtheorem{examp}[defi]{Example}
\newtheorem{rem}[defi]{Remark}

\usepackage{todonotes}

\newcommand{\inv}{{-1}}

\newcommand{\id}{\mathrm{id}}

\newcommand{\ideal}{\trianglelefteq}
\newcommand{\dom}{\mathrm{Dom}}
\newcommand{\Cast}{\text{C}^*}

\newcommand{\trivial}{\{0\}}
\newcommand{\crossed}{\rtimes_\mathrm{r}}
\newcommand{\lincomb}{\mathrm{span}}

\newcommand{\ddual}{{**}}

\newcommand{\partacg}[2]{(\{{#1}_g\}_{g\in G},\allowbreak\{{#2}_g\}_{g\in G})}

\newcommand{\cC}{\mathcal{C}}

\newcommand{\hilb}{\mathcal{H}}